\documentclass{amsart}
\usepackage[margin=1.85cm]{geometry}
\usepackage[all]{xy}
\usepackage{graphicx}
\usepackage{indentfirst}
\usepackage{bm}
\usepackage{amsthm}
\usepackage{mathrsfs}
\usepackage{latexsym}
\usepackage{amsmath}
\usepackage{amssymb}
\usepackage{color}
\usepackage{hyperref}
\usepackage{microtype}
\usepackage{tikz}
\usepackage{enumerate}
\usepackage{comment}
\usepackage{braket}
\usepackage{manfnt}
\usepackage{hyperref}
\usepackage[utf8]{inputenc}
\usepackage{relsize}
\usepackage{listings}
\usepackage{textcomp}
\usepackage{spverbatim}
\usepackage[T1]{fontenc}
\usepackage{lmodern}
\allowdisplaybreaks

\begin{document}

\newtheorem{theorem}{Theorem}[section]
\newtheorem{main}{Main Theorem}
\newtheorem{proposition}[theorem]{Proposition}
\newtheorem{corollary}[theorem]{Corollary}
\newtheorem{definition}[theorem]{Definition}
\newtheorem{notation}[theorem]{Notation}
\newtheorem{convention}[theorem]{Convention}
\newtheorem{lemma}[theorem]{Lemma}
\newtheorem{example}[theorem]{Example}
\newtheorem{remark}[theorem]{Remark}
\newtheorem{question}[theorem]{Question}
\newtheorem{conjecture}[theorem]{Conjecture}
\newtheorem{fact}[theorem]{Fact}
\newtheorem*{ac}{Acknowledgements}
\newtheorem*{st}{Statements}

\newcommand{\Rep}{\mathrm{Rep}}
\newcommand{\VVec}{\mathrm{Vec}}
\newcommand{\PSU}{\mathrm{PSU}}
\newcommand{\SO}{\mathrm{SO}}
\newcommand{\SU}{\mathrm{SU}}
\newcommand{\FPdim}{\mathrm{FPdim}}

\newcommand{\mC}{\mathcal{C}}
\newcommand{\mB}{B}
\newcommand{\hc}{\hom_{\mC}}
\newcommand{\id}{{\bf 1}} 
\newcommand{\spec}{i_0} 
\newcommand{\Sp}{i_0} 
\newcommand{\SpecS}{I_{s}} 
\newcommand{\SpS}{I_{s}'} 
\newcommand{\field}{\mathbb{K}} 
\newcommand{\white}{\textcolor{red}{\bullet}} 
\newcommand{\black}{\bullet} 

\newcommand{\TP}[9]{\begin{tikzpicture}[scale=1]
	\draw (0,2)--(3,2);
	\draw (1,1)--(2,1);
	\draw (0,0)--(3,0);
	\draw (0,0) -- (0,2) -- (1,1) -- (0,0);
	\draw (3,0)--(3,2)--(2,1)--(3,0);
	\draw[->] (0,2) -- (1.5,2) node [above] {$#1$};
	\draw[->] (1,1) -- (1.5,1) node [above] {$#2$};
	\draw[->] (0,0) -- (1.5,0) node [above] {$#3$};
	\draw[->] (0,2) --++ (.5,-.5) node [right] {$#4$}; 
	\draw[->] (1,1) --++ (-.5,-.5) node [right] {$#5$}; 
	\draw[->] (0,0) -- (0,1) node [left] {$#6$}; 
	\draw[->] (3,2) --++ (-.5,-.5) node [left] {$#7$}; 
	\draw[->] (2,1) --++ (.5,-.5) node [left] {$#8$}; 
	\draw[->] (3,0) -- (3,1) node [right] {$#9$};
	\end{tikzpicture}}

\newcommand{\tetra}[6]{	
\scalebox{.66}{\raisebox{-1cm}{
\begin{tikzpicture}[scale=1.5]
\draw (0,0)-- (1,0) node [below] 	{$#4$}; \draw[->] (2,0) -- (1,0);
\draw (2,0) -- (1.5,.866) 	node [right] 	{$#6$}; \draw[->] (1,1.732) -- (1.5,.866);
\draw (1,1.732)	-- (.5,.866) 	node [left] 	{$#5$}; \draw[->] (0,0) -- (.5,.866);
\draw (0,0) 	-- (.5,.2885) 	node [above] 	{$#3$}; \draw[->] (1,0.577) -- (.5,.2885);
\draw (1,0.577) -- (1,1.1545) 	node at (.85,1.1545)	{$#1$}; \draw[->] (1,1.732)	-- (1,1.1545);
\draw (1,0.577) -- (1.5,.2885) 	node [above] 	{$#2$}; \draw[->] (2,0)	-- (1.5,.2885);
\end{tikzpicture}}}}

\newcommand{\UTP}[9]{\begin{tikzpicture}[scale=1]
	\draw (0,2)--(3,2);
	\draw (1,1)--(2,1);
	\draw (0,0)--(3,0);
	\draw (0,0) -- (0,2) -- (1,1) -- (0,0);
	\draw (3,0)--(3,2)--(2,1)--(3,0);
	\draw [fill=black] (.15,.3) circle [radius=.03];
	\draw [fill=black] (.15,1.7) circle [radius=.03];
	\draw [fill=black] (.8,1) circle [radius=.03];
	\draw [fill=black] (3-.15,.3) circle [radius=.03];
	\draw [fill=black] (3-.15,1.7) circle [radius=.03];
	\draw [fill=black] (3-.8,1) circle [radius=.03];
	\draw (0,2) -- (1.5,2) node [above] {$#1$};
	\draw (1,1) -- (1.5,1) node [above] {$#2$};
	\draw (0,0) -- (1.5,0) node [above] {$#3$};
	\draw (0,2) --++ (.5,-.5) node [right] {$#4$}; 
	\draw (1,1) --++ (-.5,-.5) node [right] {$#5$}; 
	\draw (0,0) -- (0,1) node [left] {$#6$}; 
	\draw (3,2) --++ (-.5,-.5) node [left] {$#7$}; 
	\draw (2,1) --++ (.5,-.5) node [left] {$#8$}; 
	\draw (3,0) -- (3,1) node [right] {$#9$};
	\end{tikzpicture}}
\newcommand{\UUTP}[9]{\begin{tikzpicture}[scale=1]
	\draw (0,2)--(3,2);
	\draw (1,1)--(2,1);
	\draw (0,0)--(3,0);
	\draw (0,0) -- (0,2) -- (1,1) -- (0,0);
	\draw (3,0)--(3,2)--(2,1)--(3,0);
	\draw [fill=black] (.15,.3) circle [radius=.03];
	\draw [fill=black] (3-.15,.3) circle [radius=.03];
	\draw (0,2) -- (1.5,2) node [above] {$#1$};
	\draw (1,1) -- (1.5,1) node [above] {$#2$};
	\draw (0,0) -- (1.5,0) node [above] {$#3$};
	\draw (0,2) --++ (.5,-.5) node [right] {$#4$}; 
	\draw (1,1) --++ (-.5,-.5) node [right] {$#5$}; 
	\draw (0,0) -- (0,1) node [left] {$#6$}; 
	\draw (3,2) --++ (-.5,-.5) node [left] {$#7$}; 
	\draw (2,1) --++ (.5,-.5) node [left] {$#8$}; 
	\draw (3,0) -- (3,1) node [right] {$#9$};
	\end{tikzpicture}}
\newcommand{\Fsym}[6]{
\raisebox{0cm}{
\scalebox{.66}{
\begin{tikzpicture}[scale=1.5]
\draw (0,0)--(2,0)--(1,1.732)--(0,0)--(1,0.577)--(1,1.732);
\draw (1,0.577)--(2,0);
\node at (1+.12,1) 		{$#1$};
\node at (1,.1) 			{$#2$};
\node at (.5-.1,0.866)	[above]	{$#3$};
\node at (1.5,0.2887+.2)	{$#4$};
\node at (1.5+.1,0.866) [above]	{$#5$};
\node at (.5,0.2887+.2) 	{$#6$};
\end{tikzpicture}}}}
\newcommand{\Fsymm}[6]{
	\raisebox{0cm}{
		\scalebox{.66}{
			\begin{tikzpicture}[scale=1.5]
			\draw (0,0)--(2,0)--(1,1.732)--(0,0)--(1,0.577)--(1,1.732);
			\draw (1,0.577)--(2,0);
			\node at (1+.12,1) 		{$#1$};
			\node at (1,.1) 			{$#2$};
			\node at (.5-.1,0.866)	[above]	{$#3$};
			\node at (1.5,0.2887+.2)	{$#4$};
			\node at (1.5+.1,0.866) [above]	{$#5$};
			\node at (.5,0.2887+.2) 	{$#6$};
			\draw [fill=black] (.2,.05) circle [radius=.03];
			\draw [fill=black] (1.8,.05) circle [radius=.03];
			\draw [fill=black] (1,1.832) circle [radius=.03];
			\draw [fill=black] (1,.45) circle [radius=.03];
			\end{tikzpicture}}}}

\newcommand\T[1]{%
    \def\temp10{#1}%
    \Tcontinued
}
\newcommand\Tcontinued[9]{%
\raisebox{-1cm}{
\begin{tikzpicture}[scale=1.3]
\draw[-<] (0,0) 	node at (.5,.1) {$\textcolor{orange}{#7}$} 	-- (1,0) node [below] 	{${#3}$}; \draw (1,0) 		-- (2,0);
\draw[-<] (2,0) 	node at (1.5,.1) {$\textcolor{orange}{#8^*}$} 	-- (1.5,.866) 	node [right] 	{${#5}$}; \draw (1.5,.866) 	-- (1,1.732);
\draw[-<] (1,1.732)	node at (1.15,1.2) {$\textcolor{orange}{#9^*}$} 	-- (.5,.866) 	node [left] 	{${#4}$}; \draw (.5,.866) 	-- (0,0);
\draw[-<] (0,0) 	-- (.5,.2885) 	node [above] 	{${#2}$}; \draw (.5,.2885) 	-- (1,0.577);
\draw[-<] (1,0.577) node [right] {$\textcolor{orange}{#6}$} -- (1,1.1545) 	node at (.85,1.1545)	{${\temp10}$}; \draw (1,1.1545) 	-- (1,1.732);
\draw[-<] (1,0.577) -- (1.5,.2885) 	node [above] 	{${#1}$}; \draw (1.5,.2885) 	-- (2,0);
\end{tikzpicture}}
}

\newcommand{\zhengwei}[1]{\textcolor{green}{#1 - Zhengwei}} 
\newcommand{\sebastien}[1]{\textcolor{blue}{#1 - Sebastien}}
\newcommand{\yunxiang}[1]{\textcolor{magenta}{#1 - Yunxiang}}

\newcommand\Tetra[6]{
\scalebox{.75}{\begin{tikzpicture}[scale=1.5]
\draw[->] (0,0) 	node [below] {\tiny $\bullet$} 	-- (1,0) 		node [below] 	{$#6$}; \draw (1,0) 		-- (2,0);
\draw[->] (2,0) 	node [below] {\tiny $\bullet$} 	-- (1.5,.866) 	node [right] 	{$#5$}; \draw (1.5,.866) 	-- (1,1.732);
\draw[->] (1,1.732)	node [above] {\tiny $\bullet$} 	-- (.5,.866) 	node [left] 	{$#4$}; \draw (.5,.866) 	-- (0,0);
\draw[-<] (0,0) 	-- (.5,.2885) 	node [above] 	{$#1$}; \draw (.5,.2885) 	-- (1,0.577);
\draw[->] (1,0.577) node [below] {\tiny $\bullet$} -- (1,1.1545) 	node [right]	{$#2$}; \draw (1,1.1545) 	-- (1,1.732);
\draw[->] (1,0.577) -- (1.5,.2885) 	node [above] 	{$#3$}; \draw (1.5,.2885) 	-- (2,0);
\end{tikzpicture}}}

\newcommand\TPE[1]{%
    \def\temp10{#1}%
    \TPEcontinued
}
\newcommand\TPEcontinued[9]{%
\raisebox{-1.25cm}{
\scalebox{.75}{
\begin{tikzpicture}[scale=1.5]
\draw[->] (0,0) 	node [below] {\tiny $\bullet$} 	-- (1,0) 		node [below] 	{${#6}$}; \draw (1,0) 		-- (2,0);
\draw[->] (2,0) 	node [below] {\tiny $\bullet$} 	-- (1.5,.866) 	node [right] 	{${#4}$}; \draw (1.5,.866) 	-- (1,1.732);
\draw[->] (1,1.732)	node [above] {\tiny $\bullet$} 	-- (.5,.866) 	node [left] 	{${#5}$}; \draw (.5,.866) 	-- (0,0);
\draw[->] (0,0) 	-- (.5,.2885) 	node [above] 	{${#3}$}; \draw (.5,.2885) 	-- (1,0.577);
\draw[-<] (1,0.577) node [below] {\tiny $\bullet$} -- (1,1.1545) 	node [right]	{${#2}$}; \draw (1,1.1545) 	-- (1,1.732);
\draw[-<] (1,0.577) -- (1.5,.2885) 	node [above] 	{${#1}$}; \draw (1.5,.2885) 	-- (2,0);
\end{tikzpicture}
\begin{tikzpicture}[scale=1.5]
\draw[-<] (0,0) 	node [below] {\tiny $\bullet$} 	-- (1,0) 		node [below] 	{${#9}$}; \draw (1,0) 		-- (2,0);
\draw[-<] (2,0) 	node [below] {\tiny $\bullet$} 	-- (1.5,.866) 	node [right] 	{${#8}$}; \draw (1.5,.866) 	-- (1,1.732);
\draw[-<] (1,1.732)	node [above] {\tiny $\bullet$} 	-- (.5,.866) 	node [left] 	{${#7}$}; \draw (.5,.866) 	-- (0,0);
\draw[-<] (0,0) 	-- (.5,.2885) 	node [above] 	{${#1}$}; \draw (.5,.2885) 	-- (1,0.577);
\draw[->] (1,0.577) node [below] {\tiny $\bullet$} -- (1,1.1545) 	node [right]	{${#2}$}; \draw (1,1.1545) 	-- (1,1.732);
\draw[->] (1,0.577) -- (1.5,.2885) 	node [above] 	{${#3}$}; \draw (1.5,.2885) 	-- (2,0);
\end{tikzpicture}}}
\hspace*{-.2cm} =
\sum_{s \in \temp10} d(X_s)
\raisebox{-1.25cm}{
\scalebox{.75}{
\begin{tikzpicture}[scale=1.5]
\draw[->] (0,0)  	-- (1,0) 		node [below] 	{${#1}$}; \draw (1,0) 		-- (2,0);
\draw[->] (2,0) 	-- (1.5,.866) 	node [right] 	{${#7}$}; \draw (1.5,.866) 	-- (1,1.732);
\draw[-<] (1,1.732)	node [above] {\tiny $\bullet$} 	-- (.5,.866) 	node [left] 	{${i_4}$}; \draw (.5,.866) 	-- (0,0);
\draw[-<] (0,0) 	-- (.75,.433) 	node [above] 	{${#6}$}; \draw (.75,.433) -- (1,0.577);
\draw[-<] (1,0.577) node [below] {\tiny $\bullet$} -- (1,1.1545) 	node [left]	{$s$}; \draw (1,1.1545) 	-- (1,1.732);
\draw[->] (1,0.577) -- (1.25,.433) 	node [above] 	{${#9}$}; \draw (1.25,.433) -- (2,0);
\node at (.3,.3) {\tiny $\bullet$};
\node at (2-.3,.3) {\tiny $\bullet$};
\end{tikzpicture}
\begin{tikzpicture}[scale=1.5]
\draw[->] (0,0) 	-- (1,0) 		node [below] 	{${#2}$}; \draw (1,0) 		-- (2,0);
\draw[->] (2,0) 	-- (1.5,.866) 	node [right] 	{${#8}$}; \draw (1.5,.866) 	-- (1,1.732);
\draw[->] (1,1.732)	node [above] {\tiny $\bullet$} 	-- (.5,.866) 	node [left] 	{${i_5}$}; \draw (.5,.866) 	-- (0,0);
\draw[-<] (0,0) 	-- (.75,.433) 	node [above] 	{${#4}$}; \draw (.75,.433) -- (1,0.577);
\draw[-<] (1,0.577) node [below] {\tiny $\bullet$} -- (1,1.1545) 	node [left]	{$s$}; \draw (1,1.1545) 	-- (1,1.732);
\draw[->] (1,0.577) -- (1.25,.433) 	node [above] 	{${#7}$}; \draw (1.25,.433) -- (2,0);
\node at (.3,.3) {\tiny $\bullet$};
\node at (2-.3,.3) {\tiny $\bullet$};
\end{tikzpicture}
\begin{tikzpicture}[scale=1.5]
\draw[->] (0,0) -- (1,0) 		node [below] 	{${#3}$}; \draw (1,0) 		-- (2,0);
\draw[->] (2,0) -- (1.5,.866) 	node [right] 	{${#9}$}; \draw (1.5,.866) 	-- (1,1.732);
\draw[-<] (1,1.732)	node [above] {\tiny $\bullet$} 	-- (.5,.866) 	node [left] 	{${i_6}$}; \draw (.5,.866) 	-- (0,0);
\draw[-<] (0,0) 	-- (.75,.433) 	node [above] 	{${#5}$}; \draw (.75,.433) -- (1,0.577);
\draw[-<] (1,0.577) node [below] {\tiny $\bullet$} -- (1,1.1545) 	node [left]	{$s$}; \draw (1,1.1545) 	-- (1,1.732);
\draw[->] (1,0.577) -- (1.25,.433) 	node [above] 	{${#8}$}; \draw (1.25,.433) -- (2,0);
\node at (.3,.3) {\tiny $\bullet$};
\node at (2-.3,.3) {\tiny $\bullet$};
\end{tikzpicture}
}}
}

\title[Classification of Grothendieck rings of multiplicity one]{Classification of Grothendieck rings of complex fusion categories of multiplicity one up to rank six}

\author{Zhengwei Liu}
\address{Z. Liu, Yau Mathematical Sciences Center and Department of Mathematics, Tsinghua University, and Beijing Institute of Mathematical Sciences and Applications, Beijing, China}
\email{liuzhengwei@mail.tsinghua.edu.cn}

\author{Sebastien Palcoux}
\address{S. Palcoux, Beijing Institute of Mathematical Sciences and Applications, Huairou District, Beijing, China}
\email{sebastienpalcoux@gmail.com}
\urladdr{https://sites.google.com/view/sebastienpalcoux}

\author{Yunxiang Ren}
\address{Y. Ren, Department of Physics, Harvard University, Cambridge, 02138, USA}
\email{yren@g.harvard.edu}

\maketitle

\begin{abstract} This paper classifies the Grothendieck rings of complex fusion categories of multiplicity one up to rank six. Among $72$ possible fusion rings, $25$ ones are filtered out by using categorification criteria. Each of the remaining $47$ fusion rings admits a unitary complex categorification. We found $6$ new Grothendieck rings, categorified by applying a localization approach of the Pentagon Equation.
\end{abstract}

%
%
%

\section{Introduction} 

A complex fusion category is a $\mathbb{C}$-linear semisimple rigid tensor category with finitely many simple objects and finite dimensional spaces of morphisms, such that the neutral object is simple \cite{ENO05}. The Grothendieck ring of a fusion category is a fusion ring, first introduced (and called based ring) in \cite{Lus87}. A (complex/unitary) Grothendieck ring is a fusion ring admitting a categorification into a (complex/unitary) fusion category. One of the main challenges of the subject is to decide which fusion rings are Grothendieck rings; some ones (for example mentioned in this paper) are not. In theory, a fusion ring is a Grothendieck ring if and only if its Pentagon Equation (PE) admit a solution, but in practice, this direct approach is not workable without specific strategies. In this paper, we use two strategies: several criteria (necessary conditions) in \S \ref{sec:crit} to rule out some fusion rings directly, and a localization of the PE
to categorify the remaining ones, in \S \ref{sec:tpe}. The notion of fusion ring is purely combinatorial and easy to list, and \S \ref{sec:fus} provides the list of all the fusion rings of multiplicity one up to rank six, obtained by brute-force computation \cite{FusionAtlas, sage} (see also the work of an independent group \cite{SliVer}), there are $72$ ones, and (as we will show) exactly $47$ of them are complex Grothendieck rings\footnote{Below $n_r$ and $m_r$ are the numbers of fusion rings and complex Grothendieck rings, of multiplicity one and rank $r$, see also \cite{A348305, A352506}. 
$$\begin{array}{c|cccccc} 
r&1&2&3&4&5&6 \\ \hline
n_r&1&2&4&10&16&39 \\ \hline
m_r&1&2&4&9&10&21
\end{array}$$}, all of them are unitary.
The main result is the following classification (proved in Subsection \ref{sub:proof}):
\begin{theorem} \label{thm:main}
The complex Grothendieck rings of multiplicity one up to rank six are given by the following: 
\begin{itemize}
\item known fusion categories (see the references below): 
\begin{itemize}
\item $\VVec(G)$ with finite group $G = C_n \ (n \le 6)$, $C_2^2$, $S_3$, 
\item $\Rep(G)$ with finite group $G = S_3$, $S_4$, $D_n \ (4 \le n \le 7)$, $D_9$, $Q_8$, $C_3 \rtimes C_4$, $C_3 \rtimes S_3$,
\item near-group $C_n + 0$, $n \le 5$ (also called Tambara-Yamagami $TY(C_n)$), see \cite{EG14,TY98},
\item $\SU(2)_n \ (n \le 5)$, $\PSU(2)_n \ (3 \le n \le 11)$, $\SO(3)_2$, $\SO(5)_2$, see \cite{ACMRW16,BGNPRW16,HNW14},
\item even part of a 1-supertransitive subfactor of index $3+2\sqrt{2}$, see \cite{LMP15},
\item products of two above,
\end{itemize}
\noindent where $C_n, D_n, Q_n, S_n$ are respectively the usual notations for cyclic, dihedral, quaternion, symmetric groups.
Note that $\PSU(2)_k = \SU(2)_k/C_2 \simeq$ even part of TLJ $A_{k+1}$ subfactor.
\item \textit{new} fusion categories: all with non self-adjoint objects, so none modular by Theorem \ref{thm:mod}; all weakly integral with $\FPdim<84$, so all weakly group-theoretical by \cite{ENO11}; some come from the new zesting construction \cite{zest}. In the following table, \# counts the number of Grothendieck rings:
$$\begin{array}{c|c|c|c|c}
\# & \FPdim  & \text{rank}     & \text{type}   & \text{zesting of} \\ \hline
 3 & 8 & 6 & [1,1,1,1,\sqrt{2},\sqrt{2}] & \VVec(C_2) \otimes \SU(2)_2  \\ \hline
 1 & 12 & 5 & [1,1,\sqrt{3},\sqrt{3},2] & \SO(3)_2 \\ \hline
 1 & 20 & 6 & [1, 1, 2, 2, \sqrt{5}, \sqrt{5}] & \SO(5)_2 \\ \hline
 1 & 24 & 5 & [1,1,2,3,3]  & 
\end{array}$$
\end{itemize}
\end{theorem} 
 Partial classifications exist in the literature \cite{PhDBond,GP95}. Note that \S \ref{sec:tpe} computes some categorifications for each new ring, but does not state whether the zested ones are among them. Finally, \S \ref{sec:O&Q} gives observations and questions.

%

\begin{ac} The authors would like to thank Eric C. Rowell for pointing out the new zesting construction \cite{zest}, Andrew Schopieray for pointing out the PhD thesis of Josiah E. Thornton about generalized near-group categories \cite{thor}, Ricardo Buring for his help with SageMath \cite{sage}, Joost Slingerland and Gert Vercleyen for useful discussions and AnyonWiki \cite{AnyonWiki}, Arthur Jaffe for his constant encouragement and helpful discussions, and finally the anonymous referee for a careful proofreading and relevant comments. The first author would like to thank Harvard University for his hospitality. The first author is supported by Grant 04200100122 from Tsinghua University and 2020YFA0713000 from NKPs. The second author is supported by BIMSA. The third author is supported by Grant TRT 0159, ARO Grants W911NF-19-1-0302 and W911NF-20-1-0082.
\end{ac}



\tableofcontents

\section{List of categorification criteria} \label{sec:crit}
This section lists all the categorification criteria applied in this paper, we checked them on every fusion ring (when possible), and \emph{a posteriori} it turns out that a strict subset of criteria cover all the exclusions of this specific classification (see Subsection \ref{sub:proof}), but it is still good to mention the complement subset as additional data. The first criterion holds for unitary categorification, the next one for pivotal complex categorification, the next two ones for (general) complex categorification, and the next two ones for every categorification (over every field). Finally, the last two ones are specific to the modular or quadratic case.
 
\subsection{Schur product criterion} \label{sub:schur}
Let $\mathcal{F}$ be a commutative fusion ring. Let $\Lambda=(\lambda_{i,j})$ be the table coming from the simultaneous diagonalization of its fusion matrices, with $\lambda_{i,1} = \max_j(|\lambda_{i,j}|)$. Here is the commutative Schur product criterion \cite[Corollary 8.5]{LPW20}: 
 \begin{theorem} \label{thm:schur}  
If $\mathcal{F}$ admits a unitary categorification then for all triples $(j_1,j_2,j_3)$ we have $$\sum_i \frac{\lambda_{i,j_1}\lambda_{i,j_2}\lambda_{i,j_3}}{\lambda_{i,1}} \ge 0.$$
 \end{theorem}

Note that Theorem \ref{thm:schur} is the corollary of a (less tractable) noncommutative version \cite[Proposition 8.3]{LPW20}.

\subsection{Drinfeld center criterion} \label{sub:drinfeld}
Let $\mathcal{F}$ be a \emph{commutative} fusion ring of basis $(b_i)$. Let $(X_i)$ be the corresponding fusion matrices. Let $A$ be $\sum_i X_iX_i^*$, and $(c_j)$ its eigenvalues (a commutative reformulation of the \emph{formal codegrees} in \cite{Ost15}). The fusion matrices commute over each other and are normal (because $X_i^*=X_{i^*}$), so are simultaneously diagonalizable, say as $(\lambda_{i,j})$, called the \emph{character table} of $\mathcal{F}$. Then $c_j = \sum_i |\lambda_{i,j}|^2$.

\begin{lemma} \label{lem:catdim}
If $\mathcal{F}$ admits a complex pivotal categorification $\mathcal{C}$ then there exists $j$ such that the categorical dimension of $\mathcal{C}$ equals the formal codegree $c_j$.
\end{lemma}
\begin{proof}
Let $a$ be a pivotal structure on $\mathcal{C}$. By \cite[Proposition 4.7.12]{EGNO15}, the dimension function $\dim_a$ on the objects of $\mathcal{C}$ induces a character $\chi$ on its Grothendieck ring $\mathcal{F}$, which then must be given by a column of the character table, i.e. there is $j$ such that $\chi(b_i) = \lambda_{i,j}$. Then the categorical dimension of $\mathcal{C}$ must be $c_j$.
\end{proof}

\begin{theorem}[Pivotal version of Drinfeld center criterion] \label{thm:drinfeld}
If $\mathcal{F}$ admits a complex pivotal categorification $\mathcal{C}$ then there exists $j$ such that for all $i$, $c_j/c_i$ is an algebraic integer.
\end{theorem}
\begin{proof}
The result follows by Lemma \ref{lem:catdim} and \cite[Corollary 2.14]{Ost15}.
\end{proof}

Now $\max_j(c_j) = \FPdim(\mathcal{F})$, say $c_1$. It is the categorical dimension in the pseudo-unitary case (by definition), so:

\begin{theorem}[Pseudo-unitary version of Drinfeld center criterion] \label{thm:drinfeld2}
If $\mathcal{F}$ admits a complex pseudo-unitary categorification $\mathcal{C}$ then for all $i$, $c_1/c_i$ is an algebraic integer.
\end{theorem}


If $\mathcal{F}$ is the Grothendieck ring of $\Rep(G)$ with $G$ a finite group, then the numbers $c_1/c_j$ are exactly the sizes of the conjugacy classes of $G$.

In general, if $\mathcal{F}$ admits a complex pseudo-unitary categorification $\mathcal{C}$ (so spherical), then by \cite[Theorem 2.13]{Ost15} the numbers $c_1/c_j$ are exactly the $\FPdim$ of the simple objects of the Drinfeld center which contains the trivial object in $\mathcal{C}$ under the forgetful functor. 

Note that Theorem \ref{thm:drinfeld2} admits the following conjectural stronger version extending Theorem 3.7 of Isaacs' book \cite{Isa}.
\begin{conjecture}[Isaacs criterion] \label{conj:isa}
If $\mathcal{F}$ is a complex pseudo-unitary commutative Grothendieck ring, then $\displaystyle{\frac{\lambda_{i,j}c_1}{\lambda_{i,1}c_j}}$ is an algebraic integer for all $i,j$.
\end{conjecture} 
\noindent Note that it should admit more general versions. As observed by P. Etingof \cite{eti20} (and then \cite{ENO21}), it is related to Kaplansky's 6th conjecture (generalized to fusion categories), because:

\begin{proposition} \label{prop:isafrob}
Conjecture \ref{conj:isa} implies that $\mathcal{F}$ is of \emph{Frobenius type} (i.e. $\frac{c_1}{\lambda_{i,1}}$ is an algebraic integer).
\end{proposition} 
\begin{proof}
First, $\lambda_{i,j}$ is an algebraic integer, and Conjecture \ref{conj:isa} states that  $\frac{\lambda_{i,j}c_1}{\lambda_{i,1}c_j}$ is an algebraic integer too, then
$$\sum_j \left(\frac{\lambda_{i,j}c_1}{\lambda_{i,1}c_j} \right) \overline{\lambda_{i',j}} = \frac{c_1}{\lambda_{i,1}} \sum_j \frac{1}{c_j} \lambda_{i,j} \overline{\lambda_{i',j}} = \frac{c_1}{\lambda_{i,1}} \delta_{i,i'},$$ is also an algebraic integer, which means Frobenius type. Note that the last equality (called Schur orthogonality relation) comes from \cite[Lemma 2.3]{ostrik} and the fact that a finite dimensional isometry is unitary. 
\end{proof}

Conjecture \ref{conj:isa} is true for the multiplicity one up to rank six case.
There exists another criterion (also proved by V. Ostrik \cite[Theorem 2.21]{Ost15}) using the formal codegrees. It is good to mention its commutative version here (as it is short to state): 

\begin{theorem}
If $2\sum_j 1/c_j^2 > 1+1/c_1$, then $\mathcal{F}$ admits no pseudo-unitary complex categorification.
\end{theorem} 

Note that up to rank six it applies at multiplicity at least two (so not the case concerned by this paper). It was visually compared in \cite{LPW20} with the criterion of \S \ref{sub:schur}.
 
\subsection{d-number criterion} \label{sub:Dnumb}
Let $\mathcal{F}$ be a \emph{commutative} fusion ring, and let $(c_j)$ be its formal codegrees as defined in Subsection \ref{sub:drinfeld}.

\begin{definition}[\cite{ostrik}, Definition 1.1] An algebraic integer $\alpha$ is called a \emph{d-number} if the ideal it generates in the ring of algebraic integers is invariant under the action of the absolute Galois group $Gal(\overline{\mathbb{Q}} / \mathbb{Q})$.
\end{definition}

\begin{theorem}[\cite{ostrik}, Theorem 1.2] \label{thm:Dnumb}
The formal codegrees of a complex (multi-)fusion category are d-numbers.
\end{theorem}

In particular, if $\mathcal{F}$ admits a complex categorification then the formal codegrees $(c_i)$ are d-numbers. Here is a practical way to check whether a number is a d-number.

\begin{lemma}[\cite{ostrik}, Lemma 2.7] \label{lem:Dnumb}
An algebraic integer $\alpha$ is a d-number if and only if its minimal polynomial $p(x) = x^n + a_1x^{n-1} + \dots + a_n$ (where $a_i \in \mathbb{Z}$) satisfies that $(a_n)^i$ divides $(a_i)^n$ for all $i$.
\end{lemma}

 \subsection{Extended cyclotomic criterion}

The following theorem is a slight extension of the usual cyclotomic criterion (on the simple object FPdims) of a fusion ring to all the entries of its formal character table, but in the commutative case only.  

\begin{theorem} \label{thm:cyclo}
Let $\mathcal{F}$ be a commutative fusion ring. If there is a fusion matrix such that the splitting field of its minimal polynomial is a non-abelian extension of $\mathbb{Q}$, then $\mathcal{F}$ admits no complex categorification.
\end{theorem}
\begin{proof}
First of all, a field extension of $\mathbb{Q}$ is abelian if and only if it is cyclotomic \cite{lang}. So the statement says that the eigenvalues of the fusion matrices of a commutative complex fusion category are cyclotomic integers. Next, a fusion ring is commutative if and only if its irreducible representations are one-dimensional, so the result follows by \cite[Theorem 8.51]{ENO05}.
\end{proof}

\begin{question}
Does Theorem \ref{thm:cyclo} extend to noncommutative fusion rings?
\end{question}


\subsection{Lagrange criterion}

Let mention here the generalization of Lagrange's theorem from finite groups to finite tensor categories. It will be used as a criterion of general categorification (over any field).

\begin{theorem}[\cite{EGNO15}, Theorem 7.17.6] \label{thm:lagrange}
Let $\mathcal{D}$ be a finite tensor category, and $\mathcal{C} \subset \mathcal{D}$ be a tensor
subcategory. Then the ratio $\FPdim(\mathcal{D})/ \FPdim(\mathcal{C})$ is an algebraic integer.
\end{theorem} 
 
\subsection{Zero spectrum criterion} \label{sub:zero} 
Here is a general categorification obstruction over every field (and see Remark \ref{spectrum}). It corresponds to the existence of an equation of the form $xy=0$ with $x,y \neq 0$.

\begin{theorem}[\cite{LPR1}] \label{thm:zero}
For a fusion ring $\mathcal{F}$, 
if there are indices $i_j \in I$, $1\leq j \leq 9$, such that
$N_{i_4,i_1}^{i_6}$, $N_{i_5,i_4}^{i_2}$, $N_{i_5,i_6}^{i_3}$, $N_{i_7,i_9}^{i_1}$, $N_{i_2,i_7}^{i_8}$, $N_{i_8,i_9}^{i_3}$ are non-zero, and 
\begin{align} \label{Equ: 6}
\sum_{k} N_{i_4,i_7}^{k} N_{i_5^*,i_8}^{k} N_{i_6,i_9^*}^{k}&=0;\\  \label{Equ: 9}
N_{i_2,i_1}^{i_{3}}&=1; \\ \label{Equ: 7}
\sum_{k} N_{i_5,i_4}^{k} N_{i_3,i_1^*}^{k} &=1 
~\text{or}~
\sum_{k} N_{i_2,i_4^*}^{k} N_{i_3,i_6^*}^{k} =1 
~\text{or}~
\sum_{k} N_{i_5^*,i_2}^{k} N_{i_6,i_1^*}^{k} =1, 
\\
\label{Equ: 8}
\sum_{k} N_{i_2,i_7}^{k} N_{i_3,i_9^*}^{k} &=1 
~\text{or}~
\sum_{k} N_{i_8,i_7^*}^{k} N_{i_3,i_1^*}^{k} =1 
~\text{or}~
\sum_{k} N_{i_2^*,i_8}^{k} N_{i_1,i_9^*}^{k} =1, 
\end{align}
then $\mathcal{F}$ cannot be categorified, i.e. $\mathcal{F}$ is not the Grothendieck ring of a fusion category, over any field.
\end{theorem} 


\subsection{Quadratic fusion rings} \label{sub:quadratic}

Let $\mathcal{F}$ be a fusion ring with basis $B=\{ b_1, \dots, b_r \}$. Let $G$ be the group of invertible elements $b_i$ of $B$ (i.e. $\FPdim(b_i)=1$). The fusion ring $\mathcal{F}$ is called \emph{pointed} if $B=G$, \emph{near-group} if $|B \setminus G|=1$, and more generally \emph{quadratic} if the action of $G$ on $B \setminus G$ is transitive. A fusion category with a quadratic Grothendieck ring is called (in the literature) a \emph{quadratic category} or a \emph{generalized near-group category}.

\begin{theorem} \label{thm:quadratic}
A categorification $\mathcal{C}$ of a quadratic fusion ring must admit a spherical (so pivotal) structure.
\end{theorem}
\begin{proof}
By \cite[Theorem IV.3.6.]{thor}, $\mathcal{C}$ must be $\varphi$-pseudo-unitary (i.e. pseudo-unitary up to Galois automorphism), and then spherical (so pivotal) by \cite[Proposition 2.16]{DGNO}.
\end{proof}
Theorem \ref{thm:quadratic} will be used to exclude from (general) complex categorification some quadratic fusion rings already excluded from pivotal complex categorification by Theorem \ref{thm:drinfeld}.

\subsection{Modular categorification criterion} \label{sub:mod}
Let mention the following result allowing us to see that some fusion rings admits no (complex) modular categorification.
\begin{theorem} \label{thm:mod}
Let $\mathcal{F}$ be a non-pointed weakly integral fusion ring of rank up to seven. If it is not a product, and has non self-adjoint objects, then it admits no complex modular categorification.
\end{theorem}
\begin{proof}
It is an immediate consequence of \cite[Theorem 1.2]{BGNPRW16}, because in the Ising and metaplectic categories, respectively Grothendieck equivalent to $\PSU(2)_3$ and $\SO(N)_2$ (with $N>1$ odd) the objects are self-adjoint (see \S \ref{sub:zest}).
\end{proof}

\section{The list of fusion rings of multiplicity one up to rank six}  \label{sec:fus}

This section provides the full list of $72$ fusion rings of multiplicity one up to rank $6$, together with additional data as whether it is a complex Grothendieck ring, a fusion category model (when known), and properties (as quadratic).

\subsection{Notations}
If a fusion ring does not check the Schur product (resp. Drinfeld center, d-number, extended cyclotomic, Lagrange, zero spectrum) criterion Theorem \ref{thm:schur} (resp. \ref{thm:drinfeld}, \ref{thm:Dnumb}, \ref{thm:cyclo}, \ref{thm:lagrange}, \ref{thm:zero}) then it will be qualified as  \emph{non-Schur} (resp. \emph{non-Drinfeld}, \emph{non-d-number}, \emph{non-cyclo}, \emph{non-Lagrange}, \emph{non-Czero}), and so ruled out from unitary (resp. pivotal complex, complex, complex, any, any) categorification. Otherwise it is Schur (resp. Drinfeld, d-number, cyclo, Lagrange, Czero) by default, except for the non-cyclo ones, on which the Drinfeld criteria, d-number and Lagrange were not tested. The \emph{type} of a fusion ring is the list of FPdim of the basic elements. Each commutative fusion ring is provided by its formal codegrees (with multiplicities) in decreasing order (exact form if quadratic, otherwise numerical with equation), then the reader can easily checked the Drinfeld criterion in cyclo case. The non-Czero ones are also provided by indices $(i_1, \dots, i_9)$ applying on Theorem \ref{thm:zero}. Let $\alpha_r$ denotes the number $2\cos(\pi/r)$, so $\alpha_3=1$, $\alpha_4=\sqrt{2}$, $\alpha_5=(1+\sqrt{5})/2$, $\alpha_6=\sqrt{3}$. For each rank the fusion rings are numbered with \textnumero. Finally, the fusion rings which admit a complex categorification are marked with $\star$.


\subsection{Rank 2}

\begin{itemize}
\item $\FPdim \ 2$, type $[1,1]$, one fusion ring (\textnumero 1): 
$$ \normalsize{\left[ \begin{smallmatrix}1 & 0 \\ 0 & 1 \end{smallmatrix} \right], \ 
 \left[ \begin{smallmatrix} 0 & 1 \\ 1 & 0  \end{smallmatrix} \right]} $$
\begin{itemize}
\item Formal codegrees: $[(2, 2)]$.
\item[$\star$] Properties: pointed, simple. 
\item Model: $\VVec(C_2)$.
\end{itemize}

\item $\FPdim \ \alpha_5 + 2 \simeq 3.618$, type $[1, \alpha_5]$, one fusion ring (\textnumero 2): 
$$ \normalsize{\left[ \begin{smallmatrix}1 & 0 \\ 0 & 1 \end{smallmatrix} \right], \ 
 \left[ \begin{smallmatrix} 0 & 1 \\ 1 & 1  \end{smallmatrix} \right]} $$
 \begin{itemize}
\item Formal codegrees: $[((5+\sqrt{5})/2, 1), ((5-\sqrt{5})/2, 1)]$.
\item[$\star$] Properties: near-group $C_1+1$, simple.
\item Model: $\PSU(2)_3$.
\end{itemize}
\end{itemize}

\subsection{Rank 3}
\begin{itemize}
\item $\FPdim \ 3$, type $[1,1,1]$, one fusion ring (\textnumero 1): 
$$ \normalsize{\left[ \begin{smallmatrix}1 & 0 & 0 \\ 0 & 1 & 0 \\ 0 & 0 & 1 \end{smallmatrix} \right], \ 
 \left[ \begin{smallmatrix} 0 & 1 & 0 \\ 0 & 0 & 1 \\ 1 & 0 & 0 \end{smallmatrix} \right], \ 
 \left[ \begin{smallmatrix} 0 & 0 & 1 \\ 1 & 0 & 0 \\ 0 & 1 & 0  \end{smallmatrix} \right]} $$
 \begin{itemize}
\item Formal codegrees: $[(3, 3)]$.
\item[$\star$] Properties: pointed, simple. 
\item Model: $\VVec(C_3)$.
\end{itemize}

\item $\FPdim \ 4$, type $[1,1,\alpha_4]$, one fusion ring (\textnumero 2): 
$$ \normalsize{
 \left[ \begin{smallmatrix}1 & 0 & 0 \\ 0 & 1 & 0 \\ 0 & 0 & 1 \end{smallmatrix} \right], \ 
 \left[ \begin{smallmatrix} 0 & 1 & 0 \\ 1 & 0 & 0 \\ 0 & 0 & 1 \end{smallmatrix} \right], \ 
 \left[ \begin{smallmatrix} 0 & 0 & 1 \\ 0 & 0 & 1 \\ 1 & 1 & 0  \end{smallmatrix} \right]} $$
\begin{itemize}
\item Formal codegrees: $[(4, 2), (2, 1)]$.
\item[$\star$] Properties: near-group $C_2+0$.
\item Model: $\SU(2)_2$.
\end{itemize}

\item $\FPdim \ 6$, type $[1,1,2]$, one fusion ring (\textnumero 3): 
$$ \normalsize{\left[ \begin{smallmatrix}1 & 0 & 0 \\ 0 & 1 & 0 \\ 0 & 0 & 1 \end{smallmatrix} \right], \ 
 \left[ \begin{smallmatrix} 0 & 1 & 0 \\ 1 & 0 & 0 \\ 0 & 0 & 1 \end{smallmatrix} \right], \ 
 \left[ \begin{smallmatrix} 0 & 0 & 1 \\ 0 & 0 & 1 \\ 1 & 1 & 1  \end{smallmatrix} \right]} $$
\begin{itemize}
\item Formal codegrees: $[(6, 1), (3, 1), (2, 1)]$.
\item[$\star$] Properties: near-group $C_2+1$.
\item Model: $\Rep(S_3)$, $\PSU(2)_4$.
\end{itemize}
 

\item $\FPdim \ \alpha_7^4-\alpha_7^2+1 \simeq 9.296$, type $[1,\alpha_7,\alpha_7^2-1]$, one fusion ring (\textnumero 4): 
$$ \normalsize{\left[ \begin{smallmatrix}1 & 0 & 0 \\ 0 & 1 & 0 \\ 0 & 0 & 1 \end{smallmatrix} \right], \ 
 \left[ \begin{smallmatrix} 0 & 1 & 0 \\ 1 & 0 & 1 \\ 0 & 1 & 1 \end{smallmatrix} \right], \ 
 \left[ \begin{smallmatrix} 0 & 0 & 1 \\ 0 & 1 & 1 \\ 1 & 1 & 1  \end{smallmatrix} \right]} $$
\begin{itemize}
\item Formal codegrees $\simeq [(9.296, 1), (2.863, 1), (1.841, 1)]$, roots of $x^3 - 14x^2 + 49x - 49$.
\item[$\star$] Properties: simple.
\item Model: $\PSU(2)_5$.
\end{itemize} 
\end{itemize}

\subsection{Rank 4}

\begin{itemize}
\item $\FPdim \ 4$, type $[1,1,1,1]$, two fusion rings (\textnumero 1,2): 
$$ \normalsize{\left[ \begin{smallmatrix}1 & 0 & 0 & 0 \\ 0 & 1 & 0 & 0 \\ 0 & 0 & 1 & 0 \\ 0 & 0 & 0 & 1 \end{smallmatrix} \right], \ 
 \left[ \begin{smallmatrix} 0 & 1 & 0 & 0 \\ 0 & 0 & 0 & 1 \\ 1 & 0 & 0 & 0 \\ 0 & 0 & 1 & 0 \end{smallmatrix} \right], \ 
 \left[ \begin{smallmatrix} 0 & 0 & 1 & 0 \\ 1 & 0 & 0 & 0 \\ 0 & 0 & 0 & 1 \\ 0 & 1 & 0 & 0 \end{smallmatrix} \right], \ 
 \left[ \begin{smallmatrix} 0 & 0 & 0 & 1 \\ 0 & 0 & 1 & 0 \\ 0 & 1 & 0 & 0 \\ 1 & 0 & 0 & 0  \end{smallmatrix} \right]} $$ 
 \begin{itemize}
\item Formal codegrees: $[(4, 4)]$.
\item[$\star$] Properties: pointed.
\item Model: $\VVec(C_4)$. 
\end{itemize} 


$$ \normalsize{
\left[ \begin{smallmatrix}1 & 0 & 0 & 0 \\ 0 & 1 & 0 & 0 \\ 0 & 0 & 1 & 0 \\ 0 & 0 & 0 & 1 \end{smallmatrix} \right], \ 
 \left[ \begin{smallmatrix} 0 & 1 & 0 & 0 \\ 1 & 0 & 0 & 0 \\ 0 & 0 & 0 & 1 \\ 0 & 0 & 1 & 0 \end{smallmatrix} \right], \ 
 \left[ \begin{smallmatrix} 0 & 0 & 1 & 0 \\ 0 & 0 & 0 & 1 \\ 1 & 0 & 0 & 0 \\ 0 & 1 & 0 & 0 \end{smallmatrix} \right], \ 
 \left[ \begin{smallmatrix} 0 & 0 & 0 & 1 \\ 0 & 0 & 1 & 0 \\ 0 & 1 & 0 & 0 \\ 1 & 0 & 0 & 0  \end{smallmatrix} \right]} $$ 
 \begin{itemize}
\item Formal codegrees: $[(4, 4)]$.
\item[$\star$] Properties: pointed.
\item Model: $\VVec(C_2^2)$.
\end{itemize} 

\item $\FPdim \ 6$, type $[1,1,1,\alpha_6]$, one fusion ring (\textnumero 3): 
$$ \normalsize{\left[ \begin{smallmatrix}1 & 0 & 0 & 0 \\ 0 & 1 & 0 & 0 \\ 0 & 0 & 1 & 0 \\ 0 & 0 & 0 & 1 \end{smallmatrix} \right], \ 
 \left[ \begin{smallmatrix} 0 & 1 & 0 & 0 \\ 0 & 0 & 1 & 0 \\ 1 & 0 & 0 & 0 \\ 0 & 0 & 0 & 1 \end{smallmatrix} \right], \ 
 \left[ \begin{smallmatrix} 0 & 0 & 1 & 0 \\ 1 & 0 & 0 & 0 \\ 0 & 1 & 0 & 0 \\ 0 & 0 & 0 & 1 \end{smallmatrix} \right], \ 
 \left[ \begin{smallmatrix} 0 & 0 & 0 & 1 \\ 0 & 0 & 0 & 1 \\ 0 & 0 & 0 & 1 \\ 1 & 1 & 1 & 0  \end{smallmatrix} \right]} $$ 
 \begin{itemize}
\item Formal codegrees: $[(6, 2), (3, 2)]$.
\item[$\star$] Properties: near-group $C_3+0$.
\item Model: $TY(C_3)$.
\end{itemize} 


\item $\FPdim \  2\alpha_5+4 \simeq 7.236$, type $[1,1,\alpha_5,\alpha_5]$, one fusion ring (\textnumero 4): 
$$ \normalsize{
 \left[ \begin{smallmatrix}1 & 0 & 0 & 0 \\ 0 & 1 & 0 & 0 \\ 0 & 0 & 1 & 0 \\ 0 & 0 & 0 & 1 \end{smallmatrix} \right], \ 
 \left[ \begin{smallmatrix} 0 & 1 & 0 & 0 \\ 1 & 0 & 0 & 0 \\ 0 & 0 & 0 & 1 \\ 0 & 0 & 1 & 0 \end{smallmatrix} \right], \ 
 \left[ \begin{smallmatrix} 0 & 0 & 1 & 0 \\ 0 & 0 & 0 & 1 \\ 1 & 0 & 1 & 0 \\ 0 & 1 & 0 & 1 \end{smallmatrix} \right], \ 
 \left[ \begin{smallmatrix} 0 & 0 & 0 & 1 \\ 0 & 0 & 1 & 0 \\ 0 & 1 & 0 & 1 \\ 1 & 0 & 1 & 0  \end{smallmatrix} \right]} $$ 
 \begin{itemize}
\item Formal codegrees: $[(5+\sqrt{5}, 2), (5-\sqrt{5}, 2)]$.
\item[$\star$] Properties: quadratic with $G=C_2$.
\item Model: $\SU(2)_3$, Bisch-Haagerup $BH_1$ , $\VVec(C_2) \otimes \PSU(2)_3$
\end{itemize}  

\item $\FPdim \  (13+\sqrt{13})/2 \simeq 8.302$, type $[1,1,1,(\sqrt{13}+1)/2]$, one fusion ring (\textnumero 5): 
$$ \normalsize{\left[ \begin{smallmatrix}1 & 0 & 0 & 0 \\ 0 & 1 & 0 & 0 \\ 0 & 0 & 1 & 0 \\ 0 & 0 & 0 & 1 \end{smallmatrix} \right], \ 
 \left[ \begin{smallmatrix} 0 & 1 & 0 & 0 \\ 0 & 0 & 1 & 0 \\ 1 & 0 & 0 & 0 \\ 0 & 0 & 0 & 1 \end{smallmatrix} \right], \ 
 \left[ \begin{smallmatrix} 0 & 0 & 1 & 0 \\ 1 & 0 & 0 & 0 \\ 0 & 1 & 0 & 0 \\ 0 & 0 & 0 & 1 \end{smallmatrix} \right], \ 
 \left[ \begin{smallmatrix} 0 & 0 & 0 & 1 \\ 0 & 0 & 0 & 1 \\ 0 & 0 & 0 & 1 \\ 1 & 1 & 1 & 1  \end{smallmatrix} \right]} $$ 
\begin{itemize}
\item Formal codegrees: $[((13+\sqrt{13})/2, 1), ((13-\sqrt{13})/2, 1), (3, 2)]$.
\item Properties: near-group $C_3+1$, non-Lagrange, non-d-number, non-Drinfeld.
\item Note: the only exclusion up to rank four (and multiplicity one).
\end{itemize}  

\item $\FPdim \ 10$, type $[1,1,2,2]$, one fusion ring (\textnumero 6): 
$$ \normalsize{
\left[ \begin{smallmatrix}  1 & 0 & 0 & 0 \\ 0 & 1 & 0 & 0 \\ 0 & 0 & 1 & 0 \\ 0 & 0 & 0 & 1 \end{smallmatrix} \right], \ 
 \left[ \begin{smallmatrix} 0 & 1 & 0 & 0 \\ 1 & 0 & 0 & 0 \\ 0 & 0 & 1 & 0 \\ 0 & 0 & 0 & 1 \end{smallmatrix} \right], \ 
 \left[ \begin{smallmatrix} 0 & 0 & 1 & 0 \\ 0 & 0 & 1 & 0 \\ 1 & 1 & 0 & 1 \\ 0 & 0 & 1 & 1 \end{smallmatrix} \right], \ 
 \left[ \begin{smallmatrix} 0 & 0 & 0 & 1 \\ 0 & 0 & 0 & 1 \\ 0 & 0 & 1 & 1 \\ 1 & 1 & 1 & 0  \end{smallmatrix} \right]} $$ 

\begin{itemize}
\item Formal codegrees: $[(10, 1), (5, 2), (2, 1)]$.
\item[$\star$] Properties: integral.
\item Model: $\Rep(D_5)$.
\end{itemize}  

\item $\FPdim \ 5\alpha_5^2 \simeq 13.090$, type $[1,\alpha_5,\alpha_5,\alpha_5+1]$, one fusion ring (\textnumero 7): 
$$ \normalsize{
\left[ \begin{smallmatrix}1 & 0 & 0 & 0 \\ 0 & 1 & 0 & 0 \\ 0 & 0 & 1 & 0 \\ 0 & 0 & 0 & 1 \end{smallmatrix} \right], \ 
 \left[ \begin{smallmatrix} 0 & 1 & 0 & 0 \\ 1 & 1 & 0 & 0 \\ 0 & 0 & 0 & 1 \\ 0 & 0 & 1 & 1 \end{smallmatrix} \right], \ 
 \left[ \begin{smallmatrix} 0 & 0 & 1 & 0 \\ 0 & 0 & 0 & 1 \\ 1 & 0 & 1 & 0 \\ 0 & 1 & 0 & 1 \end{smallmatrix} \right], \ 
 \left[ \begin{smallmatrix} 0 & 0 & 0 & 1 \\ 0 & 0 & 1 & 1 \\ 0 & 1 & 0 & 1 \\ 1 & 1 & 1 & 1  \end{smallmatrix} \right]} $$ 
\begin{itemize}
\item Formal codegrees: $[((15+5\sqrt{5})/2, 1), (5, 2), ((15-5\sqrt{5})/2, 1)]$.
\item[$\star$] Properties: perfect (the first non-simple one).
\item Model: $\PSU(2)_3^{\otimes 2}$.
\end{itemize}  

\item $\FPdim \  8+4\alpha_4 \simeq 13.657$, type $[1,1,\alpha_4+1,\alpha_4+1]$, two fusion rings (\textnumero 8,9): 
$$ \normalsize{
 \left[ \begin{smallmatrix} 1 & 0 & 0 & 0 \\ 0 & 1 & 0 & 0 \\ 0 & 0 & 1 & 0 \\ 0 & 0 & 0 & 1  \end{smallmatrix} \right], \ 
 \left[ \begin{smallmatrix} 0 & 1 & 0 & 0 \\ 1 & 0 & 0 & 0 \\ 0 & 0 & 0 & 1 \\ 0 & 0 & 1 & 0 \end{smallmatrix} \right], \ 
 \left[ \begin{smallmatrix} 0 & 0 & 1 & 0 \\ 0 & 0 & 0 & 1 \\ 1 & 0 & 1 & 1 \\ 0 & 1 & 1 & 1 \end{smallmatrix} \right], \ 
 \left[ \begin{smallmatrix} 0 & 0 & 0 & 1 \\ 0 & 0 & 1 & 0 \\ 0 & 1 & 1 & 1 \\ 1 & 0 & 1 & 1 \end{smallmatrix} \right]} $$ 
\begin{itemize}
\item Formal codegrees: $[(8+4\sqrt{2}, 1), (4, 2), (8-4\sqrt{2}, 1)]$.
\item[$\star$] Properties: quadratic $(C_2,1,1)$.
\item Model: $\PSU(2)_6$, even part the $1$-supertransitive subfactor of index $3+2\alpha_4$ without non-self-adjoint objects.
\end{itemize}  

$$ \normalsize{
 \left[ \begin{smallmatrix} 1 & 0 & 0 & 0 \\ 0 & 1 & 0 & 0 \\ 0 & 0 & 1 & 0 \\ 0 & 0 & 0 & 1 \end{smallmatrix} \right], \ 
 \left[ \begin{smallmatrix} 0 & 1 & 0 & 0 \\ 1 & 0 & 0 & 0 \\ 0 & 0 & 0 & 1 \\ 0 & 0 & 1 & 0 \end{smallmatrix} \right], \ 
 \left[ \begin{smallmatrix} 0 & 0 & 1 & 0 \\ 0 & 0 & 0 & 1 \\ 0 & 1 & 1 & 1 \\ 1 & 0 & 1 & 1 \end{smallmatrix} \right], \ 
 \left[ \begin{smallmatrix} 0 & 0 & 0 & 1 \\ 0 & 0 & 1 & 0 \\ 1 & 0 & 1 & 1 \\ 0 & 1 & 1 & 1 \end{smallmatrix} \right]} $$ 
\begin{itemize}
\item Formal codegrees: $[(8+4\sqrt{2}, 1), (4, 2), (8-4\sqrt{2}, 1)]$.
\item[$\star$] Properties: quadratic with $G=C_2$.
\item Model: even part of the $1$-supertransitive subfactor of index $3+2\alpha_4$ with non-self-adjoint objects \cite{LMP15}. 
\end{itemize}  

\item $\FPdim \ \alpha_9^4+2\alpha_9+3 \simeq 19.234$, type $[1,\alpha_9,\alpha_9^2-1,\alpha_9+1]$, one fusion ring (\textnumero 10): 
$$ \normalsize{
 \left[ \begin{smallmatrix} 1 & 0 & 0 & 0 \\ 0 & 1 & 0 & 0 \\ 0 & 0 & 1 & 0 \\ 0 & 0 & 0 & 1 \end{smallmatrix} \right], \ 
 \left[ \begin{smallmatrix} 0 & 1 & 0 & 0 \\ 1 & 0 & 1 & 0 \\ 0 & 1 & 0 & 1 \\ 0 & 0 & 1 & 1 \end{smallmatrix} \right], \ 
 \left[ \begin{smallmatrix} 0 & 0 & 1 & 0 \\ 0 & 1 & 0 & 1 \\ 1 & 0 & 1 & 1 \\ 0 & 1 & 1 & 1 \end{smallmatrix} \right], \ 
 \left[ \begin{smallmatrix} 0 & 0 & 0 & 1 \\ 0 & 0 & 1 & 1 \\ 0 & 1 & 1 & 1 \\ 1 & 1 & 1 & 1 \end{smallmatrix} \right]} $$ 
\begin{itemize}
\item Formal codegrees $\simeq [(19.234, 1), (5.445, 1), (3, 1), (2.31996, 1)]$, roots of $(x^3 - 27x^2 + 162x - 243)(x - 3)$.
\item[$\star$] Properties: simple. 
\item Model: $\PSU(2)_7$.
\end{itemize}

\end{itemize}

\subsection{Rank 5}
\begin{itemize}

\item $\FPdim \ 5$, type $[1,1,1,1,1]$, one fusion ring (\textnumero 1): 
$$ \normalsize{\left[ \begin{smallmatrix}1 & 0 & 0 & 0 & 0 \\ 0 & 1 & 0 & 0 & 0 \\ 0 & 0 & 1 & 0 & 0 \\ 0 & 0 & 0 & 1 & 0 \\ 0 & 0 & 0 & 0 & 1 \end{smallmatrix} \right], \ 
 \left[ \begin{smallmatrix} 0 & 1 & 0 & 0 & 0 \\ 0 & 0 & 0 & 0 & 1 \\ 1 & 0 & 0 & 0 & 0 \\ 0 & 0 & 1 & 0 & 0 \\ 0 & 0 & 0 & 1 & 0 \end{smallmatrix} \right], \ 
 \left[ \begin{smallmatrix} 0 & 0 & 1 & 0 & 0 \\ 1 & 0 & 0 & 0 & 0 \\ 0 & 0 & 0 & 1 & 0 \\ 0 & 0 & 0 & 0 & 1 \\ 0 & 1 & 0 & 0 & 0 \end{smallmatrix} \right], \ 
 \left[ \begin{smallmatrix} 0 & 0 & 0 & 1 & 0 \\ 0 & 0 & 1 & 0 & 0 \\ 0 & 0 & 0 & 0 & 1 \\ 0 & 1 & 0 & 0 & 0 \\ 1 & 0 & 0 & 0 & 0 \end{smallmatrix} \right], \ 
 \left[ \begin{smallmatrix} 0 & 0 & 0 & 0 & 1 \\ 0 & 0 & 0 & 1 & 0 \\ 0 & 1 & 0 & 0 & 0 \\ 1 & 0 & 0 & 0 & 0 \\ 0 & 0 & 1 & 0 & 0  \end{smallmatrix} \right]} $$ 
\begin{itemize}
\item Formal codegrees: $[(5, 5)]$.
\item[$\star$] Properties: simple.
\item Model: $\VVec(C_5)$.
\end{itemize}

\item $\FPdim \ 8$, type $[1,1,1,1,2]$, two fusion rings (\textnumero 2,3): 
$$ \normalsize{\left[ \begin{smallmatrix}1 & 0 & 0 & 0 & 0 \\ 0 & 1 & 0 & 0 & 0 \\ 0 & 0 & 1 & 0 & 0 \\ 0 & 0 & 0 & 1 & 0 \\ 0 & 0 & 0 & 0 & 1 \end{smallmatrix} \right], \ 
 \left[ \begin{smallmatrix} 0 & 1 & 0 & 0 & 0 \\ 0 & 0 & 0 & 1 & 0 \\ 1 & 0 & 0 & 0 & 0 \\ 0 & 0 & 1 & 0 & 0 \\ 0 & 0 & 0 & 0 & 1 \end{smallmatrix} \right], \ 
 \left[ \begin{smallmatrix} 0 & 0 & 1 & 0 & 0 \\ 1 & 0 & 0 & 0 & 0 \\ 0 & 0 & 0 & 1 & 0 \\ 0 & 1 & 0 & 0 & 0 \\ 0 & 0 & 0 & 0 & 1 \end{smallmatrix} \right], \ 
 \left[ \begin{smallmatrix} 0 & 0 & 0 & 1 & 0 \\ 0 & 0 & 1 & 0 & 0 \\ 0 & 1 & 0 & 0 & 0 \\ 1 & 0 & 0 & 0 & 0 \\ 0 & 0 & 0 & 0 & 1 \end{smallmatrix} \right], \ 
 \left[ \begin{smallmatrix} 0 & 0 & 0 & 0 & 1 \\ 0 & 0 & 0 & 0 & 1 \\ 0 & 0 & 0 & 0 & 1 \\ 0 & 0 & 0 & 0 & 1 \\ 1 & 1 & 1 & 1 & 0  \end{smallmatrix} \right]} $$ 
 \begin{itemize}
\item Formal codegrees: $[(8, 2), (4, 3)]$.
\item[$\star$] Properties: near-group $C_4+0$.
\item Model: $TY(C_4)$.
\end{itemize}

$$ \normalsize{\left[ \begin{smallmatrix}1 & 0 & 0 & 0 & 0 \\ 0 & 1 & 0 & 0 & 0 \\ 0 & 0 & 1 & 0 & 0 \\ 0 & 0 & 0 & 1 & 0 \\ 0 & 0 & 0 & 0 & 1 \end{smallmatrix} \right], \ 
 \left[ \begin{smallmatrix} 0 & 1 & 0 & 0 & 0 \\ 1 & 0 & 0 & 0 & 0 \\ 0 & 0 & 0 & 1 & 0 \\ 0 & 0 & 1 & 0 & 0 \\ 0 & 0 & 0 & 0 & 1 \end{smallmatrix} \right], \ 
 \left[ \begin{smallmatrix} 0 & 0 & 1 & 0 & 0 \\ 0 & 0 & 0 & 1 & 0 \\ 1 & 0 & 0 & 0 & 0 \\ 0 & 1 & 0 & 0 & 0 \\ 0 & 0 & 0 & 0 & 1 \end{smallmatrix} \right], \ 
 \left[ \begin{smallmatrix} 0 & 0 & 0 & 1 & 0 \\ 0 & 0 & 1 & 0 & 0 \\ 0 & 1 & 0 & 0 & 0 \\ 1 & 0 & 0 & 0 & 0 \\ 0 & 0 & 0 & 0 & 1 \end{smallmatrix} \right], \ 
 \left[ \begin{smallmatrix} 0 & 0 & 0 & 0 & 1 \\ 0 & 0 & 0 & 0 & 1 \\ 0 & 0 & 0 & 0 & 1 \\ 0 & 0 & 0 & 0 & 1 \\ 1 & 1 & 1 & 1 & 0  \end{smallmatrix} \right]} $$ 
 \begin{itemize}
\item Formal codegrees: $[(8, 2), (4, 3)]$.
\item[$\star$] Properties: near-group $C_2^2 + 0$.
\item Model: $\Rep(D_4)$, $\Rep(Q_8)$.
\end{itemize}

\item $\FPdim \ (\sqrt{17}+17)/2 \simeq 10.562$, type $[1,1,1,1,(\sqrt{17}+1)/2]$, two fusion rings (\textnumero 4,5): 
$$ \normalsize{
\left[ \begin{smallmatrix}1 & 0 & 0 & 0 & 0 \\ 0 & 1 & 0 & 0 & 0 \\ 0 & 0 & 1 & 0 & 0 \\ 0 & 0 & 0 & 1 & 0 \\ 0 & 0 & 0 & 0 & 1 \end{smallmatrix} \right], \ 
 \left[ \begin{smallmatrix} 0 & 1 & 0 & 0 & 0 \\ 0 & 0 & 0 & 1 & 0 \\ 1 & 0 & 0 & 0 & 0 \\ 0 & 0 & 1 & 0 & 0 \\ 0 & 0 & 0 & 0 & 1 \end{smallmatrix} \right], \ 
 \left[ \begin{smallmatrix} 0 & 0 & 1 & 0 & 0 \\ 1 & 0 & 0 & 0 & 0 \\ 0 & 0 & 0 & 1 & 0 \\ 0 & 1 & 0 & 0 & 0 \\ 0 & 0 & 0 & 0 & 1 \end{smallmatrix} \right], \ 
 \left[ \begin{smallmatrix} 0 & 0 & 0 & 1 & 0 \\ 0 & 0 & 1 & 0 & 0 \\ 0 & 1 & 0 & 0 & 0 \\ 1 & 0 & 0 & 0 & 0 \\ 0 & 0 & 0 & 0 & 1 \end{smallmatrix} \right], \ 
 \left[ \begin{smallmatrix} 0 & 0 & 0 & 0 & 1 \\ 0 & 0 & 0 & 0 & 1 \\ 0 & 0 & 0 & 0 & 1 \\ 0 & 0 & 0 & 0 & 1 \\ 1 & 1 & 1 & 1 & 1  \end{smallmatrix} \right]} $$ 
 \begin{itemize}
\item Formal codegrees: $[((17+\sqrt{17})/2, 1), ((17-\sqrt{17})/2, 1), (4, 3)]$.
\item Properties: near-group $C_4+1$, non-Lagrange, non-d-number, non-Drinfeld.
\end{itemize}

$$ \normalsize{\left[ \begin{smallmatrix}1 & 0 & 0 & 0 & 0 \\ 0 & 1 & 0 & 0 & 0 \\ 0 & 0 & 1 & 0 & 0 \\ 0 & 0 & 0 & 1 & 0 \\ 0 & 0 & 0 & 0 & 1 \end{smallmatrix} \right], \ 
 \left[ \begin{smallmatrix} 0 & 1 & 0 & 0 & 0 \\ 1 & 0 & 0 & 0 & 0 \\ 0 & 0 & 0 & 1 & 0 \\ 0 & 0 & 1 & 0 & 0 \\ 0 & 0 & 0 & 0 & 1 \end{smallmatrix} \right], \ 
 \left[ \begin{smallmatrix} 0 & 0 & 1 & 0 & 0 \\ 0 & 0 & 0 & 1 & 0 \\ 1 & 0 & 0 & 0 & 0 \\ 0 & 1 & 0 & 0 & 0 \\ 0 & 0 & 0 & 0 & 1 \end{smallmatrix} \right], \ 
 \left[ \begin{smallmatrix} 0 & 0 & 0 & 1 & 0 \\ 0 & 0 & 1 & 0 & 0 \\ 0 & 1 & 0 & 0 & 0 \\ 1 & 0 & 0 & 0 & 0 \\ 0 & 0 & 0 & 0 & 1 \end{smallmatrix} \right], \ 
 \left[ \begin{smallmatrix} 0 & 0 & 0 & 0 & 1 \\ 0 & 0 & 0 & 0 & 1 \\ 0 & 0 & 0 & 0 & 1 \\ 0 & 0 & 0 & 0 & 1 \\ 1 & 1 & 1 & 1 & 1  \end{smallmatrix} \right]} $$  
\begin{itemize}
\item Formal codegrees: $[((17+\sqrt{17})/2, 1), ((17-\sqrt{17})/2, 1), (4, 3)]$.
\item Properties: near-group $C_2^2 + 1$, non-Lagrange, non-d-number, non-Drinfeld. 
\end{itemize}

\item $\FPdim \ 12$, type $[1,1,\alpha_6,\alpha_6,2]$, two fusion rings (\textnumero 6,7): 

$$ \normalsize{\left[ \begin{smallmatrix}1 & 0 & 0 & 0 & 0 \\ 0 & 1 & 0 & 0 & 0 \\ 0 & 0 & 1 & 0 & 0 \\ 0 & 0 & 0 & 1 & 0 \\ 0 & 0 & 0 & 0 & 1 \end{smallmatrix} \right], \ 
 \left[ \begin{smallmatrix} 0 & 1 & 0 & 0 & 0 \\ 1 & 0 & 0 & 0 & 0 \\ 0 & 0 & 0 & 1 & 0 \\ 0 & 0 & 1 & 0 & 0 \\ 0 & 0 & 0 & 0 & 1 \end{smallmatrix} \right], \ 
 \left[ \begin{smallmatrix} 0 & 0 & 1 & 0 & 0 \\ 0 & 0 & 0 & 1 & 0 \\ 1 & 0 & 0 & 0 & 1 \\ 0 & 1 & 0 & 0 & 1 \\ 0 & 0 & 1 & 1 & 0 \end{smallmatrix} \right], \ 
 \left[ \begin{smallmatrix} 0 & 0 & 0 & 1 & 0 \\ 0 & 0 & 1 & 0 & 0 \\ 0 & 1 & 0 & 0 & 1 \\ 1 & 0 & 0 & 0 & 1 \\ 0 & 0 & 1 & 1 & 0 \end{smallmatrix} \right], \ 
 \left[ \begin{smallmatrix} 0 & 0 & 0 & 0 & 1 \\ 0 & 0 & 0 & 0 & 1 \\ 0 & 0 & 1 & 1 & 0 \\ 0 & 0 & 1 & 1 & 0 \\ 1 & 1 & 0 & 0 & 1  \end{smallmatrix} \right]} $$ 
 \begin{itemize}
\item Formal codegrees: $[(12, 2), (4, 2), (3, 1)]$.
\item[$\star$] Properties: Extension of $\VVec(C_2)$.
\item Model: $\SU(2)_4$, $\SO(3)_2$.
\end{itemize}

$$ \normalsize{\left[ \begin{smallmatrix}1 & 0 & 0 & 0 & 0 \\ 0 & 1 & 0 & 0 & 0 \\ 0 & 0 & 1 & 0 & 0 \\ 0 & 0 & 0 & 1 & 0 \\ 0 & 0 & 0 & 0 & 1 \end{smallmatrix} \right], \ 
 \left[ \begin{smallmatrix} 0 & 1 & 0 & 0 & 0 \\ 1 & 0 & 0 & 0 & 0 \\ 0 & 0 & 0 & 1 & 0 \\ 0 & 0 & 1 & 0 & 0 \\ 0 & 0 & 0 & 0 & 1 \end{smallmatrix} \right], \ 
 \left[ \begin{smallmatrix} 0 & 0 & 1 & 0 & 0 \\ 0 & 0 & 0 & 1 & 0 \\ 0 & 1 & 0 & 0 & 1 \\ 1 & 0 & 0 & 0 & 1 \\ 0 & 0 & 1 & 1 & 0 \end{smallmatrix} \right], \ 
 \left[ \begin{smallmatrix} 0 & 0 & 0 & 1 & 0 \\ 0 & 0 & 1 & 0 & 0 \\ 1 & 0 & 0 & 0 & 1 \\ 0 & 1 & 0 & 0 & 1 \\ 0 & 0 & 1 & 1 & 0 \end{smallmatrix} \right], \ 
 \left[ \begin{smallmatrix} 0 & 0 & 0 & 0 & 1 \\ 0 & 0 & 0 & 0 & 1 \\ 0 & 0 & 1 & 1 & 0 \\ 0 & 0 & 1 & 1 & 0 \\ 1 & 1 & 0 & 0 & 1  \end{smallmatrix} \right]} $$ 
 \begin{itemize}
\item Formal codegrees: $[(12, 2), (4, 2), (3, 1)]$.
\item[$\star$] Properties: unitarily categorified (see \S \ref{r5dim12}), non-modular, weakly group-theoretical,
\item Model: zesting of $\SO(3)_2$, see \S \ref{sub:zest}.
\end{itemize}

\item $\FPdim \ 14$, type $[1,1,2,2,2]$, one fusion ring (\textnumero 8): 
$$ \normalsize{\left[ \begin{smallmatrix}1 & 0 & 0 & 0 & 0 \\ 0 & 1 & 0 & 0 & 0 \\ 0 & 0 & 1 & 0 & 0 \\ 0 & 0 & 0 & 1 & 0 \\ 0 & 0 & 0 & 0 & 1 \end{smallmatrix} \right], \ 
 \left[ \begin{smallmatrix} 0 & 1 & 0 & 0 & 0 \\ 1 & 0 & 0 & 0 & 0 \\ 0 & 0 & 1 & 0 & 0 \\ 0 & 0 & 0 & 1 & 0 \\ 0 & 0 & 0 & 0 & 1 \end{smallmatrix} \right], \ 
 \left[ \begin{smallmatrix} 0 & 0 & 1 & 0 & 0 \\ 0 & 0 & 1 & 0 & 0 \\ 1 & 1 & 0 & 1 & 0 \\ 0 & 0 & 1 & 0 & 1 \\ 0 & 0 & 0 & 1 & 1 \end{smallmatrix} \right], \ 
 \left[ \begin{smallmatrix} 0 & 0 & 0 & 1 & 0 \\ 0 & 0 & 0 & 1 & 0 \\ 0 & 0 & 1 & 0 & 1 \\ 1 & 1 & 0 & 0 & 1 \\ 0 & 0 & 1 & 1 & 0 \end{smallmatrix} \right], \ 
 \left[ \begin{smallmatrix} 0 & 0 & 0 & 0 & 1 \\ 0 & 0 & 0 & 0 & 1 \\ 0 & 0 & 0 & 1 & 1 \\ 0 & 0 & 1 & 1 & 0 \\ 1 & 1 & 1 & 0 & 0  \end{smallmatrix} \right]} $$ 
 \begin{itemize}
\item Formal codegrees: $[(14, 1), (7, 3), (2, 1)]$.
\item[$\star$] Properties: integral.
\item Model: $\Rep(D_7)$.
\end{itemize}

\item $\FPdim \ \sqrt{13}+13 \simeq 16.606$, type $[1,1,2,(\sqrt{13}+1)/2,(\sqrt{13}+1)/2]$, one fusion ring (\textnumero 9): 
$$ \normalsize{\left[ \begin{smallmatrix}1 & 0 & 0 & 0 & 0 \\ 0 & 1 & 0 & 0 & 0 \\ 0 & 0 & 1 & 0 & 0 \\ 0 & 0 & 0 & 1 & 0 \\ 0 & 0 & 0 & 0 & 1 \end{smallmatrix} \right], \ 
 \left[ \begin{smallmatrix} 0 & 1 & 0 & 0 & 0 \\ 1 & 0 & 0 & 0 & 0 \\ 0 & 0 & 1 & 0 & 0 \\ 0 & 0 & 0 & 0 & 1 \\ 0 & 0 & 0 & 1 & 0 \end{smallmatrix} \right], \ 
 \left[ \begin{smallmatrix} 0 & 0 & 1 & 0 & 0 \\ 0 & 0 & 1 & 0 & 0 \\ 1 & 1 & 1 & 0 & 0 \\ 0 & 0 & 0 & 1 & 1 \\ 0 & 0 & 0 & 1 & 1 \end{smallmatrix} \right], \ 
 \left[ \begin{smallmatrix} 0 & 0 & 0 & 1 & 0 \\ 0 & 0 & 0 & 0 & 1 \\ 0 & 0 & 0 & 1 & 1 \\ 1 & 0 & 1 & 1 & 0 \\ 0 & 1 & 1 & 0 & 1 \end{smallmatrix} \right], \ 
 \left[ \begin{smallmatrix} 0 & 0 & 0 & 0 & 1 \\ 0 & 0 & 0 & 1 & 0 \\ 0 & 0 & 0 & 1 & 1 \\ 0 & 1 & 1 & 0 & 1 \\ 1 & 0 & 1 & 1 & 0  \end{smallmatrix} \right]} $$ 
 \begin{itemize}
\item Formal codegrees: $[(13+\sqrt{13}, 1), (13-\sqrt{13}, 1), (5+\sqrt{5} + 5, 1), (3, 1), (5-\sqrt{5}, 1)]$.
\item Properties: non-Schur, non-d-number, non-Drinfeld, non-Lagrange, non-Czero $(3, 3, 2, 2, 4, 3, 3, 3, 3)$.
\end{itemize}

\item $\FPdim \ 24$, type $[1, 1, 2, 3, 3]$, two fusion rings (\textnumero 10,11): 
$$ \normalsize{\left[ \begin{smallmatrix}1 & 0 & 0 & 0 & 0 \\ 0 & 1 & 0 & 0 & 0 \\ 0 & 0 & 1 & 0 & 0 \\ 0 & 0 & 0 & 1 & 0 \\ 0 & 0 & 0 & 0 & 1 \end{smallmatrix} \right], \ 
 \left[ \begin{smallmatrix} 0 & 1 & 0 & 0 & 0 \\ 1 & 0 & 0 & 0 & 0 \\ 0 & 0 & 1 & 0 & 0 \\ 0 & 0 & 0 & 0 & 1 \\ 0 & 0 & 0 & 1 & 0 \end{smallmatrix} \right], \ 
 \left[ \begin{smallmatrix} 0 & 0 & 1 & 0 & 0 \\ 0 & 0 & 1 & 0 & 0 \\ 1 & 1 & 1 & 0 & 0 \\ 0 & 0 & 0 & 1 & 1 \\ 0 & 0 & 0 & 1 & 1 \end{smallmatrix} \right], \ 
 \left[ \begin{smallmatrix} 0 & 0 & 0 & 1 & 0 \\ 0 & 0 & 0 & 0 & 1 \\ 0 & 0 & 0 & 1 & 1 \\ 1 & 0 & 1 & 1 & 1 \\ 0 & 1 & 1 & 1 & 1 \end{smallmatrix} \right], \ 
 \left[ \begin{smallmatrix} 0 & 0 & 0 & 0 & 1 \\ 0 & 0 & 0 & 1 & 0 \\ 0 & 0 & 0 & 1 & 1 \\ 0 & 1 & 1 & 1 & 1 \\ 1 & 0 & 1 & 1 & 1  \end{smallmatrix} \right]} $$ 
\begin{itemize}
\item Formal codegrees: $[(24, 1), (8, 1), (4, 2), (3, 1)]$.
\item[$\star$] Properties: integral.
\item Model: $\Rep(S_4)$.
\end{itemize}

$$ \normalsize{\left[ \begin{smallmatrix}1 & 0 & 0 & 0 & 0 \\ 0 & 1 & 0 & 0 & 0 \\ 0 & 0 & 1 & 0 & 0 \\ 0 & 0 & 0 & 1 & 0 \\ 0 & 0 & 0 & 0 & 1 \end{smallmatrix} \right], \ 
 \left[ \begin{smallmatrix} 0 & 1 & 0 & 0 & 0 \\ 1 & 0 & 0 & 0 & 0 \\ 0 & 0 & 1 & 0 & 0 \\ 0 & 0 & 0 & 0 & 1 \\ 0 & 0 & 0 & 1 & 0 \end{smallmatrix} \right], \ 
 \left[ \begin{smallmatrix} 0 & 0 & 1 & 0 & 0 \\ 0 & 0 & 1 & 0 & 0 \\ 1 & 1 & 1 & 0 & 0 \\ 0 & 0 & 0 & 1 & 1 \\ 0 & 0 & 0 & 1 & 1 \end{smallmatrix} \right], \ 
 \left[ \begin{smallmatrix} 0 & 0 & 0 & 1 & 0 \\ 0 & 0 & 0 & 0 & 1 \\ 0 & 0 & 0 & 1 & 1 \\ 0 & 1 & 1 & 1 & 1 \\ 1 & 0 & 1 & 1 & 1 \end{smallmatrix} \right], \ 
 \left[ \begin{smallmatrix} 0 & 0 & 0 & 0 & 1 \\ 0 & 0 & 0 & 1 & 0 \\ 0 & 0 & 0 & 1 & 1 \\ 1 & 0 & 1 & 1 & 1 \\ 0 & 1 & 1 & 1 & 1  \end{smallmatrix} \right]} $$ 
 \begin{itemize}
\item Formal codegrees: $[(24, 1), (8, 1), (4, 2), (3, 1)]$.
\item[$\star$] Properties: extension of $\Rep(S_3)$, unitarily categorified (see \S \ref{sub:r5dim24}), non-modular, weakly group-theoretical.
\item Model: unknown, but see Question \ref{M+}. 
\end{itemize}


\item $\FPdim \ 10\alpha_5^2 \simeq 26.180$, type $[1, 1, \alpha_5+1, \alpha_5+1, 2\alpha_5]$, one fusion ring (\textnumero 12): 
$$ \normalsize{
 \left[ \begin{smallmatrix}1 & 0 & 0 & 0 & 0 \\ 0 & 1 & 0 & 0 & 0 \\ 0 & 0 & 1 & 0 & 0 \\ 0 & 0 & 0 & 1 & 0 \\ 0 & 0 & 0 & 0 & 1 \end{smallmatrix} \right], \ 
 \left[ \begin{smallmatrix} 0 & 1 & 0 & 0 & 0 \\ 1 & 0 & 0 & 0 & 0 \\ 0 & 0 & 0 & 1 & 0 \\ 0 & 0 & 1 & 0 & 0 \\ 0 & 0 & 0 & 0 & 1 \end{smallmatrix} \right], \ 
 \left[ \begin{smallmatrix} 0 & 0 & 1 & 0 & 0 \\ 0 & 0 & 0 & 1 & 0 \\ 1 & 0 & 1 & 0 & 1 \\ 0 & 1 & 0 & 1 & 1 \\ 0 & 0 & 1 & 1 & 1 \end{smallmatrix} \right], \ 
 \left[ \begin{smallmatrix} 0 & 0 & 0 & 1 & 0 \\ 0 & 0 & 1 & 0 & 0 \\ 0 & 1 & 0 & 1 & 1 \\ 1 & 0 & 1 & 0 & 1 \\ 0 & 0 & 1 & 1 & 1 \end{smallmatrix} \right], \ 
 \left[ \begin{smallmatrix} 0 & 0 & 0 & 0 & 1 \\ 0 & 0 & 0 & 0 & 1 \\ 0 & 0 & 1 & 1 & 1 \\ 0 & 0 & 1 & 1 & 1 \\ 1 & 1 & 1 & 1 & 1  \end{smallmatrix} \right]} $$ 
 \begin{itemize}
\item Formal codegrees: $[(15+5\sqrt{5}, 1), (5+\sqrt{5}, 1), (5, 1), (15-5\sqrt{5}, 1), (5-\sqrt{5}, 1)]$.
\item[$\star$] Properties: extension of $\VVec(C_2)$.
\item Model: $\PSU(2)_8$.
\end{itemize}

\item $\FPdim \ 16+10\alpha_4 \simeq 30.142$, type $[1, 1+\alpha_4, 1+\alpha_4, 1+\alpha_4,2+\alpha_4]$, one fusion ring (\textnumero 13): 
$$ \normalsize{\left[ \begin{smallmatrix}1 & 0 & 0 & 0 & 0 \\ 0 & 1 & 0 & 0 & 0 \\ 0 & 0 & 1 & 0 & 0 \\ 0 & 0 & 0 & 1 & 0 \\ 0 & 0 & 0 & 0 & 1 \end{smallmatrix} \right], \ 
 \left[ \begin{smallmatrix} 0 & 1 & 0 & 0 & 0 \\ 1 & 1 & 1 & 0 & 0 \\ 0 & 1 & 0 & 0 & 1 \\ 0 & 0 & 0 & 1 & 1 \\ 0 & 0 & 1 & 1 & 1 \end{smallmatrix} \right], \ 
 \left[ \begin{smallmatrix} 0 & 0 & 1 & 0 & 0 \\ 0 & 1 & 0 & 0 & 1 \\ 1 & 0 & 1 & 1 & 0 \\ 0 & 0 & 1 & 0 & 1 \\ 0 & 1 & 0 & 1 & 1 \end{smallmatrix} \right], \ 
 \left[ \begin{smallmatrix} 0 & 0 & 0 & 1 & 0 \\ 0 & 0 & 0 & 1 & 1 \\ 0 & 0 & 1 & 0 & 1 \\ 1 & 1 & 0 & 1 & 0 \\ 0 & 1 & 1 & 0 & 1 \end{smallmatrix} \right], \ 
 \left[ \begin{smallmatrix} 0 & 0 & 0 & 0 & 1 \\ 0 & 0 & 1 & 1 & 1 \\ 0 & 1 & 0 & 1 & 1 \\ 0 & 1 & 1 & 0 & 1 \\ 1 & 1 & 1 & 1 & 1  \end{smallmatrix} \right]} $$ 
\begin{itemize}
\item Formal codegrees: $[(16+10\sqrt{2}, 1), (7, 3), (16-10\sqrt{2}, 1)]$.
\item Properties: simple, non-Schur, non-d-number, non-Drinfeld, non-Czero $(1, 1, 2, 1, 1, 1, 2, 4, 1)$.
\end{itemize}

\item $\FPdim \ \simeq 31.092$, type $\simeq [1,1,2.903, 3.214, 3.214]$, two fusion rings (\textnumero 14,15): 
$$ \normalsize{\left[ \begin{smallmatrix}1 & 0 & 0 & 0 & 0 \\ 0 & 1 & 0 & 0 & 0 \\ 0 & 0 & 1 & 0 & 0 \\ 0 & 0 & 0 & 1 & 0 \\ 0 & 0 & 0 & 0 & 1 \end{smallmatrix} \right], \ 
 \left[ \begin{smallmatrix} 0 & 1 & 0 & 0 & 0 \\ 1 & 0 & 0 & 0 & 0 \\ 0 & 0 & 1 & 0 & 0 \\ 0 & 0 & 0 & 0 & 1 \\ 0 & 0 & 0 & 1 & 0 \end{smallmatrix} \right], \ 
 \left[ \begin{smallmatrix} 0 & 0 & 1 & 0 & 0 \\ 0 & 0 & 1 & 0 & 0 \\ 1 & 1 & 0 & 1 & 1 \\ 0 & 0 & 1 & 1 & 1 \\ 0 & 0 & 1 & 1 & 1 \end{smallmatrix} \right], \ 
 \left[ \begin{smallmatrix} 0 & 0 & 0 & 1 & 0 \\ 0 & 0 & 0 & 0 & 1 \\ 0 & 0 & 1 & 1 & 1 \\ 0 & 1 & 1 & 1 & 1 \\ 1 & 0 & 1 & 1 & 1 \end{smallmatrix} \right], \ 
 \left[ \begin{smallmatrix} 0 & 0 & 0 & 0 & 1 \\ 0 & 0 & 0 & 1 & 0 \\ 0 & 0 & 1 & 1 & 1 \\ 1 & 0 & 1 & 1 & 1 \\ 0 & 1 & 1 & 1 & 1  \end{smallmatrix} \right]} $$ 
 \begin{itemize}
\item Formal codegrees $\simeq [(31.092, 1), (5.346, 1), (4, 2), (3.561, 1)]$, roots of $x^5 - 48x^4 + 632x^3 - 3600x^2 + 9472x - 9472$.
\item Properties: non-cyclo.
\end{itemize}

$$ \normalsize{\left[ \begin{smallmatrix}1 & 0 & 0 & 0 & 0 \\ 0 & 1 & 0 & 0 & 0 \\ 0 & 0 & 1 & 0 & 0 \\ 0 & 0 & 0 & 1 & 0 \\ 0 & 0 & 0 & 0 & 1 \end{smallmatrix} \right], \ 
 \left[ \begin{smallmatrix} 0 & 1 & 0 & 0 & 0 \\ 1 & 0 & 0 & 0 & 0 \\ 0 & 0 & 1 & 0 & 0 \\ 0 & 0 & 0 & 0 & 1 \\ 0 & 0 & 0 & 1 & 0 \end{smallmatrix} \right], \ 
 \left[ \begin{smallmatrix} 0 & 0 & 1 & 0 & 0 \\ 0 & 0 & 1 & 0 & 0 \\ 1 & 1 & 0 & 1 & 1 \\ 0 & 0 & 1 & 1 & 1 \\ 0 & 0 & 1 & 1 & 1 \end{smallmatrix} \right], \ 
 \left[ \begin{smallmatrix} 0 & 0 & 0 & 1 & 0 \\ 0 & 0 & 0 & 0 & 1 \\ 0 & 0 & 1 & 1 & 1 \\ 1 & 0 & 1 & 1 & 1 \\ 0 & 1 & 1 & 1 & 1 \end{smallmatrix} \right], \ 
 \left[ \begin{smallmatrix} 0 & 0 & 0 & 0 & 1 \\ 0 & 0 & 0 & 1 & 0 \\ 0 & 0 & 1 & 1 & 1 \\ 0 & 1 & 1 & 1 & 1 \\ 1 & 0 & 1 & 1 & 1  \end{smallmatrix} \right]} $$ 
\begin{itemize}
\item Formal codegrees $\simeq [(31.092, 1), (5.346, 1), (4, 2), (3.561, 1)]$, roots of $x^5 - 48x^4 + 632x^3 - 3600x^2 + 9472x - 9472$.
\item Properties: non-cyclo.
\end{itemize}

\item $\FPdim \ \alpha_{11}^8 - 5\alpha_{11}^6 + 8\alpha_{11}^4 - 3\alpha_{11}^2 + 3 \simeq 34.646$, type $[1, \alpha_{11}, \alpha_{11}^2-1,\alpha_{11}^3-2\alpha_{11},\alpha_{11}^4-3\alpha_{11}^2+1]$, one fusion ring (\textnumero 16): 
$$ \normalsize{
 \left[ \begin{smallmatrix}1 & 0 & 0 & 0 & 0 \\ 0 & 1 & 0 & 0 & 0 \\ 0 & 0 & 1 & 0 & 0 \\ 0 & 0 & 0 & 1 & 0 \\ 0 & 0 & 0 & 0 & 1 \end{smallmatrix} \right], \ 
 \left[ \begin{smallmatrix} 0 & 1 & 0 & 0 & 0 \\ 1 & 0 & 1 & 0 & 0 \\ 0 & 1 & 0 & 1 & 0 \\ 0 & 0 & 1 & 0 & 1 \\ 0 & 0 & 0 & 1 & 1 \end{smallmatrix} \right], \ 
 \left[ \begin{smallmatrix} 0 & 0 & 1 & 0 & 0 \\ 0 & 1 & 0 & 1 & 0 \\ 1 & 0 & 1 & 0 & 1 \\ 0 & 1 & 0 & 1 & 1 \\ 0 & 0 & 1 & 1 & 1 \end{smallmatrix} \right], \ 
 \left[ \begin{smallmatrix} 0 & 0 & 0 & 1 & 0 \\ 0 & 0 & 1 & 0 & 1 \\ 0 & 1 & 0 & 1 & 1 \\ 1 & 0 & 1 & 1 & 1 \\ 0 & 1 & 1 & 1 & 1 \end{smallmatrix} \right], \ 
 \left[ \begin{smallmatrix} 0 & 0 & 0 & 0 & 1 \\ 0 & 0 & 0 & 1 & 1 \\ 0 & 0 & 1 & 1 & 1 \\ 0 & 1 & 1 & 1 & 1 \\ 1 & 1 & 1 & 1 & 1  \end{smallmatrix} \right]} $$ 
\begin{itemize}
\item Formal codegrees $\simeq [(34.646, 1), (9.408, 1), (4.815, 1), (3.324, 1), (2.807, 1)]$, roots of $x^5 - 55x^4 + 847x^3 - 5324x^2 + 14641x - 14641$.
\item[$\star$] Properties: simple.
\item Model: $\PSU(2)_9$. 
\end{itemize}
\end{itemize}

\subsection{Rank 6} \label{thm:rank6}
\begin{itemize}

\item $\FPdim \ 6$, type $[1,1,1,1,1,1]$, two fusion rings (\textnumero 1,2): 
$$ \normalsize{ \begin{smallmatrix}1 & 0 & 0 & 0 & 0 & 0 \\ 0 & 1 & 0 & 0 & 0 & 0 \\ 0 & 0 & 1 & 0 & 0 & 0 \\ 0 & 0 & 0 & 1 & 0 & 0 \\ 0 & 0 & 0 & 0 & 1 & 0 \\ 0 & 0 & 0 & 0 & 0 & 1 \end{smallmatrix} , \ 
  \begin{smallmatrix} 0 & 1 & 0 & 0 & 0 & 0 \\ 0 & 0 & 1 & 0 & 0 & 0 \\ 1 & 0 & 0 & 0 & 0 & 0 \\ 0 & 0 & 0 & 0 & 0 & 1 \\ 0 & 0 & 0 & 1 & 0 & 0 \\ 0 & 0 & 0 & 0 & 1 & 0 \end{smallmatrix} , \ 
  \begin{smallmatrix} 0 & 0 & 1 & 0 & 0 & 0 \\ 1 & 0 & 0 & 0 & 0 & 0 \\ 0 & 1 & 0 & 0 & 0 & 0 \\ 0 & 0 & 0 & 0 & 1 & 0 \\ 0 & 0 & 0 & 0 & 0 & 1 \\ 0 & 0 & 0 & 1 & 0 & 0 \end{smallmatrix} , \ 
  \begin{smallmatrix} 0 & 0 & 0 & 1 & 0 & 0 \\ 0 & 0 & 0 & 0 & 0 & 1 \\ 0 & 0 & 0 & 0 & 1 & 0 \\ 0 & 1 & 0 & 0 & 0 & 0 \\ 1 & 0 & 0 & 0 & 0 & 0 \\ 0 & 0 & 1 & 0 & 0 & 0 \end{smallmatrix} , \ 
  \begin{smallmatrix} 0 & 0 & 0 & 0 & 1 & 0 \\ 0 & 0 & 0 & 1 & 0 & 0 \\ 0 & 0 & 0 & 0 & 0 & 1 \\ 1 & 0 & 0 & 0 & 0 & 0 \\ 0 & 0 & 1 & 0 & 0 & 0 \\ 0 & 1 & 0 & 0 & 0 & 0 \end{smallmatrix} , \ 
  \begin{smallmatrix} 0 & 0 & 0 & 0 & 0 & 1 \\ 0 & 0 & 0 & 0 & 1 & 0 \\ 0 & 0 & 0 & 1 & 0 & 0 \\ 0 & 0 & 1 & 0 & 0 & 0 \\ 0 & 1 & 0 & 0 & 0 & 0 \\ 1 & 0 & 0 & 0 & 0 & 0  \end{smallmatrix} } $$ 
 \begin{itemize}
\item Formal codegrees: $[(6, 6)]$.
\item[$\star$] Properties: pointed, quadratic $(C_3,1,0)$.
\item Model: $\VVec(C_6)$. 
\end{itemize}

$$ \normalsize{\left[ \begin{smallmatrix}1 & 0 & 0 & 0 & 0 & 0 \\ 0 & 1 & 0 & 0 & 0 & 0 \\ 0 & 0 & 1 & 0 & 0 & 0 \\ 0 & 0 & 0 & 1 & 0 & 0 \\ 0 & 0 & 0 & 0 & 1 & 0 \\ 0 & 0 & 0 & 0 & 0 & 1 \end{smallmatrix} \right], \ 
 \left[ \begin{smallmatrix} 0 & 1 & 0 & 0 & 0 & 0 \\ 0 & 0 & 1 & 0 & 0 & 0 \\ 1 & 0 & 0 & 0 & 0 & 0 \\ 0 & 0 & 0 & 0 & 1 & 0 \\ 0 & 0 & 0 & 0 & 0 & 1 \\ 0 & 0 & 0 & 1 & 0 & 0 \end{smallmatrix} \right], \ 
 \left[ \begin{smallmatrix} 0 & 0 & 1 & 0 & 0 & 0 \\ 1 & 0 & 0 & 0 & 0 & 0 \\ 0 & 1 & 0 & 0 & 0 & 0 \\ 0 & 0 & 0 & 0 & 0 & 1 \\ 0 & 0 & 0 & 1 & 0 & 0 \\ 0 & 0 & 0 & 0 & 1 & 0 \end{smallmatrix} \right], \ 
 \left[ \begin{smallmatrix} 0 & 0 & 0 & 1 & 0 & 0 \\ 0 & 0 & 0 & 0 & 0 & 1 \\ 0 & 0 & 0 & 0 & 1 & 0 \\ 1 & 0 & 0 & 0 & 0 & 0 \\ 0 & 0 & 1 & 0 & 0 & 0 \\ 0 & 1 & 0 & 0 & 0 & 0 \end{smallmatrix} \right], \ 
 \left[ \begin{smallmatrix} 0 & 0 & 0 & 0 & 1 & 0 \\ 0 & 0 & 0 & 1 & 0 & 0 \\ 0 & 0 & 0 & 0 & 0 & 1 \\ 0 & 1 & 0 & 0 & 0 & 0 \\ 1 & 0 & 0 & 0 & 0 & 0 \\ 0 & 0 & 1 & 0 & 0 & 0 \end{smallmatrix} \right], \ 
 \left[ \begin{smallmatrix} 0 & 0 & 0 & 0 & 0 & 1 \\ 0 & 0 & 0 & 0 & 1 & 0 \\ 0 & 0 & 0 & 1 & 0 & 0 \\ 0 & 0 & 1 & 0 & 0 & 0 \\ 0 & 1 & 0 & 0 & 0 & 0 \\ 1 & 0 & 0 & 0 & 0 & 0  \end{smallmatrix} \right]} $$ 
\begin{itemize}
\item[$\star$] Properties: noncommutative, quadratic $(C_3,-1,0)$.
\item Model: $\VVec(S_3)$. 
\end{itemize}

\item $\FPdim \ 8$, type $[1,1,1,1,\alpha_4,\alpha_4]$, four fusion rings (\textnumero 3-6): 
$$ \normalsize{\left[ \begin{smallmatrix}1 & 0 & 0 & 0 & 0 & 0 \\ 0 & 1 & 0 & 0 & 0 & 0 \\ 0 & 0 & 1 & 0 & 0 & 0 \\ 0 & 0 & 0 & 1 & 0 & 0 \\ 0 & 0 & 0 & 0 & 1 & 0 \\ 0 & 0 & 0 & 0 & 0 & 1 \end{smallmatrix} \right], \ 
 \left[ \begin{smallmatrix} 0 & 1 & 0 & 0 & 0 & 0 \\ 1 & 0 & 0 & 0 & 0 & 0 \\ 0 & 0 & 0 & 1 & 0 & 0 \\ 0 & 0 & 1 & 0 & 0 & 0 \\ 0 & 0 & 0 & 0 & 1 & 0 \\ 0 & 0 & 0 & 0 & 0 & 1 \end{smallmatrix} \right], \ 
 \left[ \begin{smallmatrix} 0 & 0 & 1 & 0 & 0 & 0 \\ 0 & 0 & 0 & 1 & 0 & 0 \\ 1 & 0 & 0 & 0 & 0 & 0 \\ 0 & 1 & 0 & 0 & 0 & 0 \\ 0 & 0 & 0 & 0 & 0 & 1 \\ 0 & 0 & 0 & 0 & 1 & 0 \end{smallmatrix} \right], \ 
 \left[ \begin{smallmatrix} 0 & 0 & 0 & 1 & 0 & 0 \\ 0 & 0 & 1 & 0 & 0 & 0 \\ 0 & 1 & 0 & 0 & 0 & 0 \\ 1 & 0 & 0 & 0 & 0 & 0 \\ 0 & 0 & 0 & 0 & 0 & 1 \\ 0 & 0 & 0 & 0 & 1 & 0 \end{smallmatrix} \right], \ 
 \left[ \begin{smallmatrix} 0 & 0 & 0 & 0 & 1 & 0 \\ 0 & 0 & 0 & 0 & 1 & 0 \\ 0 & 0 & 0 & 0 & 0 & 1 \\ 0 & 0 & 0 & 0 & 0 & 1 \\ 1 & 1 & 0 & 0 & 0 & 0 \\ 0 & 0 & 1 & 1 & 0 & 0 \end{smallmatrix} \right], \ 
 \left[ \begin{smallmatrix} 0 & 0 & 0 & 0 & 0 & 1 \\ 0 & 0 & 0 & 0 & 0 & 1 \\ 0 & 0 & 0 & 0 & 1 & 0 \\ 0 & 0 & 0 & 0 & 1 & 0 \\ 0 & 0 & 1 & 1 & 0 & 0 \\ 1 & 1 & 0 & 0 & 0 & 0  \end{smallmatrix} \right]} $$ 
\begin{itemize}
\item Formal codegrees: $[(8, 4), (4, 2)]$.
\item[$\star$] Properties: quadratic with $G=C_2^2$. 
\item Model: $\VVec(C_2) \otimes \SU(2)_2$.
\end{itemize}

$$ \normalsize{\left[ \begin{smallmatrix}1 & 0 & 0 & 0 & 0 & 0 \\ 0 & 1 & 0 & 0 & 0 & 0 \\ 0 & 0 & 1 & 0 & 0 & 0 \\ 0 & 0 & 0 & 1 & 0 & 0 \\ 0 & 0 & 0 & 0 & 1 & 0 \\ 0 & 0 & 0 & 0 & 0 & 1 \end{smallmatrix} \right], \ 
 \left[ \begin{smallmatrix} 0 & 1 & 0 & 0 & 0 & 0 \\ 0 & 0 & 0 & 1 & 0 & 0 \\ 1 & 0 & 0 & 0 & 0 & 0 \\ 0 & 0 & 1 & 0 & 0 & 0 \\ 0 & 0 & 0 & 0 & 0 & 1 \\ 0 & 0 & 0 & 0 & 1 & 0 \end{smallmatrix} \right], \ 
 \left[ \begin{smallmatrix} 0 & 0 & 1 & 0 & 0 & 0 \\ 1 & 0 & 0 & 0 & 0 & 0 \\ 0 & 0 & 0 & 1 & 0 & 0 \\ 0 & 1 & 0 & 0 & 0 & 0 \\ 0 & 0 & 0 & 0 & 0 & 1 \\ 0 & 0 & 0 & 0 & 1 & 0 \end{smallmatrix} \right], \ 
 \left[ \begin{smallmatrix} 0 & 0 & 0 & 1 & 0 & 0 \\ 0 & 0 & 1 & 0 & 0 & 0 \\ 0 & 1 & 0 & 0 & 0 & 0 \\ 1 & 0 & 0 & 0 & 0 & 0 \\ 0 & 0 & 0 & 0 & 1 & 0 \\ 0 & 0 & 0 & 0 & 0 & 1 \end{smallmatrix} \right], \ 
 \left[ \begin{smallmatrix} 0 & 0 & 0 & 0 & 1 & 0 \\ 0 & 0 & 0 & 0 & 0 & 1 \\ 0 & 0 & 0 & 0 & 0 & 1 \\ 0 & 0 & 0 & 0 & 1 & 0 \\ 0 & 1 & 1 & 0 & 0 & 0 \\ 1 & 0 & 0 & 1 & 0 & 0 \end{smallmatrix} \right], \ 
 \left[ \begin{smallmatrix} 0 & 0 & 0 & 0 & 0 & 1 \\ 0 & 0 & 0 & 0 & 1 & 0 \\ 0 & 0 & 0 & 0 & 1 & 0 \\ 0 & 0 & 0 & 0 & 0 & 1 \\ 1 & 0 & 0 & 1 & 0 & 0 \\ 0 & 1 & 1 & 0 & 0 & 0  \end{smallmatrix} \right]} $$ 
\begin{itemize}
\item Formal codegrees: $[(8, 4), (4, 2)]$.
\item[$\star$] Properties:  quadratic with $G=C_4$, unitarily categorified (see \S \ref{r6dim8}), non-modular, weakly group-theoretical,  
\item Model: zesting of $\VVec(C_2) \otimes \SU(2)_2$, see \S \ref{sub:zest}.
\end{itemize}

$$ \normalsize{\left[ \begin{smallmatrix}1 & 0 & 0 & 0 & 0 & 0 \\ 0 & 1 & 0 & 0 & 0 & 0 \\ 0 & 0 & 1 & 0 & 0 & 0 \\ 0 & 0 & 0 & 1 & 0 & 0 \\ 0 & 0 & 0 & 0 & 1 & 0 \\ 0 & 0 & 0 & 0 & 0 & 1 \end{smallmatrix} \right], \ 
 \left[ \begin{smallmatrix} 0 & 1 & 0 & 0 & 0 & 0 \\ 0 & 0 & 0 & 1 & 0 & 0 \\ 1 & 0 & 0 & 0 & 0 & 0 \\ 0 & 0 & 1 & 0 & 0 & 0 \\ 0 & 0 & 0 & 0 & 0 & 1 \\ 0 & 0 & 0 & 0 & 1 & 0 \end{smallmatrix} \right], \ 
 \left[ \begin{smallmatrix} 0 & 0 & 1 & 0 & 0 & 0 \\ 1 & 0 & 0 & 0 & 0 & 0 \\ 0 & 0 & 0 & 1 & 0 & 0 \\ 0 & 1 & 0 & 0 & 0 & 0 \\ 0 & 0 & 0 & 0 & 0 & 1 \\ 0 & 0 & 0 & 0 & 1 & 0 \end{smallmatrix} \right], \ 
 \left[ \begin{smallmatrix} 0 & 0 & 0 & 1 & 0 & 0 \\ 0 & 0 & 1 & 0 & 0 & 0 \\ 0 & 1 & 0 & 0 & 0 & 0 \\ 1 & 0 & 0 & 0 & 0 & 0 \\ 0 & 0 & 0 & 0 & 1 & 0 \\ 0 & 0 & 0 & 0 & 0 & 1 \end{smallmatrix} \right], \ 
 \left[ \begin{smallmatrix} 0 & 0 & 0 & 0 & 1 & 0 \\ 0 & 0 & 0 & 0 & 0 & 1 \\ 0 & 0 & 0 & 0 & 0 & 1 \\ 0 & 0 & 0 & 0 & 1 & 0 \\ 1 & 0 & 0 & 1 & 0 & 0 \\ 0 & 1 & 1 & 0 & 0 & 0 \end{smallmatrix} \right], \ 
 \left[ \begin{smallmatrix} 0 & 0 & 0 & 0 & 0 & 1 \\ 0 & 0 & 0 & 0 & 1 & 0 \\ 0 & 0 & 0 & 0 & 1 & 0 \\ 0 & 0 & 0 & 0 & 0 & 1 \\ 0 & 1 & 1 & 0 & 0 & 0 \\ 1 & 0 & 0 & 1 & 0 & 0  \end{smallmatrix} \right]} $$ 
\begin{itemize}
\item Formal codegrees: $[(8, 4), (4, 2)]$.
\item[$\star$] Properties: quadratic with $G=C_4$, unitarily categorified (see \S \ref{r6dim8}), non-modular, weakly group-theoretical, 
\item Model: zesting of $\VVec(C_2) \otimes \SU(2)_2$, see \S \ref{sub:zest}.
\end{itemize}

$$ \normalsize{\left[ \begin{smallmatrix}1 & 0 & 0 & 0 & 0 & 0 \\ 0 & 1 & 0 & 0 & 0 & 0 \\ 0 & 0 & 1 & 0 & 0 & 0 \\ 0 & 0 & 0 & 1 & 0 & 0 \\ 0 & 0 & 0 & 0 & 1 & 0 \\ 0 & 0 & 0 & 0 & 0 & 1 \end{smallmatrix} \right], \ 
 \left[ \begin{smallmatrix} 0 & 1 & 0 & 0 & 0 & 0 \\ 1 & 0 & 0 & 0 & 0 & 0 \\ 0 & 0 & 0 & 1 & 0 & 0 \\ 0 & 0 & 1 & 0 & 0 & 0 \\ 0 & 0 & 0 & 0 & 1 & 0 \\ 0 & 0 & 0 & 0 & 0 & 1 \end{smallmatrix} \right], \ 
 \left[ \begin{smallmatrix} 0 & 0 & 1 & 0 & 0 & 0 \\ 0 & 0 & 0 & 1 & 0 & 0 \\ 1 & 0 & 0 & 0 & 0 & 0 \\ 0 & 1 & 0 & 0 & 0 & 0 \\ 0 & 0 & 0 & 0 & 0 & 1 \\ 0 & 0 & 0 & 0 & 1 & 0 \end{smallmatrix} \right], \ 
 \left[ \begin{smallmatrix} 0 & 0 & 0 & 1 & 0 & 0 \\ 0 & 0 & 1 & 0 & 0 & 0 \\ 0 & 1 & 0 & 0 & 0 & 0 \\ 1 & 0 & 0 & 0 & 0 & 0 \\ 0 & 0 & 0 & 0 & 0 & 1 \\ 0 & 0 & 0 & 0 & 1 & 0 \end{smallmatrix} \right], \ 
 \left[ \begin{smallmatrix} 0 & 0 & 0 & 0 & 1 & 0 \\ 0 & 0 & 0 & 0 & 1 & 0 \\ 0 & 0 & 0 & 0 & 0 & 1 \\ 0 & 0 & 0 & 0 & 0 & 1 \\ 0 & 0 & 1 & 1 & 0 & 0 \\ 1 & 1 & 0 & 0 & 0 & 0 \end{smallmatrix} \right], \ 
 \left[ \begin{smallmatrix} 0 & 0 & 0 & 0 & 0 & 1 \\ 0 & 0 & 0 & 0 & 0 & 1 \\ 0 & 0 & 0 & 0 & 1 & 0 \\ 0 & 0 & 0 & 0 & 1 & 0 \\ 1 & 1 & 0 & 0 & 0 & 0 \\ 0 & 0 & 1 & 1 & 0 & 0  \end{smallmatrix} \right]} $$ 
\begin{itemize}
\item Formal codegrees: $[(8, 4), (4, 2)]$.
\item[$\star$] Properties: quadratic with $G=C_2^2$, unitarily categorified (see \S \ref{r6dim8}), non-modular, weakly group-theoretical, 
\item Model: zesting of $\VVec(C_2) \otimes \SU(2)_2$, see \S \ref{sub:zest}.
\end{itemize}

\item $\FPdim \ 10$, type $[1, 1, 1, 1, 1, \sqrt{5}]$, one fusion ring (\textnumero 7): 
$$ \normalsize{\left[ \begin{smallmatrix}1 & 0 & 0 & 0 & 0 & 0 \\ 0 & 1 & 0 & 0 & 0 & 0 \\ 0 & 0 & 1 & 0 & 0 & 0 \\ 0 & 0 & 0 & 1 & 0 & 0 \\ 0 & 0 & 0 & 0 & 1 & 0 \\ 0 & 0 & 0 & 0 & 0 & 1 \end{smallmatrix} \right], \ 
 \left[ \begin{smallmatrix} 0 & 1 & 0 & 0 & 0 & 0 \\ 0 & 0 & 0 & 0 & 1 & 0 \\ 1 & 0 & 0 & 0 & 0 & 0 \\ 0 & 0 & 1 & 0 & 0 & 0 \\ 0 & 0 & 0 & 1 & 0 & 0 \\ 0 & 0 & 0 & 0 & 0 & 1 \end{smallmatrix} \right], \ 
 \left[ \begin{smallmatrix} 0 & 0 & 1 & 0 & 0 & 0 \\ 1 & 0 & 0 & 0 & 0 & 0 \\ 0 & 0 & 0 & 1 & 0 & 0 \\ 0 & 0 & 0 & 0 & 1 & 0 \\ 0 & 1 & 0 & 0 & 0 & 0 \\ 0 & 0 & 0 & 0 & 0 & 1 \end{smallmatrix} \right], \ 
 \left[ \begin{smallmatrix} 0 & 0 & 0 & 1 & 0 & 0 \\ 0 & 0 & 1 & 0 & 0 & 0 \\ 0 & 0 & 0 & 0 & 1 & 0 \\ 0 & 1 & 0 & 0 & 0 & 0 \\ 1 & 0 & 0 & 0 & 0 & 0 \\ 0 & 0 & 0 & 0 & 0 & 1 \end{smallmatrix} \right], \ 
 \left[ \begin{smallmatrix} 0 & 0 & 0 & 0 & 1 & 0 \\ 0 & 0 & 0 & 1 & 0 & 0 \\ 0 & 1 & 0 & 0 & 0 & 0 \\ 1 & 0 & 0 & 0 & 0 & 0 \\ 0 & 0 & 1 & 0 & 0 & 0 \\ 0 & 0 & 0 & 0 & 0 & 1 \end{smallmatrix} \right], \ 
 \left[ \begin{smallmatrix} 0 & 0 & 0 & 0 & 0 & 1 \\ 0 & 0 & 0 & 0 & 0 & 1 \\ 0 & 0 & 0 & 0 & 0 & 1 \\ 0 & 0 & 0 & 0 & 0 & 1 \\ 0 & 0 & 0 & 0 & 0 & 1 \\ 1 & 1 & 1 & 1 & 1 & 0  \end{smallmatrix} \right]} $$ 
\begin{itemize}
\item Formal codegrees: $[(10, 2), (5, 4)]$.
\item[$\star$] Properties: near-group $C_5+0$.
\item Model: $TY(C_5)$.
\end{itemize}

\item FPdim  $3\alpha_5 + 6 \simeq 10.854$, type $[1, 1, 1, \alpha_5, \alpha_5, \alpha_5]$, one fusion ring (\textnumero 8): 
$$ \normalsize{\left[ \begin{smallmatrix}1 & 0 & 0 & 0 & 0 & 0 \\ 0 & 1 & 0 & 0 & 0 & 0 \\ 0 & 0 & 1 & 0 & 0 & 0 \\ 0 & 0 & 0 & 1 & 0 & 0 \\ 0 & 0 & 0 & 0 & 1 & 0 \\ 0 & 0 & 0 & 0 & 0 & 1 \end{smallmatrix} \right], \ 
 \left[ \begin{smallmatrix} 0 & 1 & 0 & 0 & 0 & 0 \\ 0 & 0 & 1 & 0 & 0 & 0 \\ 1 & 0 & 0 & 0 & 0 & 0 \\ 0 & 0 & 0 & 0 & 0 & 1 \\ 0 & 0 & 0 & 1 & 0 & 0 \\ 0 & 0 & 0 & 0 & 1 & 0 \end{smallmatrix} \right], \ 
 \left[ \begin{smallmatrix} 0 & 0 & 1 & 0 & 0 & 0 \\ 1 & 0 & 0 & 0 & 0 & 0 \\ 0 & 1 & 0 & 0 & 0 & 0 \\ 0 & 0 & 0 & 0 & 1 & 0 \\ 0 & 0 & 0 & 0 & 0 & 1 \\ 0 & 0 & 0 & 1 & 0 & 0 \end{smallmatrix} \right], \ 
 \left[ \begin{smallmatrix} 0 & 0 & 0 & 1 & 0 & 0 \\ 0 & 0 & 0 & 0 & 0 & 1 \\ 0 & 0 & 0 & 0 & 1 & 0 \\ 0 & 1 & 0 & 0 & 1 & 0 \\ 1 & 0 & 0 & 0 & 0 & 1 \\ 0 & 0 & 1 & 1 & 0 & 0 \end{smallmatrix} \right], \ 
 \left[ \begin{smallmatrix} 0 & 0 & 0 & 0 & 1 & 0 \\ 0 & 0 & 0 & 1 & 0 & 0 \\ 0 & 0 & 0 & 0 & 0 & 1 \\ 1 & 0 & 0 & 0 & 0 & 1 \\ 0 & 0 & 1 & 1 & 0 & 0 \\ 0 & 1 & 0 & 0 & 1 & 0 \end{smallmatrix} \right], \ 
 \left[ \begin{smallmatrix} 0 & 0 & 0 & 0 & 0 & 1 \\ 0 & 0 & 0 & 0 & 1 & 0 \\ 0 & 0 & 0 & 1 & 0 & 0 \\ 0 & 0 & 1 & 1 & 0 & 0 \\ 0 & 1 & 0 & 0 & 1 & 0 \\ 1 & 0 & 0 & 0 & 0 & 1  \end{smallmatrix} \right]} $$ 
\begin{itemize}
\item Formal codegrees: $[((15+3\sqrt{5})/2, 3), ((15-3\sqrt{5})/2, 3)]$.
\item[$\star$] Properties: quadratic with $G=C_3$.
\item Model: $\VVec(C_3) \otimes \PSU(2)_3$.
\end{itemize}

\item $\FPdim \ 12$, type $[1, 1, 1, 1, 2, 2]$, two fusion rings (\textnumero 9,10): 
$$ \normalsize{\left[ \begin{smallmatrix}1 & 0 & 0 & 0 & 0 & 0 \\ 0 & 1 & 0 & 0 & 0 & 0 \\ 0 & 0 & 1 & 0 & 0 & 0 \\ 0 & 0 & 0 & 1 & 0 & 0 \\ 0 & 0 & 0 & 0 & 1 & 0 \\ 0 & 0 & 0 & 0 & 0 & 1 \end{smallmatrix} \right], \ 
 \left[ \begin{smallmatrix} 0 & 1 & 0 & 0 & 0 & 0 \\ 0 & 0 & 0 & 1 & 0 & 0 \\ 1 & 0 & 0 & 0 & 0 & 0 \\ 0 & 0 & 1 & 0 & 0 & 0 \\ 0 & 0 & 0 & 0 & 0 & 1 \\ 0 & 0 & 0 & 0 & 1 & 0 \end{smallmatrix} \right], \ 
 \left[ \begin{smallmatrix} 0 & 0 & 1 & 0 & 0 & 0 \\ 1 & 0 & 0 & 0 & 0 & 0 \\ 0 & 0 & 0 & 1 & 0 & 0 \\ 0 & 1 & 0 & 0 & 0 & 0 \\ 0 & 0 & 0 & 0 & 0 & 1 \\ 0 & 0 & 0 & 0 & 1 & 0 \end{smallmatrix} \right], \ 
 \left[ \begin{smallmatrix} 0 & 0 & 0 & 1 & 0 & 0 \\ 0 & 0 & 1 & 0 & 0 & 0 \\ 0 & 1 & 0 & 0 & 0 & 0 \\ 1 & 0 & 0 & 0 & 0 & 0 \\ 0 & 0 & 0 & 0 & 1 & 0 \\ 0 & 0 & 0 & 0 & 0 & 1 \end{smallmatrix} \right], \ 
 \left[ \begin{smallmatrix} 0 & 0 & 0 & 0 & 1 & 0 \\ 0 & 0 & 0 & 0 & 0 & 1 \\ 0 & 0 & 0 & 0 & 0 & 1 \\ 0 & 0 & 0 & 0 & 1 & 0 \\ 1 & 0 & 0 & 1 & 1 & 0 \\ 0 & 1 & 1 & 0 & 0 & 1 \end{smallmatrix} \right], \ 
 \left[ \begin{smallmatrix} 0 & 0 & 0 & 0 & 0 & 1 \\ 0 & 0 & 0 & 0 & 1 & 0 \\ 0 & 0 & 0 & 0 & 1 & 0 \\ 0 & 0 & 0 & 0 & 0 & 1 \\ 0 & 1 & 1 & 0 & 0 & 1 \\ 1 & 0 & 0 & 1 & 1 & 0  \end{smallmatrix} \right]} $$ 
\begin{itemize}
\item Formal codegrees: $[(12, 2), (6, 2), (4, 2)]$.
\item[$\star$] Properties: quadratic with $G=C_4$.
\item Model: $\Rep(C_3 \rtimes C_4)$.
\end{itemize}

$$ \normalsize{\left[ \begin{smallmatrix}1 & 0 & 0 & 0 & 0 & 0 \\ 0 & 1 & 0 & 0 & 0 & 0 \\ 0 & 0 & 1 & 0 & 0 & 0 \\ 0 & 0 & 0 & 1 & 0 & 0 \\ 0 & 0 & 0 & 0 & 1 & 0 \\ 0 & 0 & 0 & 0 & 0 & 1 \end{smallmatrix} \right], \ 
 \left[ \begin{smallmatrix} 0 & 1 & 0 & 0 & 0 & 0 \\ 1 & 0 & 0 & 0 & 0 & 0 \\ 0 & 0 & 0 & 1 & 0 & 0 \\ 0 & 0 & 1 & 0 & 0 & 0 \\ 0 & 0 & 0 & 0 & 1 & 0 \\ 0 & 0 & 0 & 0 & 0 & 1 \end{smallmatrix} \right], \ 
 \left[ \begin{smallmatrix} 0 & 0 & 1 & 0 & 0 & 0 \\ 0 & 0 & 0 & 1 & 0 & 0 \\ 1 & 0 & 0 & 0 & 0 & 0 \\ 0 & 1 & 0 & 0 & 0 & 0 \\ 0 & 0 & 0 & 0 & 0 & 1 \\ 0 & 0 & 0 & 0 & 1 & 0 \end{smallmatrix} \right], \ 
 \left[ \begin{smallmatrix} 0 & 0 & 0 & 1 & 0 & 0 \\ 0 & 0 & 1 & 0 & 0 & 0 \\ 0 & 1 & 0 & 0 & 0 & 0 \\ 1 & 0 & 0 & 0 & 0 & 0 \\ 0 & 0 & 0 & 0 & 0 & 1 \\ 0 & 0 & 0 & 0 & 1 & 0 \end{smallmatrix} \right], \ 
 \left[ \begin{smallmatrix} 0 & 0 & 0 & 0 & 1 & 0 \\ 0 & 0 & 0 & 0 & 1 & 0 \\ 0 & 0 & 0 & 0 & 0 & 1 \\ 0 & 0 & 0 & 0 & 0 & 1 \\ 1 & 1 & 0 & 0 & 1 & 0 \\ 0 & 0 & 1 & 1 & 0 & 1 \end{smallmatrix} \right], \ 
 \left[ \begin{smallmatrix} 0 & 0 & 0 & 0 & 0 & 1 \\ 0 & 0 & 0 & 0 & 0 & 1 \\ 0 & 0 & 0 & 0 & 1 & 0 \\ 0 & 0 & 0 & 0 & 1 & 0 \\ 0 & 0 & 1 & 1 & 0 & 1 \\ 1 & 1 & 0 & 0 & 1 & 0  \end{smallmatrix} \right]} $$ 
\begin{itemize}
\item Formal codegrees: $[(12, 2), (6, 2), (4, 2)]$.
\item[$\star$] Properties:  quadratic with $G=C_2^2$.
\item Model: $\VVec(C_2) \otimes \Rep(S_3)$, $\Rep(D_6)$.
\end{itemize}

\item $\FPdim \ (\sqrt{21}+21)/2 \simeq 12.791$, type $[1, 1, 1, 1, 1, (\sqrt{21}+1)/2]$, one fusion ring (\textnumero 11): 
$$ \normalsize{\left[ \begin{smallmatrix}1 & 0 & 0 & 0 & 0 & 0 \\ 0 & 1 & 0 & 0 & 0 & 0 \\ 0 & 0 & 1 & 0 & 0 & 0 \\ 0 & 0 & 0 & 1 & 0 & 0 \\ 0 & 0 & 0 & 0 & 1 & 0 \\ 0 & 0 & 0 & 0 & 0 & 1 \end{smallmatrix} \right], \ 
 \left[ \begin{smallmatrix} 0 & 1 & 0 & 0 & 0 & 0 \\ 0 & 0 & 0 & 0 & 1 & 0 \\ 1 & 0 & 0 & 0 & 0 & 0 \\ 0 & 0 & 1 & 0 & 0 & 0 \\ 0 & 0 & 0 & 1 & 0 & 0 \\ 0 & 0 & 0 & 0 & 0 & 1 \end{smallmatrix} \right], \ 
 \left[ \begin{smallmatrix} 0 & 0 & 1 & 0 & 0 & 0 \\ 1 & 0 & 0 & 0 & 0 & 0 \\ 0 & 0 & 0 & 1 & 0 & 0 \\ 0 & 0 & 0 & 0 & 1 & 0 \\ 0 & 1 & 0 & 0 & 0 & 0 \\ 0 & 0 & 0 & 0 & 0 & 1 \end{smallmatrix} \right], \ 
 \left[ \begin{smallmatrix} 0 & 0 & 0 & 1 & 0 & 0 \\ 0 & 0 & 1 & 0 & 0 & 0 \\ 0 & 0 & 0 & 0 & 1 & 0 \\ 0 & 1 & 0 & 0 & 0 & 0 \\ 1 & 0 & 0 & 0 & 0 & 0 \\ 0 & 0 & 0 & 0 & 0 & 1 \end{smallmatrix} \right], \ 
 \left[ \begin{smallmatrix} 0 & 0 & 0 & 0 & 1 & 0 \\ 0 & 0 & 0 & 1 & 0 & 0 \\ 0 & 1 & 0 & 0 & 0 & 0 \\ 1 & 0 & 0 & 0 & 0 & 0 \\ 0 & 0 & 1 & 0 & 0 & 0 \\ 0 & 0 & 0 & 0 & 0 & 1 \end{smallmatrix} \right], \ 
 \left[ \begin{smallmatrix} 0 & 0 & 0 & 0 & 0 & 1 \\ 0 & 0 & 0 & 0 & 0 & 1 \\ 0 & 0 & 0 & 0 & 0 & 1 \\ 0 & 0 & 0 & 0 & 0 & 1 \\ 0 & 0 & 0 & 0 & 0 & 1 \\ 1 & 1 & 1 & 1 & 1 & 1  \end{smallmatrix} \right]} $$ 
 \begin{itemize}
\item Formal codegrees: $[((21+\sqrt{21})/2, 1), ((21-\sqrt{21})/2, 1), (5, 4)]$.
\item Properties: near-group $C_5+1$, non-Lagrange, non-d-number, non-Drinfeld.
\end{itemize}

\item $\FPdim \ 4\alpha_5+8 \simeq 14.472$, type $[1, 1, \alpha_4, \alpha_5, \alpha_5, \alpha_4\alpha_5]$, one fusion ring (\textnumero 12): 
$$ \normalsize{\left[ \begin{smallmatrix}1 & 0 & 0 & 0 & 0 & 0 \\ 0 & 1 & 0 & 0 & 0 & 0 \\ 0 & 0 & 1 & 0 & 0 & 0 \\ 0 & 0 & 0 & 1 & 0 & 0 \\ 0 & 0 & 0 & 0 & 1 & 0 \\ 0 & 0 & 0 & 0 & 0 & 1 \end{smallmatrix} \right], \ 
 \left[ \begin{smallmatrix} 0 & 1 & 0 & 0 & 0 & 0 \\ 1 & 0 & 0 & 0 & 0 & 0 \\ 0 & 0 & 1 & 0 & 0 & 0 \\ 0 & 0 & 0 & 0 & 1 & 0 \\ 0 & 0 & 0 & 1 & 0 & 0 \\ 0 & 0 & 0 & 0 & 0 & 1 \end{smallmatrix} \right], \ 
 \left[ \begin{smallmatrix} 0 & 0 & 1 & 0 & 0 & 0 \\ 0 & 0 & 1 & 0 & 0 & 0 \\ 1 & 1 & 0 & 0 & 0 & 0 \\ 0 & 0 & 0 & 0 & 0 & 1 \\ 0 & 0 & 0 & 0 & 0 & 1 \\ 0 & 0 & 0 & 1 & 1 & 0 \end{smallmatrix} \right], \ 
 \left[ \begin{smallmatrix} 0 & 0 & 0 & 1 & 0 & 0 \\ 0 & 0 & 0 & 0 & 1 & 0 \\ 0 & 0 & 0 & 0 & 0 & 1 \\ 1 & 0 & 0 & 1 & 0 & 0 \\ 0 & 1 & 0 & 0 & 1 & 0 \\ 0 & 0 & 1 & 0 & 0 & 1 \end{smallmatrix} \right], \ 
 \left[ \begin{smallmatrix} 0 & 0 & 0 & 0 & 1 & 0 \\ 0 & 0 & 0 & 1 & 0 & 0 \\ 0 & 0 & 0 & 0 & 0 & 1 \\ 0 & 1 & 0 & 0 & 1 & 0 \\ 1 & 0 & 0 & 1 & 0 & 0 \\ 0 & 0 & 1 & 0 & 0 & 1 \end{smallmatrix} \right], \ 
 \left[ \begin{smallmatrix} 0 & 0 & 0 & 0 & 0 & 1 \\ 0 & 0 & 0 & 0 & 0 & 1 \\ 0 & 0 & 0 & 1 & 1 & 0 \\ 0 & 0 & 1 & 0 & 0 & 1 \\ 0 & 0 & 1 & 0 & 0 & 1 \\ 1 & 1 & 0 & 1 & 1 & 0  \end{smallmatrix} \right]} $$ 
\begin{itemize}
\item Formal codegrees: $[(10+2\sqrt{5}, 2), (5+\sqrt{5}, 1), (10-2\sqrt{5}, 2), (5-\sqrt{5}, 1)]$.
\item[$\star$] Properties: extension of $\VVec(C_2)$.
\item Model: $\PSU(2)_3 \otimes \SU(2)_2$.
\end{itemize}

\item $\FPdim \ 18$, type $[1, 1, 2, 2, 2, 2]$, two fusion rings (\textnumero 13,14): 
$$ \normalsize{\left[ \begin{smallmatrix}1 & 0 & 0 & 0 & 0 & 0 \\ 0 & 1 & 0 & 0 & 0 & 0 \\ 0 & 0 & 1 & 0 & 0 & 0 \\ 0 & 0 & 0 & 1 & 0 & 0 \\ 0 & 0 & 0 & 0 & 1 & 0 \\ 0 & 0 & 0 & 0 & 0 & 1 \end{smallmatrix} \right], \ 
 \left[ \begin{smallmatrix} 0 & 1 & 0 & 0 & 0 & 0 \\ 1 & 0 & 0 & 0 & 0 & 0 \\ 0 & 0 & 1 & 0 & 0 & 0 \\ 0 & 0 & 0 & 1 & 0 & 0 \\ 0 & 0 & 0 & 0 & 1 & 0 \\ 0 & 0 & 0 & 0 & 0 & 1 \end{smallmatrix} \right], \ 
 \left[ \begin{smallmatrix} 0 & 0 & 1 & 0 & 0 & 0 \\ 0 & 0 & 1 & 0 & 0 & 0 \\ 1 & 1 & 1 & 0 & 0 & 0 \\ 0 & 0 & 0 & 0 & 1 & 1 \\ 0 & 0 & 0 & 1 & 0 & 1 \\ 0 & 0 & 0 & 1 & 1 & 0 \end{smallmatrix} \right], \ 
 \left[ \begin{smallmatrix} 0 & 0 & 0 & 1 & 0 & 0 \\ 0 & 0 & 0 & 1 & 0 & 0 \\ 0 & 0 & 0 & 0 & 1 & 1 \\ 1 & 1 & 0 & 1 & 0 & 0 \\ 0 & 0 & 1 & 0 & 0 & 1 \\ 0 & 0 & 1 & 0 & 1 & 0 \end{smallmatrix} \right], \ 
 \left[ \begin{smallmatrix} 0 & 0 & 0 & 0 & 1 & 0 \\ 0 & 0 & 0 & 0 & 1 & 0 \\ 0 & 0 & 0 & 1 & 0 & 1 \\ 0 & 0 & 1 & 0 & 0 & 1 \\ 1 & 1 & 0 & 0 & 1 & 0 \\ 0 & 0 & 1 & 1 & 0 & 0 \end{smallmatrix} \right], \ 
 \left[ \begin{smallmatrix} 0 & 0 & 0 & 0 & 0 & 1 \\ 0 & 0 & 0 & 0 & 0 & 1 \\ 0 & 0 & 0 & 1 & 1 & 0 \\ 0 & 0 & 1 & 0 & 1 & 0 \\ 0 & 0 & 1 & 1 & 0 & 0 \\ 1 & 1 & 0 & 0 & 0 & 1  \end{smallmatrix} \right]} $$ 
\begin{itemize}
\item Formal codegrees: $[(18, 1), (9, 4), (2, 1)]$.
\item[$\star$] Properties: integral.
\item Model: $\Rep(C_3 \rtimes S_3)$.
\end{itemize}

$$ \normalsize{\left[ \begin{smallmatrix}1 & 0 & 0 & 0 & 0 & 0 \\ 0 & 1 & 0 & 0 & 0 & 0 \\ 0 & 0 & 1 & 0 & 0 & 0 \\ 0 & 0 & 0 & 1 & 0 & 0 \\ 0 & 0 & 0 & 0 & 1 & 0 \\ 0 & 0 & 0 & 0 & 0 & 1 \end{smallmatrix} \right], \ 
 \left[ \begin{smallmatrix} 0 & 1 & 0 & 0 & 0 & 0 \\ 1 & 0 & 0 & 0 & 0 & 0 \\ 0 & 0 & 1 & 0 & 0 & 0 \\ 0 & 0 & 0 & 1 & 0 & 0 \\ 0 & 0 & 0 & 0 & 1 & 0 \\ 0 & 0 & 0 & 0 & 0 & 1 \end{smallmatrix} \right], \ 
 \left[ \begin{smallmatrix} 0 & 0 & 1 & 0 & 0 & 0 \\ 0 & 0 & 1 & 0 & 0 & 0 \\ 1 & 1 & 1 & 0 & 0 & 0 \\ 0 & 0 & 0 & 0 & 1 & 1 \\ 0 & 0 & 0 & 1 & 0 & 1 \\ 0 & 0 & 0 & 1 & 1 & 0 \end{smallmatrix} \right], \ 
 \left[ \begin{smallmatrix} 0 & 0 & 0 & 1 & 0 & 0 \\ 0 & 0 & 0 & 1 & 0 & 0 \\ 0 & 0 & 0 & 0 & 1 & 1 \\ 1 & 1 & 0 & 0 & 1 & 0 \\ 0 & 0 & 1 & 1 & 0 & 0 \\ 0 & 0 & 1 & 0 & 0 & 1 \end{smallmatrix} \right], \ 
 \left[ \begin{smallmatrix} 0 & 0 & 0 & 0 & 1 & 0 \\ 0 & 0 & 0 & 0 & 1 & 0 \\ 0 & 0 & 0 & 1 & 0 & 1 \\ 0 & 0 & 1 & 1 & 0 & 0 \\ 1 & 1 & 0 & 0 & 0 & 1 \\ 0 & 0 & 1 & 0 & 1 & 0 \end{smallmatrix} \right], \ 
 \left[ \begin{smallmatrix} 0 & 0 & 0 & 0 & 0 & 1 \\ 0 & 0 & 0 & 0 & 0 & 1 \\ 0 & 0 & 0 & 1 & 1 & 0 \\ 0 & 0 & 1 & 0 & 0 & 1 \\ 0 & 0 & 1 & 0 & 1 & 0 \\ 1 & 1 & 0 & 1 & 0 & 0  \end{smallmatrix} \right]} $$ 
\begin{itemize}
\item Formal codegrees: $[(18, 1), (9, 4), (2, 1)]$.
\item[$\star$] Properties: integral.
\item Model: $\Rep(D_9)$.
\end{itemize}

\item $\FPdim \ 2(\alpha_7^4-\alpha_7^2+1) \simeq 18.592$, type $[1, 1, \alpha_7, \alpha_7, \alpha_7^2-1,\alpha_7^2-1]$, one fusion ring (\textnumero 15): 
$$ \normalsize{\left[ \begin{smallmatrix}1 & 0 & 0 & 0 & 0 & 0 \\ 0 & 1 & 0 & 0 & 0 & 0 \\ 0 & 0 & 1 & 0 & 0 & 0 \\ 0 & 0 & 0 & 1 & 0 & 0 \\ 0 & 0 & 0 & 0 & 1 & 0 \\ 0 & 0 & 0 & 0 & 0 & 1 \end{smallmatrix} \right], \ 
 \left[ \begin{smallmatrix} 0 & 1 & 0 & 0 & 0 & 0 \\ 1 & 0 & 0 & 0 & 0 & 0 \\ 0 & 0 & 0 & 1 & 0 & 0 \\ 0 & 0 & 1 & 0 & 0 & 0 \\ 0 & 0 & 0 & 0 & 0 & 1 \\ 0 & 0 & 0 & 0 & 1 & 0 \end{smallmatrix} \right], \ 
 \left[ \begin{smallmatrix} 0 & 0 & 1 & 0 & 0 & 0 \\ 0 & 0 & 0 & 1 & 0 & 0 \\ 1 & 0 & 0 & 0 & 1 & 0 \\ 0 & 1 & 0 & 0 & 0 & 1 \\ 0 & 0 & 1 & 0 & 1 & 0 \\ 0 & 0 & 0 & 1 & 0 & 1 \end{smallmatrix} \right], \ 
 \left[ \begin{smallmatrix} 0 & 0 & 0 & 1 & 0 & 0 \\ 0 & 0 & 1 & 0 & 0 & 0 \\ 0 & 1 & 0 & 0 & 0 & 1 \\ 1 & 0 & 0 & 0 & 1 & 0 \\ 0 & 0 & 0 & 1 & 0 & 1 \\ 0 & 0 & 1 & 0 & 1 & 0 \end{smallmatrix} \right], \ 
 \left[ \begin{smallmatrix} 0 & 0 & 0 & 0 & 1 & 0 \\ 0 & 0 & 0 & 0 & 0 & 1 \\ 0 & 0 & 1 & 0 & 1 & 0 \\ 0 & 0 & 0 & 1 & 0 & 1 \\ 1 & 0 & 1 & 0 & 1 & 0 \\ 0 & 1 & 0 & 1 & 0 & 1 \end{smallmatrix} \right], \ 
 \left[ \begin{smallmatrix} 0 & 0 & 0 & 0 & 0 & 1 \\ 0 & 0 & 0 & 0 & 1 & 0 \\ 0 & 0 & 0 & 1 & 0 & 1 \\ 0 & 0 & 1 & 0 & 1 & 0 \\ 0 & 1 & 0 & 1 & 0 & 1 \\ 1 & 0 & 1 & 0 & 1 & 0  \end{smallmatrix} \right]} $$ 
\begin{itemize}
\item Formal codegrees $\simeq [(18.592, 2), (5.726, 2), (3.682, 2)]$ roots of $(x^3 - 28x^2 + 196x - 392)^2$.
\item[$\star$] Properties: extension of $\VVec(C_2)$
\item Model: $\SU(2)_5$.
\end{itemize}

\item $\FPdim \ 12+4\alpha_6 \simeq 18.928$, type $[1, 1, 1, 1, \alpha_6 + 1,  \alpha_6 + 1]$, four fusion rings (\textnumero 16-19): 
$$ \normalsize{\left[ \begin{smallmatrix}1 & 0 & 0 & 0 & 0 & 0 \\ 0 & 1 & 0 & 0 & 0 & 0 \\ 0 & 0 & 1 & 0 & 0 & 0 \\ 0 & 0 & 0 & 1 & 0 & 0 \\ 0 & 0 & 0 & 0 & 1 & 0 \\ 0 & 0 & 0 & 0 & 0 & 1 \end{smallmatrix} \right], \ 
 \left[ \begin{smallmatrix} 0 & 1 & 0 & 0 & 0 & 0 \\ 0 & 0 & 0 & 1 & 0 & 0 \\ 1 & 0 & 0 & 0 & 0 & 0 \\ 0 & 0 & 1 & 0 & 0 & 0 \\ 0 & 0 & 0 & 0 & 0 & 1 \\ 0 & 0 & 0 & 0 & 1 & 0 \end{smallmatrix} \right], \ 
 \left[ \begin{smallmatrix} 0 & 0 & 1 & 0 & 0 & 0 \\ 1 & 0 & 0 & 0 & 0 & 0 \\ 0 & 0 & 0 & 1 & 0 & 0 \\ 0 & 1 & 0 & 0 & 0 & 0 \\ 0 & 0 & 0 & 0 & 0 & 1 \\ 0 & 0 & 0 & 0 & 1 & 0 \end{smallmatrix} \right], \ 
 \left[ \begin{smallmatrix} 0 & 0 & 0 & 1 & 0 & 0 \\ 0 & 0 & 1 & 0 & 0 & 0 \\ 0 & 1 & 0 & 0 & 0 & 0 \\ 1 & 0 & 0 & 0 & 0 & 0 \\ 0 & 0 & 0 & 0 & 1 & 0 \\ 0 & 0 & 0 & 0 & 0 & 1 \end{smallmatrix} \right], \ 
 \left[ \begin{smallmatrix} 0 & 0 & 0 & 0 & 1 & 0 \\ 0 & 0 & 0 & 0 & 0 & 1 \\ 0 & 0 & 0 & 0 & 0 & 1 \\ 0 & 0 & 0 & 0 & 1 & 0 \\ 0 & 1 & 1 & 0 & 1 & 1 \\ 1 & 0 & 0 & 1 & 1 & 1 \end{smallmatrix} \right], \ 
 \left[ \begin{smallmatrix} 0 & 0 & 0 & 0 & 0 & 1 \\ 0 & 0 & 0 & 0 & 1 & 0 \\ 0 & 0 & 0 & 0 & 1 & 0 \\ 0 & 0 & 0 & 0 & 0 & 1 \\ 1 & 0 & 0 & 1 & 1 & 1 \\ 0 & 1 & 1 & 0 & 1 & 1  \end{smallmatrix} \right]} $$ 
  \begin{itemize}
\item Formal codegrees: $[(12+4\sqrt{3}, 1), (8, 2), (12-4\sqrt{3}, 1), (4, 2)]$.
\item Properties: quadratic with $G=C_4$, non-Drinfeld.
\end{itemize}

$$ \normalsize{\left[ \begin{smallmatrix}1 & 0 & 0 & 0 & 0 & 0 \\ 0 & 1 & 0 & 0 & 0 & 0 \\ 0 & 0 & 1 & 0 & 0 & 0 \\ 0 & 0 & 0 & 1 & 0 & 0 \\ 0 & 0 & 0 & 0 & 1 & 0 \\ 0 & 0 & 0 & 0 & 0 & 1 \end{smallmatrix} \right], \ 
 \left[ \begin{smallmatrix} 0 & 1 & 0 & 0 & 0 & 0 \\ 1 & 0 & 0 & 0 & 0 & 0 \\ 0 & 0 & 0 & 1 & 0 & 0 \\ 0 & 0 & 1 & 0 & 0 & 0 \\ 0 & 0 & 0 & 0 & 1 & 0 \\ 0 & 0 & 0 & 0 & 0 & 1 \end{smallmatrix} \right], \ 
 \left[ \begin{smallmatrix} 0 & 0 & 1 & 0 & 0 & 0 \\ 0 & 0 & 0 & 1 & 0 & 0 \\ 1 & 0 & 0 & 0 & 0 & 0 \\ 0 & 1 & 0 & 0 & 0 & 0 \\ 0 & 0 & 0 & 0 & 0 & 1 \\ 0 & 0 & 0 & 0 & 1 & 0 \end{smallmatrix} \right], \ 
 \left[ \begin{smallmatrix} 0 & 0 & 0 & 1 & 0 & 0 \\ 0 & 0 & 1 & 0 & 0 & 0 \\ 0 & 1 & 0 & 0 & 0 & 0 \\ 1 & 0 & 0 & 0 & 0 & 0 \\ 0 & 0 & 0 & 0 & 0 & 1 \\ 0 & 0 & 0 & 0 & 1 & 0 \end{smallmatrix} \right], \ 
 \left[ \begin{smallmatrix} 0 & 0 & 0 & 0 & 1 & 0 \\ 0 & 0 & 0 & 0 & 1 & 0 \\ 0 & 0 & 0 & 0 & 0 & 1 \\ 0 & 0 & 0 & 0 & 0 & 1 \\ 0 & 0 & 1 & 1 & 1 & 1 \\ 1 & 1 & 0 & 0 & 1 & 1 \end{smallmatrix} \right], \ 
 \left[ \begin{smallmatrix} 0 & 0 & 0 & 0 & 0 & 1 \\ 0 & 0 & 0 & 0 & 0 & 1 \\ 0 & 0 & 0 & 0 & 1 & 0 \\ 0 & 0 & 0 & 0 & 1 & 0 \\ 1 & 1 & 0 & 0 & 1 & 1 \\ 0 & 0 & 1 & 1 & 1 & 1  \end{smallmatrix} \right]} $$ 
\begin{itemize}
\item Formal codegrees: $[(12+4\sqrt{3}, 1), (8, 2), (12-4\sqrt{3}, 1), (4, 2)]$. 
\item Properties: quadratic with $G=C_2^2$, non-Drinfeld. 
\end{itemize}

$$ \normalsize{\left[ \begin{smallmatrix}1 & 0 & 0 & 0 & 0 & 0 \\ 0 & 1 & 0 & 0 & 0 & 0 \\ 0 & 0 & 1 & 0 & 0 & 0 \\ 0 & 0 & 0 & 1 & 0 & 0 \\ 0 & 0 & 0 & 0 & 1 & 0 \\ 0 & 0 & 0 & 0 & 0 & 1 \end{smallmatrix} \right], \ 
 \left[ \begin{smallmatrix} 0 & 1 & 0 & 0 & 0 & 0 \\ 0 & 0 & 0 & 1 & 0 & 0 \\ 1 & 0 & 0 & 0 & 0 & 0 \\ 0 & 0 & 1 & 0 & 0 & 0 \\ 0 & 0 & 0 & 0 & 0 & 1 \\ 0 & 0 & 0 & 0 & 1 & 0 \end{smallmatrix} \right], \ 
 \left[ \begin{smallmatrix} 0 & 0 & 1 & 0 & 0 & 0 \\ 1 & 0 & 0 & 0 & 0 & 0 \\ 0 & 0 & 0 & 1 & 0 & 0 \\ 0 & 1 & 0 & 0 & 0 & 0 \\ 0 & 0 & 0 & 0 & 0 & 1 \\ 0 & 0 & 0 & 0 & 1 & 0 \end{smallmatrix} \right], \ 
 \left[ \begin{smallmatrix} 0 & 0 & 0 & 1 & 0 & 0 \\ 0 & 0 & 1 & 0 & 0 & 0 \\ 0 & 1 & 0 & 0 & 0 & 0 \\ 1 & 0 & 0 & 0 & 0 & 0 \\ 0 & 0 & 0 & 0 & 1 & 0 \\ 0 & 0 & 0 & 0 & 0 & 1 \end{smallmatrix} \right], \ 
 \left[ \begin{smallmatrix} 0 & 0 & 0 & 0 & 1 & 0 \\ 0 & 0 & 0 & 0 & 0 & 1 \\ 0 & 0 & 0 & 0 & 0 & 1 \\ 0 & 0 & 0 & 0 & 1 & 0 \\ 1 & 0 & 0 & 1 & 1 & 1 \\ 0 & 1 & 1 & 0 & 1 & 1 \end{smallmatrix} \right], \ 
 \left[ \begin{smallmatrix} 0 & 0 & 0 & 0 & 0 & 1 \\ 0 & 0 & 0 & 0 & 1 & 0 \\ 0 & 0 & 0 & 0 & 1 & 0 \\ 0 & 0 & 0 & 0 & 0 & 1 \\ 0 & 1 & 1 & 0 & 1 & 1 \\ 1 & 0 & 0 & 1 & 1 & 1  \end{smallmatrix} \right]} $$ 
\begin{itemize}
\item Formal codegrees: $[(12+4\sqrt{3}, 1), (8, 2), (12-4\sqrt{3}, 1), (4, 2)]$.
\item Properties: quadratic with $G=C_4$, non-Drinfeld. 
\end{itemize}

$$ \normalsize{\left[ \begin{smallmatrix}1 & 0 & 0 & 0 & 0 & 0 \\ 0 & 1 & 0 & 0 & 0 & 0 \\ 0 & 0 & 1 & 0 & 0 & 0 \\ 0 & 0 & 0 & 1 & 0 & 0 \\ 0 & 0 & 0 & 0 & 1 & 0 \\ 0 & 0 & 0 & 0 & 0 & 1 \end{smallmatrix} \right], \ 
 \left[ \begin{smallmatrix} 0 & 1 & 0 & 0 & 0 & 0 \\ 1 & 0 & 0 & 0 & 0 & 0 \\ 0 & 0 & 0 & 1 & 0 & 0 \\ 0 & 0 & 1 & 0 & 0 & 0 \\ 0 & 0 & 0 & 0 & 1 & 0 \\ 0 & 0 & 0 & 0 & 0 & 1 \end{smallmatrix} \right], \ 
 \left[ \begin{smallmatrix} 0 & 0 & 1 & 0 & 0 & 0 \\ 0 & 0 & 0 & 1 & 0 & 0 \\ 1 & 0 & 0 & 0 & 0 & 0 \\ 0 & 1 & 0 & 0 & 0 & 0 \\ 0 & 0 & 0 & 0 & 0 & 1 \\ 0 & 0 & 0 & 0 & 1 & 0 \end{smallmatrix} \right], \ 
 \left[ \begin{smallmatrix} 0 & 0 & 0 & 1 & 0 & 0 \\ 0 & 0 & 1 & 0 & 0 & 0 \\ 0 & 1 & 0 & 0 & 0 & 0 \\ 1 & 0 & 0 & 0 & 0 & 0 \\ 0 & 0 & 0 & 0 & 0 & 1 \\ 0 & 0 & 0 & 0 & 1 & 0 \end{smallmatrix} \right], \ 
 \left[ \begin{smallmatrix} 0 & 0 & 0 & 0 & 1 & 0 \\ 0 & 0 & 0 & 0 & 1 & 0 \\ 0 & 0 & 0 & 0 & 0 & 1 \\ 0 & 0 & 0 & 0 & 0 & 1 \\ 1 & 1 & 0 & 0 & 1 & 1 \\ 0 & 0 & 1 & 1 & 1 & 1 \end{smallmatrix} \right], \ 
 \left[ \begin{smallmatrix} 0 & 0 & 0 & 0 & 0 & 1 \\ 0 & 0 & 0 & 0 & 0 & 1 \\ 0 & 0 & 0 & 0 & 1 & 0 \\ 0 & 0 & 0 & 0 & 1 & 0 \\ 0 & 0 & 1 & 1 & 1 & 1 \\ 1 & 1 & 0 & 0 & 1 & 1  \end{smallmatrix} \right]} $$ 
\begin{itemize}
\item Formal codegrees: $[(12+4\sqrt{3}, 1), (8, 2), (12-4\sqrt{3}, 1), (4, 2)]$.
\item Properties: quadratic with $G=C_2^2$, non-Drinfeld. 
\end{itemize}

\item $\FPdim \ 20$, type $[1, 1, 2, 2, \sqrt{5}, \sqrt{5}]$, two fusion rings (\textnumero 20,21): 

$$ \normalsize{\left[ \begin{smallmatrix}1 & 0 & 0 & 0 & 0 & 0 \\ 0 & 1 & 0 & 0 & 0 & 0 \\ 0 & 0 & 1 & 0 & 0 & 0 \\ 0 & 0 & 0 & 1 & 0 & 0 \\ 0 & 0 & 0 & 0 & 1 & 0 \\ 0 & 0 & 0 & 0 & 0 & 1 \end{smallmatrix} \right], \ 
 \left[ \begin{smallmatrix} 0 & 1 & 0 & 0 & 0 & 0 \\ 1 & 0 & 0 & 0 & 0 & 0 \\ 0 & 0 & 1 & 0 & 0 & 0 \\ 0 & 0 & 0 & 1 & 0 & 0 \\ 0 & 0 & 0 & 0 & 0 & 1 \\ 0 & 0 & 0 & 0 & 1 & 0 \end{smallmatrix} \right], \ 
 \left[ \begin{smallmatrix} 0 & 0 & 1 & 0 & 0 & 0 \\ 0 & 0 & 1 & 0 & 0 & 0 \\ 1 & 1 & 0 & 1 & 0 & 0 \\ 0 & 0 & 1 & 1 & 0 & 0 \\ 0 & 0 & 0 & 0 & 1 & 1 \\ 0 & 0 & 0 & 0 & 1 & 1 \end{smallmatrix} \right], \ 
 \left[ \begin{smallmatrix} 0 & 0 & 0 & 1 & 0 & 0 \\ 0 & 0 & 0 & 1 & 0 & 0 \\ 0 & 0 & 1 & 1 & 0 & 0 \\ 1 & 1 & 1 & 0 & 0 & 0 \\ 0 & 0 & 0 & 0 & 1 & 1 \\ 0 & 0 & 0 & 0 & 1 & 1 \end{smallmatrix} \right], \ 
 \left[ \begin{smallmatrix} 0 & 0 & 0 & 0 & 1 & 0 \\ 0 & 0 & 0 & 0 & 0 & 1 \\ 0 & 0 & 0 & 0 & 1 & 1 \\ 0 & 0 & 0 & 0 & 1 & 1 \\ 1 & 0 & 1 & 1 & 0 & 0 \\ 0 & 1 & 1 & 1 & 0 & 0 \end{smallmatrix} \right], \ 
 \left[ \begin{smallmatrix} 0 & 0 & 0 & 0 & 0 & 1 \\ 0 & 0 & 0 & 0 & 1 & 0 \\ 0 & 0 & 0 & 0 & 1 & 1 \\ 0 & 0 & 0 & 0 & 1 & 1 \\ 0 & 1 & 1 & 1 & 0 & 0 \\ 1 & 0 & 1 & 1 & 0 & 0  \end{smallmatrix} \right]} $$ 
\begin{itemize}
\item Formal codegrees: $[(20, 2), (5, 2), (4, 2)]$.
\item[$\star$] Properties: extension of $\Rep(D_5)$. 
\item Model: $SO(5)_2$.
\end{itemize}

$$ \normalsize{\left[ \begin{smallmatrix}1 & 0 & 0 & 0 & 0 & 0 \\ 0 & 1 & 0 & 0 & 0 & 0 \\ 0 & 0 & 1 & 0 & 0 & 0 \\ 0 & 0 & 0 & 1 & 0 & 0 \\ 0 & 0 & 0 & 0 & 1 & 0 \\ 0 & 0 & 0 & 0 & 0 & 1 \end{smallmatrix} \right], \ 
 \left[ \begin{smallmatrix} 0 & 1 & 0 & 0 & 0 & 0 \\ 1 & 0 & 0 & 0 & 0 & 0 \\ 0 & 0 & 1 & 0 & 0 & 0 \\ 0 & 0 & 0 & 1 & 0 & 0 \\ 0 & 0 & 0 & 0 & 0 & 1 \\ 0 & 0 & 0 & 0 & 1 & 0 \end{smallmatrix} \right], \ 
 \left[ \begin{smallmatrix} 0 & 0 & 1 & 0 & 0 & 0 \\ 0 & 0 & 1 & 0 & 0 & 0 \\ 1 & 1 & 0 & 1 & 0 & 0 \\ 0 & 0 & 1 & 1 & 0 & 0 \\ 0 & 0 & 0 & 0 & 1 & 1 \\ 0 & 0 & 0 & 0 & 1 & 1 \end{smallmatrix} \right], \ 
 \left[ \begin{smallmatrix} 0 & 0 & 0 & 1 & 0 & 0 \\ 0 & 0 & 0 & 1 & 0 & 0 \\ 0 & 0 & 1 & 1 & 0 & 0 \\ 1 & 1 & 1 & 0 & 0 & 0 \\ 0 & 0 & 0 & 0 & 1 & 1 \\ 0 & 0 & 0 & 0 & 1 & 1 \end{smallmatrix} \right], \ 
 \left[ \begin{smallmatrix} 0 & 0 & 0 & 0 & 1 & 0 \\ 0 & 0 & 0 & 0 & 0 & 1 \\ 0 & 0 & 0 & 0 & 1 & 1 \\ 0 & 0 & 0 & 0 & 1 & 1 \\ 0 & 1 & 1 & 1 & 0 & 0 \\ 1 & 0 & 1 & 1 & 0 & 0 \end{smallmatrix} \right], \ 
 \left[ \begin{smallmatrix} 0 & 0 & 0 & 0 & 0 & 1 \\ 0 & 0 & 0 & 0 & 1 & 0 \\ 0 & 0 & 0 & 0 & 1 & 1 \\ 0 & 0 & 0 & 0 & 1 & 1 \\ 1 & 0 & 1 & 1 & 0 & 0 \\ 0 & 1 & 1 & 1 & 0 & 0  \end{smallmatrix} \right]} $$ 
\begin{itemize}
\item Formal codegrees: $[(20, 2), (5, 2), (4, 2)]$.
\item[$\star$] Properties: extension of $\Rep(D_5)$, unitarily categorified (see \S \ref{r6dim20}), non-modular, weakly group-theoretical, 
\item Model: zesting of $\SO(5)_2$, see \S \ref{sub:zest}.
\end{itemize}

\item $\FPdim \ 12+6\alpha_4 \simeq 20.485$, type $[1,1,1,\alpha_4+1,\alpha_4+1,\alpha_4+1]$, one fusion ring (\textnumero 22): 
$$ \normalsize{\left[ \begin{smallmatrix}1 & 0 & 0 & 0 & 0 & 0 \\ 0 & 1 & 0 & 0 & 0 & 0 \\ 0 & 0 & 1 & 0 & 0 & 0 \\ 0 & 0 & 0 & 1 & 0 & 0 \\ 0 & 0 & 0 & 0 & 1 & 0 \\ 0 & 0 & 0 & 0 & 0 & 1 \end{smallmatrix} \right], \ 
 \left[ \begin{smallmatrix} 0 & 1 & 0 & 0 & 0 & 0 \\ 0 & 0 & 1 & 0 & 0 & 0 \\ 1 & 0 & 0 & 0 & 0 & 0 \\ 0 & 0 & 0 & 0 & 0 & 1 \\ 0 & 0 & 0 & 1 & 0 & 0 \\ 0 & 0 & 0 & 0 & 1 & 0 \end{smallmatrix} \right], \ 
 \left[ \begin{smallmatrix} 0 & 0 & 1 & 0 & 0 & 0 \\ 1 & 0 & 0 & 0 & 0 & 0 \\ 0 & 1 & 0 & 0 & 0 & 0 \\ 0 & 0 & 0 & 0 & 1 & 0 \\ 0 & 0 & 0 & 0 & 0 & 1 \\ 0 & 0 & 0 & 1 & 0 & 0 \end{smallmatrix} \right], \ 
 \left[ \begin{smallmatrix} 0 & 0 & 0 & 1 & 0 & 0 \\ 0 & 0 & 0 & 0 & 0 & 1 \\ 0 & 0 & 0 & 0 & 1 & 0 \\ 0 & 1 & 0 & 1 & 0 & 1 \\ 1 & 0 & 0 & 1 & 1 & 0 \\ 0 & 0 & 1 & 0 & 1 & 1 \end{smallmatrix} \right], \ 
 \left[ \begin{smallmatrix} 0 & 0 & 0 & 0 & 1 & 0 \\ 0 & 0 & 0 & 1 & 0 & 0 \\ 0 & 0 & 0 & 0 & 0 & 1 \\ 1 & 0 & 0 & 1 & 1 & 0 \\ 0 & 0 & 1 & 0 & 1 & 1 \\ 0 & 1 & 0 & 1 & 0 & 1 \end{smallmatrix} \right], \ 
 \left[ \begin{smallmatrix} 0 & 0 & 0 & 0 & 0 & 1 \\ 0 & 0 & 0 & 0 & 1 & 0 \\ 0 & 0 & 0 & 1 & 0 & 0 \\ 0 & 0 & 1 & 0 & 1 & 1 \\ 0 & 1 & 0 & 1 & 0 & 1 \\ 1 & 0 & 0 & 1 & 1 & 0  \end{smallmatrix} \right]} $$ 
\begin{itemize}
\item Formal codegrees: $[(12+6\sqrt{2}, 1), ((15+3\sqrt{5})/2, 2), ((15-3\sqrt{5})/2, 2), (12-6\sqrt{2}, 1)]$.
\item Properties: quadratic with $G=C_3$, non-Schur, non-Drinfeld.
\end{itemize}

\item $\FPdim \ 12+6\alpha_5 \simeq 21.708$, type $[1,1,\alpha_5,\alpha_5,2,2\alpha_5]$, one fusion ring (\textnumero 23): 
$$ \normalsize{\left[ \begin{smallmatrix}1 & 0 & 0 & 0 & 0 & 0 \\ 0 & 1 & 0 & 0 & 0 & 0 \\ 0 & 0 & 1 & 0 & 0 & 0 \\ 0 & 0 & 0 & 1 & 0 & 0 \\ 0 & 0 & 0 & 0 & 1 & 0 \\ 0 & 0 & 0 & 0 & 0 & 1 \end{smallmatrix} \right], \ 
 \left[ \begin{smallmatrix} 0 & 1 & 0 & 0 & 0 & 0 \\ 1 & 0 & 0 & 0 & 0 & 0 \\ 0 & 0 & 0 & 1 & 0 & 0 \\ 0 & 0 & 1 & 0 & 0 & 0 \\ 0 & 0 & 0 & 0 & 1 & 0 \\ 0 & 0 & 0 & 0 & 0 & 1 \end{smallmatrix} \right], \ 
 \left[ \begin{smallmatrix} 0 & 0 & 1 & 0 & 0 & 0 \\ 0 & 0 & 0 & 1 & 0 & 0 \\ 1 & 0 & 1 & 0 & 0 & 0 \\ 0 & 1 & 0 & 1 & 0 & 0 \\ 0 & 0 & 0 & 0 & 0 & 1 \\ 0 & 0 & 0 & 0 & 1 & 1 \end{smallmatrix} \right], \ 
 \left[ \begin{smallmatrix} 0 & 0 & 0 & 1 & 0 & 0 \\ 0 & 0 & 1 & 0 & 0 & 0 \\ 0 & 1 & 0 & 1 & 0 & 0 \\ 1 & 0 & 1 & 0 & 0 & 0 \\ 0 & 0 & 0 & 0 & 0 & 1 \\ 0 & 0 & 0 & 0 & 1 & 1 \end{smallmatrix} \right], \ 
 \left[ \begin{smallmatrix} 0 & 0 & 0 & 0 & 1 & 0 \\ 0 & 0 & 0 & 0 & 1 & 0 \\ 0 & 0 & 0 & 0 & 0 & 1 \\ 0 & 0 & 0 & 0 & 0 & 1 \\ 1 & 1 & 0 & 0 & 1 & 0 \\ 0 & 0 & 1 & 1 & 0 & 1 \end{smallmatrix} \right], \ 
 \left[ \begin{smallmatrix} 0 & 0 & 0 & 0 & 0 & 1 \\ 0 & 0 & 0 & 0 & 0 & 1 \\ 0 & 0 & 0 & 0 & 1 & 1 \\ 0 & 0 & 0 & 0 & 1 & 1 \\ 0 & 0 & 1 & 1 & 0 & 1 \\ 1 & 1 & 1 & 1 & 1 & 1  \end{smallmatrix} \right]} $$ 
\begin{itemize}
\item Formal codegrees: $[(15+3\sqrt{5}, 1), ((15+3\sqrt{5})/2, 1), (15-3\sqrt{5}, 1), (5+\sqrt{5}, 1), ((15-3\sqrt{5})/2, 1), (-5-\sqrt{5}, 1)]$.
\item[$\star$] Properties: extension of $\VVec(C_2)$.
\item Model: $ \PSU(2)_3 \otimes \Rep(S_3)$.
\end{itemize}

\item $\FPdim \ 21+\sqrt{21} \simeq 25.583$, type $[1,1,2,2,(1+\sqrt{21})/2,(1+\sqrt{21})/2]$, one fusion ring (\textnumero 24): 
$$ \normalsize{\left[ \begin{smallmatrix}1 & 0 & 0 & 0 & 0 & 0 \\ 0 & 1 & 0 & 0 & 0 & 0 \\ 0 & 0 & 1 & 0 & 0 & 0 \\ 0 & 0 & 0 & 1 & 0 & 0 \\ 0 & 0 & 0 & 0 & 1 & 0 \\ 0 & 0 & 0 & 0 & 0 & 1 \end{smallmatrix} \right], \ 
 \left[ \begin{smallmatrix} 0 & 1 & 0 & 0 & 0 & 0 \\ 1 & 0 & 0 & 0 & 0 & 0 \\ 0 & 0 & 1 & 0 & 0 & 0 \\ 0 & 0 & 0 & 1 & 0 & 0 \\ 0 & 0 & 0 & 0 & 0 & 1 \\ 0 & 0 & 0 & 0 & 1 & 0 \end{smallmatrix} \right], \ 
 \left[ \begin{smallmatrix} 0 & 0 & 1 & 0 & 0 & 0 \\ 0 & 0 & 1 & 0 & 0 & 0 \\ 1 & 1 & 0 & 1 & 0 & 0 \\ 0 & 0 & 1 & 1 & 0 & 0 \\ 0 & 0 & 0 & 0 & 1 & 1 \\ 0 & 0 & 0 & 0 & 1 & 1 \end{smallmatrix} \right], \ 
 \left[ \begin{smallmatrix} 0 & 0 & 0 & 1 & 0 & 0 \\ 0 & 0 & 0 & 1 & 0 & 0 \\ 0 & 0 & 1 & 1 & 0 & 0 \\ 1 & 1 & 1 & 0 & 0 & 0 \\ 0 & 0 & 0 & 0 & 1 & 1 \\ 0 & 0 & 0 & 0 & 1 & 1 \end{smallmatrix} \right], \ 
 \left[ \begin{smallmatrix} 0 & 0 & 0 & 0 & 1 & 0 \\ 0 & 0 & 0 & 0 & 0 & 1 \\ 0 & 0 & 0 & 0 & 1 & 1 \\ 0 & 0 & 0 & 0 & 1 & 1 \\ 1 & 0 & 1 & 1 & 1 & 0 \\ 0 & 1 & 1 & 1 & 0 & 1 \end{smallmatrix} \right], \ 
 \left[ \begin{smallmatrix} 0 & 0 & 0 & 0 & 0 & 1 \\ 0 & 0 & 0 & 0 & 1 & 0 \\ 0 & 0 & 0 & 0 & 1 & 1 \\ 0 & 0 & 0 & 0 & 1 & 1 \\ 0 & 1 & 1 & 1 & 0 & 1 \\ 1 & 0 & 1 & 1 & 1 & 0  \end{smallmatrix} \right]} $$ 
\begin{itemize}
\item Formal codegrees: $[(21+\sqrt{21}, 1), (21-\sqrt{21}, 1), (5+\sqrt{5}, 1), (5, 2), (5-\sqrt{5}, 1)]$.
\item Properties: extension of $\Rep(D_5)$, non-Schur, non-d-number, non-Drinfeld, non-Lagrange. 
\end{itemize}

\item $\FPdim \ 18+6\alpha_6 \simeq 28.392$, type $[1, 1, 2, 1+\alpha_6, 1+\alpha_6, 1+\alpha_6]$, two fusion rings (\textnumero 25,26): 
$$ \normalsize{\left[ \begin{smallmatrix}1 & 0 & 0 & 0 & 0 & 0 \\ 0 & 1 & 0 & 0 & 0 & 0 \\ 0 & 0 & 1 & 0 & 0 & 0 \\ 0 & 0 & 0 & 1 & 0 & 0 \\ 0 & 0 & 0 & 0 & 1 & 0 \\ 0 & 0 & 0 & 0 & 0 & 1 \end{smallmatrix} \right], \ 
 \left[ \begin{smallmatrix} 0 & 1 & 0 & 0 & 0 & 0 \\ 1 & 0 & 0 & 0 & 0 & 0 \\ 0 & 0 & 1 & 0 & 0 & 0 \\ 0 & 0 & 0 & 1 & 0 & 0 \\ 0 & 0 & 0 & 0 & 1 & 0 \\ 0 & 0 & 0 & 0 & 0 & 1 \end{smallmatrix} \right], \ 
 \left[ \begin{smallmatrix} 0 & 0 & 1 & 0 & 0 & 0 \\ 0 & 0 & 1 & 0 & 0 & 0 \\ 1 & 1 & 1 & 0 & 0 & 0 \\ 0 & 0 & 0 & 0 & 1 & 1 \\ 0 & 0 & 0 & 1 & 0 & 1 \\ 0 & 0 & 0 & 1 & 1 & 0 \end{smallmatrix} \right], \ 
 \left[ \begin{smallmatrix} 0 & 0 & 0 & 1 & 0 & 0 \\ 0 & 0 & 0 & 1 & 0 & 0 \\ 0 & 0 & 0 & 0 & 1 & 1 \\ 1 & 1 & 0 & 1 & 1 & 0 \\ 0 & 0 & 1 & 1 & 0 & 1 \\ 0 & 0 & 1 & 0 & 1 & 1 \end{smallmatrix} \right], \ 
 \left[ \begin{smallmatrix} 0 & 0 & 0 & 0 & 1 & 0 \\ 0 & 0 & 0 & 0 & 1 & 0 \\ 0 & 0 & 0 & 1 & 0 & 1 \\ 0 & 0 & 1 & 1 & 0 & 1 \\ 1 & 1 & 0 & 0 & 1 & 1 \\ 0 & 0 & 1 & 1 & 1 & 0 \end{smallmatrix} \right], \ 
 \left[ \begin{smallmatrix} 0 & 0 & 0 & 0 & 0 & 1 \\ 0 & 0 & 0 & 0 & 0 & 1 \\ 0 & 0 & 0 & 1 & 1 & 0 \\ 0 & 0 & 1 & 0 & 1 & 1 \\ 0 & 0 & 1 & 1 & 1 & 0 \\ 1 & 1 & 0 & 1 & 0 & 1  \end{smallmatrix} \right]} $$ 
\begin{itemize}
\item Formal codegrees: $[(8+6\sqrt{3}, 1), (9, 3), (18-6\sqrt{3}, 1), (2, 1)]$.
\item Properties: extension of $\Rep(S_3)$, non-Schur, non-Drinfeld, non-Czero $(3, 3, 3, 2, 5, 4, 3, 4, 3)$.
\end{itemize}

$$ \normalsize{\left[ \begin{smallmatrix}1 & 0 & 0 & 0 & 0 & 0 \\ 0 & 1 & 0 & 0 & 0 & 0 \\ 0 & 0 & 1 & 0 & 0 & 0 \\ 0 & 0 & 0 & 1 & 0 & 0 \\ 0 & 0 & 0 & 0 & 1 & 0 \\ 0 & 0 & 0 & 0 & 0 & 1 \end{smallmatrix} \right], \ 
 \left[ \begin{smallmatrix} 0 & 1 & 0 & 0 & 0 & 0 \\ 1 & 0 & 0 & 0 & 0 & 0 \\ 0 & 0 & 1 & 0 & 0 & 0 \\ 0 & 0 & 0 & 1 & 0 & 0 \\ 0 & 0 & 0 & 0 & 1 & 0 \\ 0 & 0 & 0 & 0 & 0 & 1 \end{smallmatrix} \right], \ 
 \left[ \begin{smallmatrix} 0 & 0 & 1 & 0 & 0 & 0 \\ 0 & 0 & 1 & 0 & 0 & 0 \\ 1 & 1 & 1 & 0 & 0 & 0 \\ 0 & 0 & 0 & 0 & 1 & 1 \\ 0 & 0 & 0 & 1 & 0 & 1 \\ 0 & 0 & 0 & 1 & 1 & 0 \end{smallmatrix} \right], \ 
 \left[ \begin{smallmatrix} 0 & 0 & 0 & 1 & 0 & 0 \\ 0 & 0 & 0 & 1 & 0 & 0 \\ 0 & 0 & 0 & 0 & 1 & 1 \\ 1 & 1 & 0 & 0 & 1 & 1 \\ 0 & 0 & 1 & 1 & 1 & 0 \\ 0 & 0 & 1 & 1 & 0 & 1 \end{smallmatrix} \right], \ 
 \left[ \begin{smallmatrix} 0 & 0 & 0 & 0 & 1 & 0 \\ 0 & 0 & 0 & 0 & 1 & 0 \\ 0 & 0 & 0 & 1 & 0 & 1 \\ 0 & 0 & 1 & 1 & 1 & 0 \\ 1 & 1 & 0 & 1 & 0 & 1 \\ 0 & 0 & 1 & 0 & 1 & 1 \end{smallmatrix} \right], \ 
 \left[ \begin{smallmatrix} 0 & 0 & 0 & 0 & 0 & 1 \\ 0 & 0 & 0 & 0 & 0 & 1 \\ 0 & 0 & 0 & 1 & 1 & 0 \\ 0 & 0 & 1 & 1 & 0 & 1 \\ 0 & 0 & 1 & 0 & 1 & 1 \\ 1 & 1 & 0 & 1 & 1 & 0  \end{smallmatrix} \right]} $$ 
 \begin{itemize}
\item Formal codegrees: $[(8+6\sqrt{3}, 1), (9, 3), (18-6\sqrt{3}, 1), (2, 1)]$.
\item Properties: extension of $\Rep(S_3)$, non-Schur, non-Drinfeld, non-Czero $(3, 3, 4, 2, 5, 4, 4, 4, 3)$.
\end{itemize}

\item $\FPdim \ (\alpha_5+2)(\alpha_7^4-\alpha_7^2+1) \simeq 33.633$, type $[1, \alpha_5, \alpha_7, \alpha_7^2-1, \alpha_5 \alpha_7, \alpha_5 (\alpha_7^2-1)]$, one fusion ring (\textnumero 27): 
$$ \normalsize{\left[ \begin{smallmatrix}1 & 0 & 0 & 0 & 0 & 0 \\ 0 & 1 & 0 & 0 & 0 & 0 \\ 0 & 0 & 1 & 0 & 0 & 0 \\ 0 & 0 & 0 & 1 & 0 & 0 \\ 0 & 0 & 0 & 0 & 1 & 0 \\ 0 & 0 & 0 & 0 & 0 & 1 \end{smallmatrix} \right], \ 
 \left[ \begin{smallmatrix} 0 & 1 & 0 & 0 & 0 & 0 \\ 1 & 1 & 0 & 0 & 0 & 0 \\ 0 & 0 & 0 & 0 & 1 & 0 \\ 0 & 0 & 0 & 0 & 0 & 1 \\ 0 & 0 & 1 & 0 & 1 & 0 \\ 0 & 0 & 0 & 1 & 0 & 1 \end{smallmatrix} \right], \ 
 \left[ \begin{smallmatrix} 0 & 0 & 1 & 0 & 0 & 0 \\ 0 & 0 & 0 & 0 & 1 & 0 \\ 1 & 0 & 0 & 1 & 0 & 0 \\ 0 & 0 & 1 & 1 & 0 & 0 \\ 0 & 1 & 0 & 0 & 0 & 1 \\ 0 & 0 & 0 & 0 & 1 & 1 \end{smallmatrix} \right], \ 
 \left[ \begin{smallmatrix} 0 & 0 & 0 & 1 & 0 & 0 \\ 0 & 0 & 0 & 0 & 0 & 1 \\ 0 & 0 & 1 & 1 & 0 & 0 \\ 1 & 0 & 1 & 1 & 0 & 0 \\ 0 & 0 & 0 & 0 & 1 & 1 \\ 0 & 1 & 0 & 0 & 1 & 1 \end{smallmatrix} \right], \ 
 \left[ \begin{smallmatrix} 0 & 0 & 0 & 0 & 1 & 0 \\ 0 & 0 & 1 & 0 & 1 & 0 \\ 0 & 1 & 0 & 0 & 0 & 1 \\ 0 & 0 & 0 & 0 & 1 & 1 \\ 1 & 1 & 0 & 1 & 0 & 1 \\ 0 & 0 & 1 & 1 & 1 & 1 \end{smallmatrix} \right], \ 
 \left[ \begin{smallmatrix} 0 & 0 & 0 & 0 & 0 & 1 \\ 0 & 0 & 0 & 1 & 0 & 1 \\ 0 & 0 & 0 & 0 & 1 & 1 \\ 0 & 1 & 0 & 0 & 1 & 1 \\ 0 & 0 & 1 & 1 & 1 & 1 \\ 1 & 1 & 1 & 1 & 1 & 1  \end{smallmatrix} \right]} $$ 
\begin{itemize}
\item Formal codegrees $\simeq [(33.633, 1), (12.847, 1), (10.358, 1), (6.661, 1), (3.956, 1), (2.544, 1)]$ roots of $x^6 - 70x^5 + 1715x^4 - 19600x^3 + 111475x^2 - 300125x + 300125$.
\item[$\star$] Properties: perfect.
\item Model: $\PSU(2)_3 \otimes \PSU(2)_5$.
\end{itemize}

\item $\FPdim \ 24+4\sqrt{6} \simeq 33.798$, type $[1, 1, 2, 2,1+\sqrt{6},1+\sqrt{6}]$, two fusion rings (\textnumero 28,29): 

$$ \normalsize{\left[ \begin{smallmatrix}1 & 0 & 0 & 0 & 0 & 0 \\ 0 & 1 & 0 & 0 & 0 & 0 \\ 0 & 0 & 1 & 0 & 0 & 0 \\ 0 & 0 & 0 & 1 & 0 & 0 \\ 0 & 0 & 0 & 0 & 1 & 0 \\ 0 & 0 & 0 & 0 & 0 & 1 \end{smallmatrix} \right], \ 
 \left[ \begin{smallmatrix} 0 & 1 & 0 & 0 & 0 & 0 \\ 1 & 0 & 0 & 0 & 0 & 0 \\ 0 & 0 & 1 & 0 & 0 & 0 \\ 0 & 0 & 0 & 1 & 0 & 0 \\ 0 & 0 & 0 & 0 & 0 & 1 \\ 0 & 0 & 0 & 0 & 1 & 0 \end{smallmatrix} \right], \ 
 \left[ \begin{smallmatrix} 0 & 0 & 1 & 0 & 0 & 0 \\ 0 & 0 & 1 & 0 & 0 & 0 \\ 1 & 1 & 0 & 1 & 0 & 0 \\ 0 & 0 & 1 & 1 & 0 & 0 \\ 0 & 0 & 0 & 0 & 1 & 1 \\ 0 & 0 & 0 & 0 & 1 & 1 \end{smallmatrix} \right], \ 
 \left[ \begin{smallmatrix} 0 & 0 & 0 & 1 & 0 & 0 \\ 0 & 0 & 0 & 1 & 0 & 0 \\ 0 & 0 & 1 & 1 & 0 & 0 \\ 1 & 1 & 1 & 0 & 0 & 0 \\ 0 & 0 & 0 & 0 & 1 & 1 \\ 0 & 0 & 0 & 0 & 1 & 1 \end{smallmatrix} \right], \ 
 \left[ \begin{smallmatrix} 0 & 0 & 0 & 0 & 1 & 0 \\ 0 & 0 & 0 & 0 & 0 & 1 \\ 0 & 0 & 0 & 0 & 1 & 1 \\ 0 & 0 & 0 & 0 & 1 & 1 \\ 0 & 1 & 1 & 1 & 1 & 1 \\ 1 & 0 & 1 & 1 & 1 & 1 \end{smallmatrix} \right], \ 
 \left[ \begin{smallmatrix} 0 & 0 & 0 & 0 & 0 & 1 \\ 0 & 0 & 0 & 0 & 1 & 0 \\ 0 & 0 & 0 & 0 & 1 & 1 \\ 0 & 0 & 0 & 0 & 1 & 1 \\ 1 & 0 & 1 & 1 & 1 & 1 \\ 0 & 1 & 1 & 1 & 1 & 1  \end{smallmatrix} \right]} $$ 
 \begin{itemize}
\item Formal codegrees: $[(24+4\sqrt{6}, 1), (24-4\sqrt{6}, 1), (5, 2), (4, 2)]$.
\item Properties: extension of $\Rep(D_5)$, non-Lagrange, non-d-number, non-Drinfeld. 
\end{itemize}

$$ \normalsize{\left[ \begin{smallmatrix}1 & 0 & 0 & 0 & 0 & 0 \\ 0 & 1 & 0 & 0 & 0 & 0 \\ 0 & 0 & 1 & 0 & 0 & 0 \\ 0 & 0 & 0 & 1 & 0 & 0 \\ 0 & 0 & 0 & 0 & 1 & 0 \\ 0 & 0 & 0 & 0 & 0 & 1 \end{smallmatrix} \right], \ 
 \left[ \begin{smallmatrix} 0 & 1 & 0 & 0 & 0 & 0 \\ 1 & 0 & 0 & 0 & 0 & 0 \\ 0 & 0 & 1 & 0 & 0 & 0 \\ 0 & 0 & 0 & 1 & 0 & 0 \\ 0 & 0 & 0 & 0 & 0 & 1 \\ 0 & 0 & 0 & 0 & 1 & 0 \end{smallmatrix} \right], \ 
 \left[ \begin{smallmatrix} 0 & 0 & 1 & 0 & 0 & 0 \\ 0 & 0 & 1 & 0 & 0 & 0 \\ 1 & 1 & 0 & 1 & 0 & 0 \\ 0 & 0 & 1 & 1 & 0 & 0 \\ 0 & 0 & 0 & 0 & 1 & 1 \\ 0 & 0 & 0 & 0 & 1 & 1 \end{smallmatrix} \right], \ 
 \left[ \begin{smallmatrix} 0 & 0 & 0 & 1 & 0 & 0 \\ 0 & 0 & 0 & 1 & 0 & 0 \\ 0 & 0 & 1 & 1 & 0 & 0 \\ 1 & 1 & 1 & 0 & 0 & 0 \\ 0 & 0 & 0 & 0 & 1 & 1 \\ 0 & 0 & 0 & 0 & 1 & 1 \end{smallmatrix} \right], \ 
 \left[ \begin{smallmatrix} 0 & 0 & 0 & 0 & 1 & 0 \\ 0 & 0 & 0 & 0 & 0 & 1 \\ 0 & 0 & 0 & 0 & 1 & 1 \\ 0 & 0 & 0 & 0 & 1 & 1 \\ 1 & 0 & 1 & 1 & 1 & 1 \\ 0 & 1 & 1 & 1 & 1 & 1 \end{smallmatrix} \right], \ 
 \left[ \begin{smallmatrix} 0 & 0 & 0 & 0 & 0 & 1 \\ 0 & 0 & 0 & 0 & 1 & 0 \\ 0 & 0 & 0 & 0 & 1 & 1 \\ 0 & 0 & 0 & 0 & 1 & 1 \\ 0 & 1 & 1 & 1 & 1 & 1 \\ 1 & 0 & 1 & 1 & 1 & 1  \end{smallmatrix} \right]} $$ 
\begin{itemize}
\item Formal codegrees: $[(24+4\sqrt{6}, 1), (24-4\sqrt{6}, 1), (5, 2), (4, 2)]$.
\item Properties: extension of $\Rep(D_5)$, non-Lagrange, non-d-number, non-Drinfeld. 
\end{itemize}

\item $\FPdim \ (39+9\sqrt{13})/2 \simeq 35.725$, type $[1,1,1,(3+\sqrt{13})/2,(3+\sqrt{13)})/2,(3+\sqrt{13})/2]$, two fusion rings (\textnumero 30,31): 
$$ \normalsize{\left[ \begin{smallmatrix}1 & 0 & 0 & 0 & 0 & 0 \\ 0 & 1 & 0 & 0 & 0 & 0 \\ 0 & 0 & 1 & 0 & 0 & 0 \\ 0 & 0 & 0 & 1 & 0 & 0 \\ 0 & 0 & 0 & 0 & 1 & 0 \\ 0 & 0 & 0 & 0 & 0 & 1 \end{smallmatrix} \right], \ 
 \left[ \begin{smallmatrix} 0 & 1 & 0 & 0 & 0 & 0 \\ 0 & 0 & 1 & 0 & 0 & 0 \\ 1 & 0 & 0 & 0 & 0 & 0 \\ 0 & 0 & 0 & 0 & 0 & 1 \\ 0 & 0 & 0 & 1 & 0 & 0 \\ 0 & 0 & 0 & 0 & 1 & 0 \end{smallmatrix} \right], \ 
 \left[ \begin{smallmatrix} 0 & 0 & 1 & 0 & 0 & 0 \\ 1 & 0 & 0 & 0 & 0 & 0 \\ 0 & 1 & 0 & 0 & 0 & 0 \\ 0 & 0 & 0 & 0 & 1 & 0 \\ 0 & 0 & 0 & 0 & 0 & 1 \\ 0 & 0 & 0 & 1 & 0 & 0 \end{smallmatrix} \right], \ 
 \left[ \begin{smallmatrix} 0 & 0 & 0 & 1 & 0 & 0 \\ 0 & 0 & 0 & 0 & 0 & 1 \\ 0 & 0 & 0 & 0 & 1 & 0 \\ 0 & 1 & 0 & 1 & 1 & 1 \\ 1 & 0 & 0 & 1 & 1 & 1 \\ 0 & 0 & 1 & 1 & 1 & 1 \end{smallmatrix} \right], \ 
 \left[ \begin{smallmatrix} 0 & 0 & 0 & 0 & 1 & 0 \\ 0 & 0 & 0 & 1 & 0 & 0 \\ 0 & 0 & 0 & 0 & 0 & 1 \\ 1 & 0 & 0 & 1 & 1 & 1 \\ 0 & 0 & 1 & 1 & 1 & 1 \\ 0 & 1 & 0 & 1 & 1 & 1 \end{smallmatrix} \right], \ 
 \left[ \begin{smallmatrix} 0 & 0 & 0 & 0 & 0 & 1 \\ 0 & 0 & 0 & 0 & 1 & 0 \\ 0 & 0 & 0 & 1 & 0 & 0 \\ 0 & 0 & 1 & 1 & 1 & 1 \\ 0 & 1 & 0 & 1 & 1 & 1 \\ 1 & 0 & 0 & 1 & 1 & 1  \end{smallmatrix} \right]} $$ 
\begin{itemize}
\item Formal codegrees: $[((39+9\sqrt{13})/2, 1), (6, 4), ((39-9\sqrt{13})/2, 1), 1)]$.
\item Properties: quadratic $(C_3,1,1)$, non-Drinfeld.
\end{itemize}

$$ \normalsize{\left[ \begin{smallmatrix}1 & 0 & 0 & 0 & 0 & 0 \\ 0 & 1 & 0 & 0 & 0 & 0 \\ 0 & 0 & 1 & 0 & 0 & 0 \\ 0 & 0 & 0 & 1 & 0 & 0 \\ 0 & 0 & 0 & 0 & 1 & 0 \\ 0 & 0 & 0 & 0 & 0 & 1 \end{smallmatrix} \right], \ 
 \left[ \begin{smallmatrix} 0 & 1 & 0 & 0 & 0 & 0 \\ 0 & 0 & 1 & 0 & 0 & 0 \\ 1 & 0 & 0 & 0 & 0 & 0 \\ 0 & 0 & 0 & 0 & 1 & 0 \\ 0 & 0 & 0 & 0 & 0 & 1 \\ 0 & 0 & 0 & 1 & 0 & 0 \end{smallmatrix} \right], \ 
 \left[ \begin{smallmatrix} 0 & 0 & 1 & 0 & 0 & 0 \\ 1 & 0 & 0 & 0 & 0 & 0 \\ 0 & 1 & 0 & 0 & 0 & 0 \\ 0 & 0 & 0 & 0 & 0 & 1 \\ 0 & 0 & 0 & 1 & 0 & 0 \\ 0 & 0 & 0 & 0 & 1 & 0 \end{smallmatrix} \right], \ 
 \left[ \begin{smallmatrix} 0 & 0 & 0 & 1 & 0 & 0 \\ 0 & 0 & 0 & 0 & 0 & 1 \\ 0 & 0 & 0 & 0 & 1 & 0 \\ 1 & 0 & 0 & 1 & 1 & 1 \\ 0 & 0 & 1 & 1 & 1 & 1 \\ 0 & 1 & 0 & 1 & 1 & 1 \end{smallmatrix} \right], \ 
 \left[ \begin{smallmatrix} 0 & 0 & 0 & 0 & 1 & 0 \\ 0 & 0 & 0 & 1 & 0 & 0 \\ 0 & 0 & 0 & 0 & 0 & 1 \\ 0 & 1 & 0 & 1 & 1 & 1 \\ 1 & 0 & 0 & 1 & 1 & 1 \\ 0 & 0 & 1 & 1 & 1 & 1 \end{smallmatrix} \right], \ 
 \left[ \begin{smallmatrix} 0 & 0 & 0 & 0 & 0 & 1 \\ 0 & 0 & 0 & 0 & 1 & 0 \\ 0 & 0 & 0 & 1 & 0 & 0 \\ 0 & 0 & 1 & 1 & 1 & 1 \\ 0 & 1 & 0 & 1 & 1 & 1 \\ 1 & 0 & 0 & 1 & 1 & 1  \end{smallmatrix} \right]} $$ 
\begin{itemize}
\item[$\star$] Properties: quadratic $(C_3,-1,1)$, noncommutative.
\item Model: Haagerup $H_6$.
\end{itemize}

\item $\FPdim \ \simeq 36.779$, type $\simeq [1, 1, 2.709, 2.709, 3.170, 3.170]$, two fusion rings (\textnumero 32,33): 
$$ \normalsize{\left[ \begin{smallmatrix}1 & 0 & 0 & 0 & 0 & 0 \\ 0 & 1 & 0 & 0 & 0 & 0 \\ 0 & 0 & 1 & 0 & 0 & 0 \\ 0 & 0 & 0 & 1 & 0 & 0 \\ 0 & 0 & 0 & 0 & 1 & 0 \\ 0 & 0 & 0 & 0 & 0 & 1 \end{smallmatrix} \right], \ 
 \left[ \begin{smallmatrix} 0 & 1 & 0 & 0 & 0 & 0 \\ 1 & 0 & 0 & 0 & 0 & 0 \\ 0 & 0 & 0 & 1 & 0 & 0 \\ 0 & 0 & 1 & 0 & 0 & 0 \\ 0 & 0 & 0 & 0 & 0 & 1 \\ 0 & 0 & 0 & 0 & 1 & 0 \end{smallmatrix} \right], \ 
 \left[ \begin{smallmatrix} 0 & 0 & 1 & 0 & 0 & 0 \\ 0 & 0 & 0 & 1 & 0 & 0 \\ 1 & 0 & 0 & 0 & 1 & 1 \\ 0 & 1 & 0 & 0 & 1 & 1 \\ 0 & 0 & 1 & 1 & 1 & 0 \\ 0 & 0 & 1 & 1 & 0 & 1 \end{smallmatrix} \right], \ 
 \left[ \begin{smallmatrix} 0 & 0 & 0 & 1 & 0 & 0 \\ 0 & 0 & 1 & 0 & 0 & 0 \\ 0 & 1 & 0 & 0 & 1 & 1 \\ 1 & 0 & 0 & 0 & 1 & 1 \\ 0 & 0 & 1 & 1 & 0 & 1 \\ 0 & 0 & 1 & 1 & 1 & 0 \end{smallmatrix} \right], \ 
 \left[ \begin{smallmatrix} 0 & 0 & 0 & 0 & 1 & 0 \\ 0 & 0 & 0 & 0 & 0 & 1 \\ 0 & 0 & 1 & 1 & 1 & 0 \\ 0 & 0 & 1 & 1 & 0 & 1 \\ 0 & 1 & 0 & 1 & 1 & 1 \\ 1 & 0 & 1 & 0 & 1 & 1 \end{smallmatrix} \right], \ 
 \left[ \begin{smallmatrix} 0 & 0 & 0 & 0 & 0 & 1 \\ 0 & 0 & 0 & 0 & 1 & 0 \\ 0 & 0 & 1 & 1 & 0 & 1 \\ 0 & 0 & 1 & 1 & 1 & 0 \\ 1 & 0 & 1 & 0 & 1 & 1 \\ 0 & 1 & 0 & 1 & 1 & 1  \end{smallmatrix} \right]} $$ 
\begin{itemize}
\item Formal codegrees $\simeq [(36.779, 1), (12.682, 1), (8, 2), (4, 1), (2.538, 1)]$, roots of $x^6 - 72x^5 + 1760x^4 - 19936x^3 + 112768x^2 - 303104x + 303104$.
\item Properties: non-Schur, non-cyclo, non-Czero $(2, 4, 4, 2, 2, 4, 5, 2, 2)$.
\end{itemize}

$$ \normalsize{\left[ \begin{smallmatrix}1 & 0 & 0 & 0 & 0 & 0 \\ 0 & 1 & 0 & 0 & 0 & 0 \\ 0 & 0 & 1 & 0 & 0 & 0 \\ 0 & 0 & 0 & 1 & 0 & 0 \\ 0 & 0 & 0 & 0 & 1 & 0 \\ 0 & 0 & 0 & 0 & 0 & 1 \end{smallmatrix} \right], \ 
 \left[ \begin{smallmatrix} 0 & 1 & 0 & 0 & 0 & 0 \\ 1 & 0 & 0 & 0 & 0 & 0 \\ 0 & 0 & 0 & 1 & 0 & 0 \\ 0 & 0 & 1 & 0 & 0 & 0 \\ 0 & 0 & 0 & 0 & 0 & 1 \\ 0 & 0 & 0 & 0 & 1 & 0 \end{smallmatrix} \right], \ 
 \left[ \begin{smallmatrix} 0 & 0 & 1 & 0 & 0 & 0 \\ 0 & 0 & 0 & 1 & 0 & 0 \\ 1 & 0 & 0 & 0 & 1 & 1 \\ 0 & 1 & 0 & 0 & 1 & 1 \\ 0 & 0 & 1 & 1 & 1 & 0 \\ 0 & 0 & 1 & 1 & 0 & 1 \end{smallmatrix} \right], \ 
 \left[ \begin{smallmatrix} 0 & 0 & 0 & 1 & 0 & 0 \\ 0 & 0 & 1 & 0 & 0 & 0 \\ 0 & 1 & 0 & 0 & 1 & 1 \\ 1 & 0 & 0 & 0 & 1 & 1 \\ 0 & 0 & 1 & 1 & 0 & 1 \\ 0 & 0 & 1 & 1 & 1 & 0 \end{smallmatrix} \right], \ 
 \left[ \begin{smallmatrix} 0 & 0 & 0 & 0 & 1 & 0 \\ 0 & 0 & 0 & 0 & 0 & 1 \\ 0 & 0 & 1 & 1 & 1 & 0 \\ 0 & 0 & 1 & 1 & 0 & 1 \\ 1 & 0 & 1 & 0 & 1 & 1 \\ 0 & 1 & 0 & 1 & 1 & 1 \end{smallmatrix} \right], \ 
 \left[ \begin{smallmatrix} 0 & 0 & 0 & 0 & 0 & 1 \\ 0 & 0 & 0 & 0 & 1 & 0 \\ 0 & 0 & 1 & 1 & 0 & 1 \\ 0 & 0 & 1 & 1 & 1 & 0 \\ 0 & 1 & 0 & 1 & 1 & 1 \\ 1 & 0 & 1 & 0 & 1 & 1  \end{smallmatrix} \right]} $$ 
\begin{itemize}
\item Formal codegrees $\simeq [(36.779, 1), (12.682, 1), (8, 2), (4, 1), (2.538, 1)]$, roots of $x^6 - 72x^5 + 1760x^4 - 19936x^3 + 112768x^2 - 303104x + 303104$.
\item Properties: non-Schur, non-cyclo, non-Czero $(2, 4, 4, 2, 2, 4, 5, 3, 2)$.
\end{itemize}

\item $\FPdim \ 24+12\alpha_6 \simeq 44.785$, type $[1, 1, 1+\alpha_6 , 1+\alpha_6 , 2+\alpha_6 , 2+\alpha_6]$, one fusion ring (\textnumero 34): 
$$ \normalsize{\left[ \begin{smallmatrix}1 & 0 & 0 & 0 & 0 & 0 \\ 0 & 1 & 0 & 0 & 0 & 0 \\ 0 & 0 & 1 & 0 & 0 & 0 \\ 0 & 0 & 0 & 1 & 0 & 0 \\ 0 & 0 & 0 & 0 & 1 & 0 \\ 0 & 0 & 0 & 0 & 0 & 1 \end{smallmatrix} \right], \ 
 \left[ \begin{smallmatrix} 0 & 1 & 0 & 0 & 0 & 0 \\ 1 & 0 & 0 & 0 & 0 & 0 \\ 0 & 0 & 0 & 1 & 0 & 0 \\ 0 & 0 & 1 & 0 & 0 & 0 \\ 0 & 0 & 0 & 0 & 0 & 1 \\ 0 & 0 & 0 & 0 & 1 & 0 \end{smallmatrix} \right], \ 
 \left[ \begin{smallmatrix} 0 & 0 & 1 & 0 & 0 & 0 \\ 0 & 0 & 0 & 1 & 0 & 0 \\ 1 & 0 & 1 & 0 & 1 & 0 \\ 0 & 1 & 0 & 1 & 0 & 1 \\ 0 & 0 & 1 & 0 & 1 & 1 \\ 0 & 0 & 0 & 1 & 1 & 1 \end{smallmatrix} \right], \ 
 \left[ \begin{smallmatrix} 0 & 0 & 0 & 1 & 0 & 0 \\ 0 & 0 & 1 & 0 & 0 & 0 \\ 0 & 1 & 0 & 1 & 0 & 1 \\ 1 & 0 & 1 & 0 & 1 & 0 \\ 0 & 0 & 0 & 1 & 1 & 1 \\ 0 & 0 & 1 & 0 & 1 & 1 \end{smallmatrix} \right], \ 
 \left[ \begin{smallmatrix} 0 & 0 & 0 & 0 & 1 & 0 \\ 0 & 0 & 0 & 0 & 0 & 1 \\ 0 & 0 & 1 & 0 & 1 & 1 \\ 0 & 0 & 0 & 1 & 1 & 1 \\ 1 & 0 & 1 & 1 & 1 & 1 \\ 0 & 1 & 1 & 1 & 1 & 1 \end{smallmatrix} \right], \ 
 \left[ \begin{smallmatrix} 0 & 0 & 0 & 0 & 0 & 1 \\ 0 & 0 & 0 & 0 & 1 & 0 \\ 0 & 0 & 0 & 1 & 1 & 1 \\ 0 & 0 & 1 & 0 & 1 & 1 \\ 0 & 1 & 1 & 1 & 1 & 1 \\ 1 & 0 & 1 & 1 & 1 & 1  \end{smallmatrix} \right]} $$ 
\begin{itemize}
\item Formal codegrees: $[(24+12\sqrt{3}, 1), (12, 1), (6, 2), (4, 1), (24-12\sqrt{3}, 1)]$. 
\item[$\star$] Properties: extension of $\VVec(C_2)$.
\item Model: $\PSU(2)_{10}$.
\end{itemize}

\item $\FPdim \ \simeq 55.144$, type $\simeq [1, 1, 2.935, 3.681, 3.935, 3.935]$, two fusion rings (\textnumero 35,36): 
$$ \normalsize{\left[ \begin{smallmatrix}1 & 0 & 0 & 0 & 0 & 0 \\ 0 & 1 & 0 & 0 & 0 & 0 \\ 0 & 0 & 1 & 0 & 0 & 0 \\ 0 & 0 & 0 & 1 & 0 & 0 \\ 0 & 0 & 0 & 0 & 1 & 0 \\ 0 & 0 & 0 & 0 & 0 & 1 \end{smallmatrix} \right], \ 
 \left[ \begin{smallmatrix} 0 & 1 & 0 & 0 & 0 & 0 \\ 1 & 0 & 0 & 0 & 0 & 0 \\ 0 & 0 & 1 & 0 & 0 & 0 \\ 0 & 0 & 0 & 1 & 0 & 0 \\ 0 & 0 & 0 & 0 & 0 & 1 \\ 0 & 0 & 0 & 0 & 1 & 0 \end{smallmatrix} \right], \ 
 \left[ \begin{smallmatrix} 0 & 0 & 1 & 0 & 0 & 0 \\ 0 & 0 & 1 & 0 & 0 & 0 \\ 1 & 1 & 1 & 1 & 0 & 0 \\ 0 & 0 & 1 & 0 & 1 & 1 \\ 0 & 0 & 0 & 1 & 1 & 1 \\ 0 & 0 & 0 & 1 & 1 & 1 \end{smallmatrix} \right], \ 
 \left[ \begin{smallmatrix} 0 & 0 & 0 & 1 & 0 & 0 \\ 0 & 0 & 0 & 1 & 0 & 0 \\ 0 & 0 & 1 & 0 & 1 & 1 \\ 1 & 1 & 0 & 1 & 1 & 1 \\ 0 & 0 & 1 & 1 & 1 & 1 \\ 0 & 0 & 1 & 1 & 1 & 1 \end{smallmatrix} \right], \ 
 \left[ \begin{smallmatrix} 0 & 0 & 0 & 0 & 1 & 0 \\ 0 & 0 & 0 & 0 & 0 & 1 \\ 0 & 0 & 0 & 1 & 1 & 1 \\ 0 & 0 & 1 & 1 & 1 & 1 \\ 0 & 1 & 1 & 1 & 1 & 1 \\ 1 & 0 & 1 & 1 & 1 & 1 \end{smallmatrix} \right], \ 
 \left[ \begin{smallmatrix} 0 & 0 & 0 & 0 & 0 & 1 \\ 0 & 0 & 0 & 0 & 1 & 0 \\ 0 & 0 & 0 & 1 & 1 & 1 \\ 0 & 0 & 1 & 1 & 1 & 1 \\ 1 & 0 & 1 & 1 & 1 & 1 \\ 0 & 1 & 1 & 1 & 1 & 1  \end{smallmatrix} \right]} $$ 
\begin{itemize}
\item Formal codegrees $\simeq [(55.144, 1), (8, 1), (7.313, 1), (4.543, 1), (4, 2)]$, roots of $x^6 - 83x^5 + 1839x^4 - 18312x^3 + 92848x^2 - 234496x + 234496$.
\item Properties: extension of $\VVec(C_2)$, non-Schur, non-cyclo, non-Czero $(2, 2, 3, 2, 2, 2, 3, 4, 2)$.
\end{itemize}

$$ \normalsize{\left[ \begin{smallmatrix}1 & 0 & 0 & 0 & 0 & 0 \\ 0 & 1 & 0 & 0 & 0 & 0 \\ 0 & 0 & 1 & 0 & 0 & 0 \\ 0 & 0 & 0 & 1 & 0 & 0 \\ 0 & 0 & 0 & 0 & 1 & 0 \\ 0 & 0 & 0 & 0 & 0 & 1 \end{smallmatrix} \right], \ 
 \left[ \begin{smallmatrix} 0 & 1 & 0 & 0 & 0 & 0 \\ 1 & 0 & 0 & 0 & 0 & 0 \\ 0 & 0 & 1 & 0 & 0 & 0 \\ 0 & 0 & 0 & 1 & 0 & 0 \\ 0 & 0 & 0 & 0 & 0 & 1 \\ 0 & 0 & 0 & 0 & 1 & 0 \end{smallmatrix} \right], \ 
 \left[ \begin{smallmatrix} 0 & 0 & 1 & 0 & 0 & 0 \\ 0 & 0 & 1 & 0 & 0 & 0 \\ 1 & 1 & 1 & 1 & 0 & 0 \\ 0 & 0 & 1 & 0 & 1 & 1 \\ 0 & 0 & 0 & 1 & 1 & 1 \\ 0 & 0 & 0 & 1 & 1 & 1 \end{smallmatrix} \right], \ 
 \left[ \begin{smallmatrix} 0 & 0 & 0 & 1 & 0 & 0 \\ 0 & 0 & 0 & 1 & 0 & 0 \\ 0 & 0 & 1 & 0 & 1 & 1 \\ 1 & 1 & 0 & 1 & 1 & 1 \\ 0 & 0 & 1 & 1 & 1 & 1 \\ 0 & 0 & 1 & 1 & 1 & 1 \end{smallmatrix} \right], \ 
 \left[ \begin{smallmatrix} 0 & 0 & 0 & 0 & 1 & 0 \\ 0 & 0 & 0 & 0 & 0 & 1 \\ 0 & 0 & 0 & 1 & 1 & 1 \\ 0 & 0 & 1 & 1 & 1 & 1 \\ 1 & 0 & 1 & 1 & 1 & 1 \\ 0 & 1 & 1 & 1 & 1 & 1 \end{smallmatrix} \right], \ 
 \left[ \begin{smallmatrix} 0 & 0 & 0 & 0 & 0 & 1 \\ 0 & 0 & 0 & 0 & 1 & 0 \\ 0 & 0 & 0 & 1 & 1 & 1 \\ 0 & 0 & 1 & 1 & 1 & 1 \\ 0 & 1 & 1 & 1 & 1 & 1 \\ 1 & 0 & 1 & 1 & 1 & 1  \end{smallmatrix} \right]} $$ 
\begin{itemize}
\item Formal codegrees $\simeq [(55.144, 1), (8, 1), (7.313, 1), (4.543, 1), (4, 2)]$, roots of $x^6 - 83x^5 + 1839x^4 - 18312x^3 + 92848x^2 - 234496x + 234496$.
\item Properties: extension of $\VVec(C_2)$, non-Schur, non-cyclo, non-Czero $(2, 2, 3, 2, 2, 2, 3, 4, 2)$.
\end{itemize}

\item $\FPdim \ \alpha_{13}^{10} - 7\alpha_{13}^8 + 17\alpha_{13}^6 - 16\alpha_{13}^4 + 6\alpha_{13}^2 + 3 \simeq 56.747$, type $[1,\alpha_{13},\alpha_{13}^2-1,\alpha_{13}^3-2\alpha_{13},\alpha_{13}^4-3\alpha_{13}^2+1, \alpha_{13}^5-4\alpha_{13}^3+3\alpha_{13}]$, one fusion ring (\textnumero 37): 
$$ \normalsize{\left[ \begin{smallmatrix}1 & 0 & 0 & 0 & 0 & 0 \\ 0 & 1 & 0 & 0 & 0 & 0 \\ 0 & 0 & 1 & 0 & 0 & 0 \\ 0 & 0 & 0 & 1 & 0 & 0 \\ 0 & 0 & 0 & 0 & 1 & 0 \\ 0 & 0 & 0 & 0 & 0 & 1 \end{smallmatrix} \right], \ 
 \left[ \begin{smallmatrix} 0 & 1 & 0 & 0 & 0 & 0 \\ 1 & 0 & 1 & 0 & 0 & 0 \\ 0 & 1 & 0 & 1 & 0 & 0 \\ 0 & 0 & 1 & 0 & 1 & 0 \\ 0 & 0 & 0 & 1 & 0 & 1 \\ 0 & 0 & 0 & 0 & 1 & 1 \end{smallmatrix} \right], \ 
 \left[ \begin{smallmatrix} 0 & 0 & 1 & 0 & 0 & 0 \\ 0 & 1 & 0 & 1 & 0 & 0 \\ 1 & 0 & 1 & 0 & 1 & 0 \\ 0 & 1 & 0 & 1 & 0 & 1 \\ 0 & 0 & 1 & 0 & 1 & 1 \\ 0 & 0 & 0 & 1 & 1 & 1 \end{smallmatrix} \right], \ 
 \left[ \begin{smallmatrix} 0 & 0 & 0 & 1 & 0 & 0 \\ 0 & 0 & 1 & 0 & 1 & 0 \\ 0 & 1 & 0 & 1 & 0 & 1 \\ 1 & 0 & 1 & 0 & 1 & 1 \\ 0 & 1 & 0 & 1 & 1 & 1 \\ 0 & 0 & 1 & 1 & 1 & 1 \end{smallmatrix} \right], \ 
 \left[ \begin{smallmatrix} 0 & 0 & 0 & 0 & 1 & 0 \\ 0 & 0 & 0 & 1 & 0 & 1 \\ 0 & 0 & 1 & 0 & 1 & 1 \\ 0 & 1 & 0 & 1 & 1 & 1 \\ 1 & 0 & 1 & 1 & 1 & 1 \\ 0 & 1 & 1 & 1 & 1 & 1 \end{smallmatrix} \right], \ 
 \left[ \begin{smallmatrix} 0 & 0 & 0 & 0 & 0 & 1 \\ 0 & 0 & 0 & 0 & 1 & 1 \\ 0 & 0 & 0 & 1 & 1 & 1 \\ 0 & 0 & 1 & 1 & 1 & 1 \\ 0 & 1 & 1 & 1 & 1 & 1 \\ 1 & 1 & 1 & 1 & 1 & 1  \end{smallmatrix} \right]} $$ 
\begin{itemize}
\item Formal codegrees $\simeq [(56.747, 1), (15.049, 1), (7.391, 1), (4.799, 1), (3.717, 1), (3.298, 1)]$ roots of $x^6 - 91x^5 + 2366x^4 - 26364x^3 + 142805x^2 - 371293x + 371293$.
\item[$\star$] Properties: simple.
\item Model: $\PSU(2)_{11}$.
\end{itemize}

\item $\FPdim \ (65+17\sqrt{13})/2 \simeq 63.147$, type $[1,(\sqrt{13}+3)/2,(\sqrt{13}+3)/2,(\sqrt{13}+3)/2,(\sqrt{13}+3)/2,(\sqrt{13}+5)/2]$, two fusion rings (\textnumero 38,39): 
$$ \normalsize{\left[ \begin{smallmatrix}1 & 0 & 0 & 0 & 0 & 0 \\ 0 & 1 & 0 & 0 & 0 & 0 \\ 0 & 0 & 1 & 0 & 0 & 0 \\ 0 & 0 & 0 & 1 & 0 & 0 \\ 0 & 0 & 0 & 0 & 1 & 0 \\ 0 & 0 & 0 & 0 & 0 & 1 \end{smallmatrix} \right], \ 
 \left[ \begin{smallmatrix} 0 & 1 & 0 & 0 & 0 & 0 \\ 1 & 1 & 1 & 1 & 0 & 0 \\ 0 & 1 & 1 & 0 & 0 & 1 \\ 0 & 1 & 0 & 0 & 1 & 1 \\ 0 & 0 & 0 & 1 & 1 & 1 \\ 0 & 0 & 1 & 1 & 1 & 1 \end{smallmatrix} \right], \ 
 \left[ \begin{smallmatrix} 0 & 0 & 1 & 0 & 0 & 0 \\ 0 & 1 & 1 & 0 & 0 & 1 \\ 1 & 1 & 0 & 1 & 1 & 0 \\ 0 & 0 & 1 & 1 & 0 & 1 \\ 0 & 0 & 1 & 0 & 1 & 1 \\ 0 & 1 & 0 & 1 & 1 & 1 \end{smallmatrix} \right], \ 
 \left[ \begin{smallmatrix} 0 & 0 & 0 & 1 & 0 & 0 \\ 0 & 1 & 0 & 0 & 1 & 1 \\ 0 & 0 & 1 & 1 & 0 & 1 \\ 1 & 0 & 1 & 1 & 1 & 0 \\ 0 & 1 & 0 & 1 & 0 & 1 \\ 0 & 1 & 1 & 0 & 1 & 1 \end{smallmatrix} \right], \ 
 \left[ \begin{smallmatrix} 0 & 0 & 0 & 0 & 1 & 0 \\ 0 & 0 & 0 & 1 & 1 & 1 \\ 0 & 0 & 1 & 0 & 1 & 1 \\ 0 & 1 & 0 & 1 & 0 & 1 \\ 1 & 1 & 1 & 0 & 1 & 0 \\ 0 & 1 & 1 & 1 & 0 & 1 \end{smallmatrix} \right], \ 
 \left[ \begin{smallmatrix} 0 & 0 & 0 & 0 & 0 & 1 \\ 0 & 0 & 1 & 1 & 1 & 1 \\ 0 & 1 & 0 & 1 & 1 & 1 \\ 0 & 1 & 1 & 0 & 1 & 1 \\ 0 & 1 & 1 & 1 & 0 & 1 \\ 1 & 1 & 1 & 1 & 1 & 1  \end{smallmatrix} \right]} $$ 
\begin{itemize}
\item Formal codegrees: $[((65+17\sqrt{13})/2, 1), (9, 4), ((65-17\sqrt{13})/2, 1)]$.
\item Properties: simple, non-Schur, cyclo but not of Frobenius type, non-d-number, non-Drinfeld, non-Czero $(1, 1, 3, 1, 1, 1, 3, 4, 1)$.
\end{itemize}

$$ \normalsize{\left[ \begin{smallmatrix}1 & 0 & 0 & 0 & 0 & 0 \\ 0 & 1 & 0 & 0 & 0 & 0 \\ 0 & 0 & 1 & 0 & 0 & 0 \\ 0 & 0 & 0 & 1 & 0 & 0 \\ 0 & 0 & 0 & 0 & 1 & 0 \\ 0 & 0 & 0 & 0 & 0 & 1 \end{smallmatrix} \right], \ 
 \left[ \begin{smallmatrix} 0 & 1 & 0 & 0 & 0 & 0 \\ 1 & 0 & 1 & 1 & 1 & 0 \\ 0 & 1 & 1 & 0 & 0 & 1 \\ 0 & 1 & 0 & 1 & 0 & 1 \\ 0 & 1 & 0 & 0 & 1 & 1 \\ 0 & 0 & 1 & 1 & 1 & 1 \end{smallmatrix} \right], \ 
 \left[ \begin{smallmatrix} 0 & 0 & 1 & 0 & 0 & 0 \\ 0 & 1 & 1 & 0 & 0 & 1 \\ 1 & 1 & 0 & 1 & 1 & 0 \\ 0 & 0 & 1 & 1 & 0 & 1 \\ 0 & 0 & 1 & 0 & 1 & 1 \\ 0 & 1 & 0 & 1 & 1 & 1 \end{smallmatrix} \right], \ 
 \left[ \begin{smallmatrix} 0 & 0 & 0 & 1 & 0 & 0 \\ 0 & 1 & 0 & 1 & 0 & 1 \\ 0 & 0 & 1 & 1 & 0 & 1 \\ 1 & 1 & 1 & 0 & 1 & 0 \\ 0 & 0 & 0 & 1 & 1 & 1 \\ 0 & 1 & 1 & 0 & 1 & 1 \end{smallmatrix} \right], \ 
 \left[ \begin{smallmatrix} 0 & 0 & 0 & 0 & 1 & 0 \\ 0 & 1 & 0 & 0 & 1 & 1 \\ 0 & 0 & 1 & 0 & 1 & 1 \\ 0 & 0 & 0 & 1 & 1 & 1 \\ 1 & 1 & 1 & 1 & 0 & 0 \\ 0 & 1 & 1 & 1 & 0 & 1 \end{smallmatrix} \right], \ 
 \left[ \begin{smallmatrix} 0 & 0 & 0 & 0 & 0 & 1 \\ 0 & 0 & 1 & 1 & 1 & 1 \\ 0 & 1 & 0 & 1 & 1 & 1 \\ 0 & 1 & 1 & 0 & 1 & 1 \\ 0 & 1 & 1 & 1 & 0 & 1 \\ 1 & 1 & 1 & 1 & 1 & 1  \end{smallmatrix} \right]} $$ 
\begin{itemize}
\item Formal codegrees: $[((65+17\sqrt{13})/2, 1), (9, 4), ((65-17\sqrt{13})/2, 1)]$.
\item Properties: simple, non-Schur, cyclo but not of Frobenius type, non-d-number, non-Drinfeld.
\end{itemize}
\end{itemize}

\subsection{Proof of Theorem \ref{thm:main}} \label{sub:proof}

A fusion ring is given by its fusion coefficients $N_{i,j}^k \in \mathbb{Z}_{\ge 0}$, so at rank $r$ and multiplicity $m$, there are $r^3$ variables and so $(m+1)^{r^3}$ possibilities to check (e.g. $2^{216}$ ones if $r=6$ and $m=1$). Fortunately, we can drastically reduce the number of variables by using the axioms of fusion rings reformulated in term of $N_{i,j}^k$:

\begin{itemize}
\item Associativity: $\sum_s N_{ij}^sN_{sk}^t = \sum_s N_{jk}^sN_{is}^t$,
\item Neutral: $N_{1i}^j = N_{i1}^j = \delta_{ij}$,
\item Dual: $N_{i^*,k}^{1} = N_{k,i^*}^{1} = \delta_{i,k}$,
\item Frobenius reciprocity: $N_{ij}^k = N_{i^*k}^j = N_{kj^*}^i$
\end{itemize}
where $i \mapsto i^*$ is a dual structure fixed beforehand (there are few ones up to equivalence). Then we can get the $72$ fusion rings listed above in a reasonable time. The SageMath code (too long for this paper) is available in \cite{FusionAtlas}, together with the data. Now, let us provide the list of fusion rings which do not pass a given criterion:

$$\begin{array}{c|c|c|c|c}
                     & \text{criterion type}& \text{\textnumero~at  rank } 4 & \text{\textnumero~at  rank } 5   & \text{\textnumero~at  rank } 6 \\ \hline
\text{non-Schur}     & \text{unitary} &  &   9,13 &   22,24,25,26,32,33,35,36,38,39  \\ \hline
\text{non-Drinfeld}  & \text{complex pivotal} &  5  &   4,5,9,13  &   11,16,17,18,19,22,24,25,26,28,29,30,38,39  \\ \hline
\text{non-d-number}  & \text{complex} &   5  &   4,5,9,13  &    11,24,28,29,38,39 \\ \hline
\text{non-cyclo} 	 & \text{complex}   &  &   14,15 &   32,33,35,36  \\ \hline
\text{non-Lagrange}  & \text{general} & 5 &   4,5,9  &   11,24,28,29  \\ \hline
\text{non-Czero}     & \text{general} &   &   9,13  &   25,26,32,33,35,36,38
\end{array}$$

Observe that all the excluded fusion rings are non-Drinfeld, non-d-number or non-cyclo, which means that the corresponding three criteria are sufficient for proving Theorem \ref{thm:main} (see their SageMath codes in Section \ref{sec:code}). We still mentioned the three other criteria because they should be useful in the future (for example, see Subsection \ref{sub:pos}). The \textnumero $16,17,18,19,22,30$ at rank $6$ are only excluded from pivotal complex categorification (non-Drinfeld) but they are all quadratic, so excluded from (general) complex categorification by Theorem \ref{thm:quadratic}. The non-excluded fusion rings are exactly the Grothendieck rings listed in Theorem \ref{thm:main}, the new ones being categorified in Section \ref{sec:tpe}. \qed

\section{Unitary categorification of the new complex Grothendieck rings}  \label{sec:tpe}

In this section, we consider the new complex Grothendieck rings of multiplicity one up to rank $6$, and provide some unitary solutions of their Pentagon Equation (PE).

\subsection{The Pentagon Equation in multiplicity one}
In this subsection, every hom-space $hom_{\mathcal{C}}(X_i \otimes X_j, X_k)$ is assumed to be of dimension $N_{i,j}^k \le 1$, so that every morphism in it is completely determined by $i,j$ and $k$, up to a multiplicative constant, which makes the PE much easier to deal with.
%
\begin{figure}
\scalebox{.7}{
\begin{tikzpicture}
\begin{scope}[scale=1.3]
\begin{scope}
\node at (0-.2,0) {$i_5$};
\node at (1-.2,0) {$i_4$};
\node at (2-.2,0) {$i_7$};
\node at (3-.2,0) {$i_9$};
\node at (1.75-.2,-.75) {$\spec$};
\node at (1.75+.2,-1.25) {$i_6$};
\node at (1.5-.2,-1.75) {$i_3$};
\node at (1.5,-.5+.2) {};
\node at (2,-1+.2) {};
\node at (1.5,-1.5+.2) {};
\draw (0,0) --++ (1.5,-1.5)--++(0,-.5);
\draw (1,0) --++ (1,-1);
\draw (2,0) --++ (-.5,-.5);
\draw (3,0) --++ (-1.5,-1.5);
\end{scope}
\begin{scope}[shift={(5,0)}]
\node at (0-.2,0) {$i_5$};
\node at (1-.2,0) {$i_4$};
\node at (2-.2,0) {$i_7$};
\node at (3-.2,0) {$i_9$};
\node at (2.25+.2,-.75) {$i_1$};
\node at (1.75+.2,-1.25) {$i_6$};
\node at (1.5-.2,-1.75) {$i_3$};
\node at (2.5,-.5+.2) {};
\node at (2,-1+.2) {};
\node at (1.5,-1.5+.2) {};
\draw (0,0) --++ (1.5,-1.5)--++(0,-.5);
\draw (1,0) --++ (1,-1);
\draw (2,0) --++ (.5,-.5);
\draw (3,0) --++ (-1.5,-1.5);
\end{scope}
\begin{scope}[shift={(-2.5,-3)}]
\node at (0-.2,0) {$i_5$};
\node at (1-.2,0) {$i_4$};
\node at (2-.2,0) {$i_7$};
\node at (3-.2,0) {$i_9$};
\node at (1.25+.2,-.75) {$\spec$};
\node at (1.25-.2,-1.25) {$i_8$};
\node at (1.5-.2,-1.75) {$i_3$};
\node at (1.5,-.5+.2) {};
\node at (1,-1+.2) {};
\node at (1.5,-1.5+.2) {};
\draw (0,0) --++ (1.5,-1.5)--++(0,-.5);
\draw (1,0) --++ (.5,-.5);
\draw (2,0) --++ (-1,-1);
\draw (3,0) --++ (-1.5,-1.5);
\end{scope}
\begin{scope}[shift={(7.5,-3)}]
\node at (0-.2,0) {$i_5$};
\node at (1-.2,0) {$i_4$};
\node at (2-.2,0) {$i_7$};
\node at (3-.2,0) {$i_9$};
\node at (.75-.2,-.75) {$i_2$};
\node at (2.25-.2,-.75) {$i_1$};
\node at (1.5-.2,-1.75) {$i_3$};
\node at (.5,-.5+.2) {};
\node at (2.5,-.5+.2) {};
\node at (1.5,-1.5+.2) {};
\draw (0,0) --++ (1.5,-1.5)--++(0,-.5);
\draw (1,0) --++ (-.5,-.5);
\draw (2,0) --++ (.5,-.5);
\draw (3,0) --++ (-1.5,-1.5);
\end{scope}
\begin{scope}[shift={(2.5,-5)}]
\node at (0-.2,0) {$i_5$};
\node at (1-.2,0) {$i_4$};
\node at (2-.2,0) {$i_7$};
\node at (3-.2,0) {$i_9$};
\node at (.75-.2,-.75) {$i_2$};
\node at (1.25-.2,-1.25) {$i_8$};
\node at (1.5-.2,-1.75) {$i_3$};
\node at (.5,-.5+.2) {};
\node at (1,-1+.2) {};
\node at (1.5,-1.5+.2) {};
\draw (0,0) --++ (1.5,-1.5)--++(0,-.5);
\draw (1,0) --++ (-.5,-.5);
\draw (2,0) --++ (-1,-1);
\draw (3,0) --++ (-1.5,-1.5);
\end{scope}
\draw [red, thick,->] (2.4,-6.5)--(-1,-5.2);
\draw [red, thick,->] (-1,-2.5)--(1,-1.5);
\draw [red, thick,->] (3.2,-.75)--(5,-.75);
\draw [blue, thick,->] (9,-2.5)--(7,-1.5);
\draw [blue, thick,->] (5.6,-6.5)--(9,-5.2);
\end{scope}
\end{tikzpicture}}
\caption{Pentagon Equation in multiplicity one}
\label{Fig: Pentagon Equation}
\end{figure}
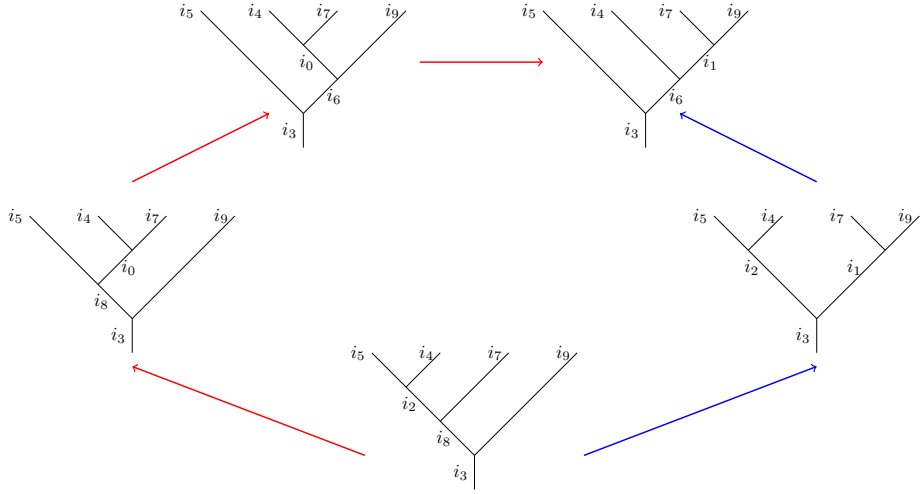
The F-symbols are defined as follows:
\begin{equation} \label{Equ: F-symbols}
\scalebox{.7}{
\raisebox{-1.25cm}{
\begin{tikzpicture}
\begin{scope}[scale=1.25]
\node at (0-.2,0) {$i_1$};
\node at (1-.2,0) {$i_2$};
\node at (2-.2,0) {$i_4$};
\node at (.75-.2,-.75) {$i_3$};
\node at (1-.2,-1.5) {$i_5$};
\draw (0,0)--++(1,-1);
\draw (1,0)--++(-.5,-.5);
\draw (2,0)--++(-1,-1)--++(0,-.5);
\end{scope}
\end{tikzpicture}}}
=
\sum_{i_6}
\left(
\begin{array}{ccc}
i_1& i_2 & i_3\\
i_4& i_5 & i_6 
\end{array}
\right)_F
%
\scalebox{.7}{
\raisebox{-1.25cm}{
\begin{tikzpicture}
\begin{scope}[scale=1.25]
\node at (0-.2,0) {$i_1$};
\node at (1-.2,0) {$i_2$};
\node at (2-.2,0) {$i_4$};
\node at (1.25+.2,-.75) {$i_6$};
\node at (1-.2,-1.5) {$i_5$};
\draw (0,0)--++(1,-1);
\draw (1,0)--++(.5,-.5);
\draw (2,0)--++(-1,-1)--++(0,-.5);
\end{scope}
\end{tikzpicture}}}.
\end{equation}
They satisfy the PE, Figure~\ref{Fig: Pentagon Equation} being a pictorial representation \cite{wang}, with the following algebraic reformulation:
\begin{equation}\label{Equ: Pentagon Equation}
\left(
\begin{array}{ccc}
i_2& i_7 & i_8  \\
i_9& i_3 & i_1 
\end{array}
\right)_F
\left(
\begin{array}{ccc}
i_5& i_4 & i_2 \\
i_1& i_3 & i_6
\end{array}
\right)_F
= 
\sum_{\spec}
\left(
\begin{array}{ccc }
i_5& i_4 & i_2  \\
i_7& i_8 & \spec 
\end{array}
\right)_F
\left(
\begin{array}{ccc }
i_5& \spec & i_8 \\
i_9& i_3 & i_6 
\end{array}
\right)_F
\left(
\begin{array}{ccc}
i_4& i_7 & \spec \\
i_9& i_6 & i_1 
\end{array}
\right)_F
\end{equation}
Let $d_i := \dim_{\mathcal{C}}(X_i)$. By using the following notation: 
\begin{equation*} 
\left(
\begin{array}{ccc}
i_1& i_2 & i_3  \\
i_4& i_5 & i_6 
\end{array}
\right)
:=d_{i_6}^{-1}
\left(
\begin{array}{ccc}
i_1& i_2 & i_3  \\
i_4& i_5 & i_6 
\end{array}
\right)_F
\end{equation*}
the PE becomes:
\begin{equation}
\left(
\begin{array}{ccc}
i_2& i_7 & i_8  \\
i_9& i_3 & i_1 
\end{array}
\right)
\left(
\begin{array}{ccc}
i_5& i_4 & i_2 \\
i_1& i_3 & i_6
\end{array}
\right)
=
\sum_{\spec} d_{\spec}
\left(
\begin{array}{ccc }
i_5& i_4 & i_2  \\
i_7& i_8 & \spec 
\end{array}
\right)
\left(
\begin{array}{ccc }
i_5& \spec & i_8 \\
i_9& i_3 & i_6 
\end{array}
\right)
\left(
\begin{array}{ccc}
i_4& i_7 & \spec \\
i_9& i_6 & i_1 
\end{array}
\right).
\end{equation}

\begin{proposition}[Pivotal axioms] \label{prop:pivaxiom}
Let $\mathcal{C}$ be a fusion category with above PE. If there are roots of unity $(t_i)$ such that: 
\begin{itemize}
\item $t_1 = 1$,
\item $t_{i^*} = t_i^{-1}$,
\item $t_i^{-1}t_j^{-1}t_k = d_{i^*} d_{j^*} d_{k}
\left(
\begin{array}{ccc }
i & j & k  \\
k^*& 1 & i^* 
\end{array}
\right) 
\left(
\begin{array}{ccc }
j& k^* & i^*  \\
i& 1 & j^* 
\end{array}
\right)
\left(
\begin{array}{ccc }
k^* & i & j^*  \\
j& 1 & k 
\end{array}
\right)$, 
$\forall i,j,k$ with $N_{i,j}^k \neq 0$, 
\end{itemize}
then $\mathcal{C}$ is pivotal, and $(t_i)$ are called the \emph{pivotal coefficients}. If moreover all $t_i = \pm 1$ then $\mathcal{C}$ is spherical.
\end{proposition}
\begin{proof}
It is a reformulation of \cite[Proposition 4.16 (1)(2)]{wang}, where
$ F_{d;n,m}^{a,b,c}  = \left(
\begin{array}{ccc }
a& b & m  \\
c& d & n 
\end{array}
\right)_F   = d_n \left(
\begin{array}{ccc }
a& b & m  \\
c& d & n 
\end{array}
\right).$
\end{proof}
\begin{corollary} \label{cor:assump}
A solution of the PE under below (a),(b),(c) gives a pseudo-unitary categorification of the fusion ring.  
\begin{itemize}
\item[(a)] $d_i=\FPdim(X_i),$ 
\item[(b)] evaluation when one object is trivial:
\[
\left(
\begin{array}{ccc}
i_1 & i_2 & i_3  \\
1 & i_3 & i_2 
\end{array}
\right)
=d_{i_2}^{-1/2}d_{i_3}^{-1/2} \;,
\]
\item[(c)] $A_4$-symmetry:
\[
\left(
\begin{array}{ccc}
i_1 & i_2 & i_3  \\
i_4 & i_5 & i_6 
\end{array}
\right)=
\left(
\begin{array}{ccc}
i_2& i_3^* & i_1^*  \\
i_5& i_6 & i_4 
\end{array}
\right)
=
\left(
\begin{array}{ccc}
i_2 & i_6 & i_4  \\
i_5^* & i_3^*& i_1^* 
\end{array}
\right) \;.
\]
\end{itemize}

The point (c) is realized as the (orientation-preserving) symmetry of the following tetrahedron, assuming the 2nd and 3rd Frobenius-Schur indicators to be trivial.
$$
\left(
\begin{array}{ccc}
i_1& i_2 & i_3  \\
i_4& i_5 & i_6 
\end{array}
\right)=
\tetra{i_1}{i_2}{i_3}{i_4}{i_5}{i_6}
$$

Furthermore, if the reflection of the tetrahedron is its complex-conjugate, in other words,
$$
\left(
\begin{array}{ccc}
i_1^* & i_3 & i_2  \\
i_4 & i_6 & i_5 
\end{array}
\right)=
\overline{
\left(
\begin{array}{ccc}
i_1 & i_2 & i_3  \\
i_4 & i_5 & i_6 
\end{array}
\right)},
$$
then the categorification is unitary.
\end{corollary}
\begin{proof}
Suppose we have a solution of the PE, then we obtain a fusion category, see \cite{EGNO15} or \cite{DaHaWa}. Now by assumption (a), $d_{i^*} = \FPdim(X_{i^*}) = \FPdim(X_{i}) = d_{i}$. So by assumption (c) and then (b), we get that 
\begin{align*}
d_{i^*} d_{j^*} d_{k}
\left(
\begin{array}{ccc }
i & j & k  \\
k^*& 1 & i^* 
\end{array}
\right) 
\left(
\begin{array}{ccc }
j& k^* & i^*  \\
i& 1 & j^* 
\end{array}
\right)
\left(
\begin{array}{ccc }
k^* & i & j^*  \\
j& 1 & k 
\end{array}
\right) =&
d_{i} d_{j} d_{k}
\left(
\begin{array}{ccc }
j& k^* & i^*  \\
1& i^* & k^* 
\end{array}
\right) 
\left(
\begin{array}{ccc }
k^*& i & j^*  \\
1& j^* & i 
\end{array}
\right)
\left(
\begin{array}{ccc }
i & j^* & k  \\
1& k & j 
\end{array}
\right) \\ 
 =& d_{i} d_{j} d_{k} [d_{k}d_{i}]^{-1/2}[d_{i}d_{j}]^{-1/2}[d_{j}d_{k}]^{-1/2}
 = 1.
 \end{align*}
Then by Proposition \ref{prop:pivaxiom}, it is pivotal by taking $t_i=1$ for all $i$, and so spherical. By assumption (a), it is pseudo-unitary. We refer to \cite{LPR1} for the details about the tetrahedral realization. If the tetrahedron has reflection symmetry, then we obtain an involution of the spherical category mapping the generating morphisms to their duals. The unitary condition follows from the fact that $\FPdim(X_i)>0$.
\end{proof}

%
%
%

Modulo its $A_4$-symmetry and when all labels and morphisms are non-zero, each tetrahedron is considered as  a \emph{complex variable} of the PE, which becomes:
\begin{equation}\label{Equ: Pentagon Equation Tetra}
\tetra{i_2}{i_7}{i_8}{i_9}{i_3}{i_1} \tetra{i_5}{i_4}{i_2}{i_1}{i_3}{i_6} = \sum_{\spec} d_{\spec} \tetra{i_5}{i_4}{i_2}{i_7}{i_8}{\spec} \tetra{i_5}{\spec}{i_8}{i_9}{i_3}{i_6} \tetra{i_4}{i_7}{\spec}{i_9}{i_6}{i_1}
\end{equation}

\begin{remark} \label{spectrum} In \cite{LPR1} the set of $\spec$ for which every hom-space in RHS of (\ref{Equ: Pentagon Equation Tetra}) with $\spec$ is not zero-dimensional, is called the spectrum $\sigma$ of the equation. The name of Theorem \ref{thm:zero} (zero spectrum criterion) means that it corresponds to an equation with $|\sigma| = 0$. That paper contains also a criterion for when $|\sigma| = 1$ denoted \emph{one spectrum criterion}.
\end{remark}


By Corollary \ref{cor:assump}, a solution of the PE, assuming (a), (b), (c) and the equality between reflection and complex-conjugate of tetrahedra, gives a unitary categorification. By \cite{LPR1}, LHS and RHS of (\ref{Equ: Pentagon Equation Tetra}) can be visualized as two ways to evaluate the following triangular prism.
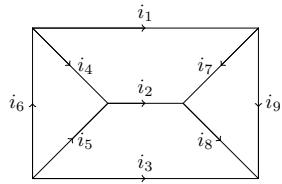
\begin{equation*}
\scalebox{.77}{
\begin{tikzpicture}[scale=1.3]
\draw (0,2)--(3,2);
\draw (1,1)--(2,1);
\draw (0,0)--(3,0);
\draw (0,0) -- (0,2) -- (1,1) -- (0,0);
\draw (3,0)--(3,2)--(2,1)--(3,0);
\draw[->] (0,2) -- (1.5,2) node [above] {$i_1$};
\draw[->] (1,1) -- (1.5,1) node [above] {$i_2$};
\draw[->] (0,0) -- (1.5,0) node [above] {$i_3$};
\draw[->] (0,2) --++ (.5,-.5) node [right] {$i_4$}; 
\draw[-<] (1,1) --++ (-.5,-.5) node [right] {$i_5$}; 
\draw[->] (0,0) -- (0,1) node [left] {$i_6$}; 
\draw[->] (3,2) --++ (-.5,-.5) node [left] {$i_7$}; 
\draw[->] (2,1) --++ (.5,-.5) node [left] {$i_8$}; 
\draw[-<] (3,0) -- (3,1) node [right] {$i_9$};
\end{tikzpicture}}
\end{equation*}
The action of $D_6$ (the symmetry group of the triangular prism) provides equivalent equations, so we only need to consider the equations modulo these symmetries. The variables are invariant, up to complex-conjugate, by the action of $S_4$ (the symmetry group of the tetrahedron). A variable invariant by reflection will be called a \emph{real variable}. 

In the next subsections, we solve (with SageMath) by ordering the equations according to the number of variables (called \emph{localization} here). Now such a resolution works much better if there is no use of complex-conjugate, so we choose to split each non-real variable into two (itself and its reflection, i.e. to consider the variables only modulo the orientation-preserving symmetry group $A_4$), but then only the unitary solutions (i.e. when these two variables are complex-conjugate) provide (proved) categorifications. To save space and time, we use the one-line notation:
\begin{equation*} 
\left(
\begin{array}{ccc}
i_1& i_2 & i_3  \\
i_4& i_5 & i_6 
\end{array}
\right)
=[i_2^*, i_3^*, i_1, i_6^*, i_4^*, i_5^*].
\end{equation*}

\subsection{Rank 5, FPdim 12 and type $[1,1,\sqrt{3},\sqrt{3},2]$} \label{r5dim12}
Let consider the fusion matrices: 
$$ \normalsize{
\left[ \begin{smallmatrix}1 & 0 & 0 & 0 & 0 \\ 0 & 1 & 0 & 0 & 0 \\ 0 & 0 & 1 & 0 & 0 \\ 0 & 0 & 0 & 1 & 0 \\ 0 & 0 & 0 & 0 & 1 \end{smallmatrix} \right], \ 
 \left[ \begin{smallmatrix} 0 & 1 & 0 & 0 & 0 \\ 1 & 0 & 0 & 0 & 0 \\ 0 & 0 & 0 & 1 & 0 \\ 0 & 0 & 1 & 0 & 0 \\ 0 & 0 & 0 & 0 & 1 \end{smallmatrix} \right], \ 
 \left[ \begin{smallmatrix} 0 & 0 & 1 & 0 & 0 \\ 0 & 0 & 0 & 1 & 0 \\ 0 & 1 & 0 & 0 & 1 \\ 1 & 0 & 0 & 0 & 1 \\ 0 & 0 & 1 & 1 & 0 \end{smallmatrix} \right], \ 
 \left[ \begin{smallmatrix} 0 & 0 & 0 & 1 & 0 \\ 0 & 0 & 1 & 0 & 0 \\ 1 & 0 & 0 & 0 & 1 \\ 0 & 1 & 0 & 0 & 1 \\ 0 & 0 & 1 & 1 & 0 \end{smallmatrix} \right], \ 
 \left[ \begin{smallmatrix} 0 & 0 & 0 & 0 & 1 \\ 0 & 0 & 0 & 0 & 1 \\ 0 & 0 & 1 & 1 & 0 \\ 0 & 0 & 1 & 1 & 0 \\ 1 & 1 & 0 & 0 & 1  \end{smallmatrix} \right]} $$

There are $12$ real variables: $[1, 2, 2, 2, 4, 3]$, $[1, 2, 2, 3, 1, 2]$, $[1, 4, 4, 4, 1, 4]$, $[1, 4, 4, 4, 4, 4]$, $[2, 2, 4, 2, 4, 4]$, $[2, 2, 4, 3, 4, 4]$, $[2, 2, 4, 4, 2, 3]$, $[2, 2, 4, 4, 3, 2]$, $[2, 3, 4, 3, 4, 4]$, $[2, 3, 4, 4, 2, 2]$, $[2, 4, 3, 4, 4, 2]$, $[4, 4, 4, 4, 4, 4]$; $4$ complex (non-real) variables: $[1, 2, 2, 3, 4, 2]$, $[1, 2, 2, 4, 2, 4]$, $[1, 2, 2, 4, 3, 4]$, $[2, 2, 4, 4, 2, 2]$, and then their reflection. Then we consider $20$ variables, ordered as above. Here a some solutions of the PE (where $I$ is the imaginary unit):
$$(\frac{\sqrt{3}}{3}, \frac{\sqrt{3}}{3}, \frac{1}{2}, -\frac{1}{2}, \frac{3^{3/4}}{6}, \frac{3^{3/4}}{6}, -\frac{\sqrt{3}}{6}, -\frac{\sqrt{3}}{6}, -\frac{3^{3/4}}{6}, -\frac{\sqrt{3}}{6}, -\frac{3^{3/4}}{6}, 0, -\frac{1}{6r^2}, \frac{\sqrt{3}}{6r}, -\frac{\sqrt{3}}{6r}, -\frac{I}{2}, -2r^2, r, -r, \frac{I}{2}),$$ with 8 variations by (pointwise) multiplying by $(s_1,1,1,1,s_2,s_1s_2,1,-s_1,-s_1s_2,1,-s_2,1,1,1,-s_1,s_3,1,1,-s_1,-s_3)$ with $s_i \in \{-1,1\}$. The unitary case corresponds to $|r|^4=1/12$.   

\begin{remark}[Resolution mode] We listed all the possible equations, up to symmetry, by a straightforward code. Note that there is no equation containing all above variables together, the maximum number of variables in a single equation is $7$ here (where a complex variable and its reflection count for one). There are exactly $59$, $53$, $39$, $64$, $44$, $13$, $4$ (non-trivial) such equations with $1, 2, \dots, 7$ variables respectively.  We first solved the equations with less than $4$ variables, which provided $16$ variations, but only $8$ variations survived after checking the rest of the equations.
\end{remark}

\begin{remark}[Variation]
In this paper, a solution is called a \emph{variation} of an other one, if they are equal up to (pointwise) signs. We did not check whether they define the same fusion category (gauge equivalent).
\end{remark}


\subsection{Rank 6, FPdim 8 and type $[1,1,1,1,\sqrt{2},\sqrt{2}]$} \label{r6dim8}
Consider the three first such fusion rings, in the same order than in \S \ref{thm:rank6}.

$$ \normalsize{\left[ \begin{smallmatrix}1 & 0 & 0 & 0 & 0 & 0 \\ 0 & 1 & 0 & 0 & 0 & 0 \\ 0 & 0 & 1 & 0 & 0 & 0 \\ 0 & 0 & 0 & 1 & 0 & 0 \\ 0 & 0 & 0 & 0 & 1 & 0 \\ 0 & 0 & 0 & 0 & 0 & 1 \end{smallmatrix} \right], \ 
 \left[ \begin{smallmatrix} 0 & 1 & 0 & 0 & 0 & 0 \\ 0 & 0 & 0 & 1 & 0 & 0 \\ 1 & 0 & 0 & 0 & 0 & 0 \\ 0 & 0 & 1 & 0 & 0 & 0 \\ 0 & 0 & 0 & 0 & 0 & 1 \\ 0 & 0 & 0 & 0 & 1 & 0 \end{smallmatrix} \right], \ 
 \left[ \begin{smallmatrix} 0 & 0 & 1 & 0 & 0 & 0 \\ 1 & 0 & 0 & 0 & 0 & 0 \\ 0 & 0 & 0 & 1 & 0 & 0 \\ 0 & 1 & 0 & 0 & 0 & 0 \\ 0 & 0 & 0 & 0 & 0 & 1 \\ 0 & 0 & 0 & 0 & 1 & 0 \end{smallmatrix} \right], \ 
 \left[ \begin{smallmatrix} 0 & 0 & 0 & 1 & 0 & 0 \\ 0 & 0 & 1 & 0 & 0 & 0 \\ 0 & 1 & 0 & 0 & 0 & 0 \\ 1 & 0 & 0 & 0 & 0 & 0 \\ 0 & 0 & 0 & 0 & 1 & 0 \\ 0 & 0 & 0 & 0 & 0 & 1 \end{smallmatrix} \right], \ 
 \left[ \begin{smallmatrix} 0 & 0 & 0 & 0 & 1 & 0 \\ 0 & 0 & 0 & 0 & 0 & 1 \\ 0 & 0 & 0 & 0 & 0 & 1 \\ 0 & 0 & 0 & 0 & 1 & 0 \\ 0 & 1 & 1 & 0 & 0 & 0 \\ 1 & 0 & 0 & 1 & 0 & 0 \end{smallmatrix} \right], \ 
 \left[ \begin{smallmatrix} 0 & 0 & 0 & 0 & 0 & 1 \\ 0 & 0 & 0 & 0 & 1 & 0 \\ 0 & 0 & 0 & 0 & 1 & 0 \\ 0 & 0 & 0 & 0 & 0 & 1 \\ 1 & 0 & 0 & 1 & 0 & 0 \\ 0 & 1 & 1 & 0 & 0 & 0  \end{smallmatrix} \right]} $$
There are $4$ real variables $[1, 1, 3, 3, 1, 2]$, $[1, 4, 4, 4, 3, 5]$, $[1, 5, 5, 5, 3, 4]$, $[3, 4, 5, 4, 3, 4]$; $3$ complex (non-real) variables $[1, 1, 3, 4, 5, 5]$, $[1, 1, 3, 5, 4, 4]$, $[1, 4, 4, 5, 1, 4]$, and then their reflection. So we need to consider $10$ variables. Here are some solutions of the equations.
$$(-1, -\frac{\sqrt{2}}{2}, \frac{\sqrt{2}}{2}, -\frac{\sqrt{2}}{2}, -Ir, r, \frac{1+I}{2}, \frac{I\sqrt{2}}{2r},  \frac{\sqrt{2}}{2r}, \frac{1-I}{2}),$$
together with 8 variations given by (pointwise) multiplying by $(1,s_1,-s_1,1,s_2,1,s_3,-s_2,1,-s_3)$ with $s_i \in \{-1,1\}$. The unitary case corresponds to $|r|^4=1/2$.   

$$ \normalsize{\left[ \begin{smallmatrix}1 & 0 & 0 & 0 & 0 & 0 \\ 0 & 1 & 0 & 0 & 0 & 0 \\ 0 & 0 & 1 & 0 & 0 & 0 \\ 0 & 0 & 0 & 1 & 0 & 0 \\ 0 & 0 & 0 & 0 & 1 & 0 \\ 0 & 0 & 0 & 0 & 0 & 1 \end{smallmatrix} \right], \ 
 \left[ \begin{smallmatrix} 0 & 1 & 0 & 0 & 0 & 0 \\ 0 & 0 & 0 & 1 & 0 & 0 \\ 1 & 0 & 0 & 0 & 0 & 0 \\ 0 & 0 & 1 & 0 & 0 & 0 \\ 0 & 0 & 0 & 0 & 0 & 1 \\ 0 & 0 & 0 & 0 & 1 & 0 \end{smallmatrix} \right], \ 
 \left[ \begin{smallmatrix} 0 & 0 & 1 & 0 & 0 & 0 \\ 1 & 0 & 0 & 0 & 0 & 0 \\ 0 & 0 & 0 & 1 & 0 & 0 \\ 0 & 1 & 0 & 0 & 0 & 0 \\ 0 & 0 & 0 & 0 & 0 & 1 \\ 0 & 0 & 0 & 0 & 1 & 0 \end{smallmatrix} \right], \ 
 \left[ \begin{smallmatrix} 0 & 0 & 0 & 1 & 0 & 0 \\ 0 & 0 & 1 & 0 & 0 & 0 \\ 0 & 1 & 0 & 0 & 0 & 0 \\ 1 & 0 & 0 & 0 & 0 & 0 \\ 0 & 0 & 0 & 0 & 1 & 0 \\ 0 & 0 & 0 & 0 & 0 & 1 \end{smallmatrix} \right], \ 
 \left[ \begin{smallmatrix} 0 & 0 & 0 & 0 & 1 & 0 \\ 0 & 0 & 0 & 0 & 0 & 1 \\ 0 & 0 & 0 & 0 & 0 & 1 \\ 0 & 0 & 0 & 0 & 1 & 0 \\ 1 & 0 & 0 & 1 & 0 & 0 \\ 0 & 1 & 1 & 0 & 0 & 0 \end{smallmatrix} \right], \ 
 \left[ \begin{smallmatrix} 0 & 0 & 0 & 0 & 0 & 1 \\ 0 & 0 & 0 & 0 & 1 & 0 \\ 0 & 0 & 0 & 0 & 1 & 0 \\ 0 & 0 & 0 & 0 & 0 & 1 \\ 0 & 1 & 1 & 0 & 0 & 0 \\ 1 & 0 & 0 & 1 & 0 & 0  \end{smallmatrix} \right]} $$
There are $5$ real variables $[1, 1, 3, 3, 1, 2]$, $[1, 4, 5, 4, 3, 5]$, $[1, 5, 4, 5, 3, 4]$, $[3, 4, 4, 4, 3, 4]$, $[3, 5, 5, 5, 3, 5]$; $4$ complex (non-real) variables $[1, 1, 3, 4, 5, 5]$, $[1, 1, 3, 5, 4, 4]$, $[1, 4, 5, 5, 1, 4]$, $[1, 4, 5, 5, 2, 4]$, and then their reflection. So we need to consider $13$ variables. Here are some solutions of the equations:

$$(-1,-\frac{\sqrt{2}}{2}, \frac{\sqrt{2}}{2}, -\frac{\sqrt{2}}{2}, -\frac{\sqrt{2}}{2}, \frac{\sqrt{2}}{2r_1}, \frac{I\sqrt{2}}{2r_1}, \frac{1}{2r_2}, \frac{I}{2r_2}, r_1, -Ir_1, r_2, -Ir_2),$$
together with $4$ variations given by (pointwise) multiplying by $(1,s_1,-s_1,1,1,1,s_2,1,s_2,1,-s_2,1,s_2)$ with $s_i \in \{-1,1\}$. The unitary case corresponds to $|r_1|^4 = |r_2|^2 = 1/2$.

$$ \normalsize{\left[ \begin{smallmatrix}1 & 0 & 0 & 0 & 0 & 0 \\ 0 & 1 & 0 & 0 & 0 & 0 \\ 0 & 0 & 1 & 0 & 0 & 0 \\ 0 & 0 & 0 & 1 & 0 & 0 \\ 0 & 0 & 0 & 0 & 1 & 0 \\ 0 & 0 & 0 & 0 & 0 & 1 \end{smallmatrix} \right], \ 
 \left[ \begin{smallmatrix} 0 & 1 & 0 & 0 & 0 & 0 \\ 1 & 0 & 0 & 0 & 0 & 0 \\ 0 & 0 & 0 & 1 & 0 & 0 \\ 0 & 0 & 1 & 0 & 0 & 0 \\ 0 & 0 & 0 & 0 & 1 & 0 \\ 0 & 0 & 0 & 0 & 0 & 1 \end{smallmatrix} \right], \ 
 \left[ \begin{smallmatrix} 0 & 0 & 1 & 0 & 0 & 0 \\ 0 & 0 & 0 & 1 & 0 & 0 \\ 1 & 0 & 0 & 0 & 0 & 0 \\ 0 & 1 & 0 & 0 & 0 & 0 \\ 0 & 0 & 0 & 0 & 0 & 1 \\ 0 & 0 & 0 & 0 & 1 & 0 \end{smallmatrix} \right], \ 
 \left[ \begin{smallmatrix} 0 & 0 & 0 & 1 & 0 & 0 \\ 0 & 0 & 1 & 0 & 0 & 0 \\ 0 & 1 & 0 & 0 & 0 & 0 \\ 1 & 0 & 0 & 0 & 0 & 0 \\ 0 & 0 & 0 & 0 & 0 & 1 \\ 0 & 0 & 0 & 0 & 1 & 0 \end{smallmatrix} \right], \ 
 \left[ \begin{smallmatrix} 0 & 0 & 0 & 0 & 1 & 0 \\ 0 & 0 & 0 & 0 & 1 & 0 \\ 0 & 0 & 0 & 0 & 0 & 1 \\ 0 & 0 & 0 & 0 & 0 & 1 \\ 0 & 0 & 1 & 1 & 0 & 0 \\ 1 & 1 & 0 & 0 & 0 & 0 \end{smallmatrix} \right], \ 
 \left[ \begin{smallmatrix} 0 & 0 & 0 & 0 & 0 & 1 \\ 0 & 0 & 0 & 0 & 0 & 1 \\ 0 & 0 & 0 & 0 & 1 & 0 \\ 0 & 0 & 0 & 0 & 1 & 0 \\ 1 & 1 & 0 & 0 & 0 & 0 \\ 0 & 0 & 1 & 1 & 0 & 0  \end{smallmatrix} \right]} $$
There are $5$ real variables  $[1, 4, 5, 4, 1, 4]$, $[1, 4, 5, 5, 2, 5]$, $[1, 4, 5, 5, 3, 5]$, $[2, 4, 4, 5, 2, 4]$, $[3, 4, 4, 5, 3, 4]$; $4$ complex (non-real) variables $[1, 2, 3, 3, 1, 2]$, $[1, 2, 3, 4, 5, 4]$, $[1, 2, 3, 5, 4, 5]$, $[2, 4, 4, 5, 3, 4]$, and then their reflection. So we need to consider $13$ variables. Here are some solutions of the equations:

$$(-\frac{\sqrt{2}}{2}, -\frac{\sqrt{2}}{2}, \frac{\sqrt{2}}{2}, -\frac{\sqrt{2}}{2}, \frac{\sqrt{2}}{2}, \frac{1}{4r_1^4 r_2^2} , \frac{1}{2r_1r_2}, \frac{\sqrt{2}}{2r_1}, \frac{1}{2r_2}, 4r_1^4r_2^2, \sqrt{2} r_1 r_2, r_1, r_2),$$
together with $4$ variations given by (pointwise) multiplying by
$(1, s_1, -s_1, s_2, -s_2, 1, -s_2, 1, 1, 1, -s_2, 1, 1)$ with $s_i \in \{-1,1\}$. The unitary case corresponds to $|r_1|^4 = |r_2|^2 = 1/2$.

\subsection{Rank 5, FPdim 24 and type $[1,1,2,3,3]$} \label{sub:r5dim24}  
Consider the fusion matrices: 

$$ \normalsize{\left[ \begin{smallmatrix}1 & 0 & 0 & 0 & 0 \\ 0 & 1 & 0 & 0 & 0 \\ 0 & 0 & 1 & 0 & 0 \\ 0 & 0 & 0 & 1 & 0 \\ 0 & 0 & 0 & 0 & 1 \end{smallmatrix} \right], \ 
 \left[ \begin{smallmatrix} 0 & 1 & 0 & 0 & 0 \\ 1 & 0 & 0 & 0 & 0 \\ 0 & 0 & 1 & 0 & 0 \\ 0 & 0 & 0 & 0 & 1 \\ 0 & 0 & 0 & 1 & 0 \end{smallmatrix} \right], \ 
 \left[ \begin{smallmatrix} 0 & 0 & 1 & 0 & 0 \\ 0 & 0 & 1 & 0 & 0 \\ 1 & 1 & 1 & 0 & 0 \\ 0 & 0 & 0 & 1 & 1 \\ 0 & 0 & 0 & 1 & 1 \end{smallmatrix} \right], \ 
 \left[ \begin{smallmatrix} 0 & 0 & 0 & 1 & 0 \\ 0 & 0 & 0 & 0 & 1 \\ 0 & 0 & 0 & 1 & 1 \\ 0 & 1 & 1 & 1 & 1 \\ 1 & 0 & 1 & 1 & 1 \end{smallmatrix} \right], \ 
 \left[ \begin{smallmatrix} 0 & 0 & 0 & 0 & 1 \\ 0 & 0 & 0 & 1 & 0 \\ 0 & 0 & 0 & 1 & 1 \\ 1 & 0 & 1 & 1 & 1 \\ 0 & 1 & 1 & 1 & 1  \end{smallmatrix} \right]} $$
 
There are $16$ real variables $[1, 2, 2, 2, 1, 2]$, $[1, 2, 2, 2, 2, 2]$, $[1, 3, 3, 3, 2, 4]$, $[1, 3, 3, 4, 1, 3]$, $[2, 2, 2, 2, 2, 2]$, $[2, 2, 2, 3, 3, 3]$, $[2, 2, 2, 3, 3, 4]$, $[2, 2, 2, 3, 4, 4]$, $[2, 2, 2, 4, 4, 4]$, $[2, 3, 3, 3, 2, 4]$, $[2, 3, 3, 4, 2, 3]$, $[2, 3, 4, 3, 2, 3]$, $[2, 3, 4, 3, 3, 3]$, $[2, 3, 4, 4, 3, 4]$, $[2, 3, 4, 4, 4, 4]$, $[2, 4, 3, 4, 3, 4]$; $16$ complex variables $[1, 2, 2, 3, 3, 4]$, $[1, 2, 2, 3, 4, 4]$, $[1, 3, 3, 3, 3, 4]$, $[1, 3, 3, 4, 2, 3]$, $[1, 3, 3, 4, 3, 3]$, $[2, 3, 3, 3, 2, 3]$, $[2, 3, 3, 3, 3, 3]$, $[2, 3, 3, 3, 3, 4]$, $[2, 3, 3, 3, 4, 3]$, $[2, 3, 3, 4, 3, 3]$, $[2, 3, 3, 4, 3, 4]$, $[2, 3, 3, 4, 4, 4]$, $[3, 3, 3, 3, 3, 3]$, $[3, 3, 3, 3, 3, 4]$, $[3, 3, 3, 4, 4, 4]$, $[3, 3, 4, 3, 3, 3]$, and then their reflection. So we need to consider $48$ variables. The PE has $1053$ equations and Krull dimension three.
\begin{remark} \label{rk:krull} In general, the dimension of the affine variety defined by an ideal $I$ in a polynomial ring $R$ is the Krull dimension of $R/I$. So here, it is the dimension $d$ of the variety of solutions of the PE. But SageMath needs dimension zero to provide explicit solutions. By fixing $d$ variables appropriately, we get a non-empty subvariety of dimension zero.
\end{remark}
%
So, by fixing three variables appropriately, we got some unitary solutions: 
\begin{itemize}
\item for the $16$ real variables (with $\epsilon \in \{ -1,1 \}$):  
$$ (\frac{1}{2}, -\frac{1}{2}, -\frac{1}{3}, \frac{1}{3}, 0, -\epsilon\frac{\sqrt{3}}{6}, \epsilon\frac{\sqrt{3}}{6}, -\epsilon\frac{\sqrt{3}}{6}, \epsilon\frac{\sqrt{3}}{6}, -\frac{1}{3}, \frac{1}{3}, \frac{1}{3}, -\frac{1}{6}, \frac{1}{6}, \frac{1}{6}, -\frac{1}{6}),$$ 
\item for the $16$ complex variables: $$((1-I)\frac{\sqrt{3}}{6}, (1-I)\frac{\sqrt{3}}{6}, -\frac{1}{3}, -\frac{I}{3}, \frac{1}{3}, 0, -\frac{\sqrt{3}}{6}, \frac{1}{6}, I\frac{\sqrt{3}}{6}, -\frac{I}{6}, -I\frac{\sqrt{3}}{6}, \frac{\sqrt{3}}{6}, -\frac{1}{6}, -\frac{I}{6}, \frac{1}{6}, -\frac{I}{6}), $$
\end{itemize}
and then their complex-conjugate (for the reflections).

\subsection{Rank 6, FPdim 20 and type $[1, 1, 2, 2, \sqrt{5}, \sqrt{5}]$}  \label{r6dim20}
Consider the fusion matrices: 
$$ \normalsize{\left[ \begin{smallmatrix}1 & 0 & 0 & 0 & 0 & 0 \\ 0 & 1 & 0 & 0 & 0 & 0 \\ 0 & 0 & 1 & 0 & 0 & 0 \\ 0 & 0 & 0 & 1 & 0 & 0 \\ 0 & 0 & 0 & 0 & 1 & 0 \\ 0 & 0 & 0 & 0 & 0 & 1 \end{smallmatrix} \right], \ 
 \left[ \begin{smallmatrix} 0 & 1 & 0 & 0 & 0 & 0 \\ 1 & 0 & 0 & 0 & 0 & 0 \\ 0 & 0 & 1 & 0 & 0 & 0 \\ 0 & 0 & 0 & 1 & 0 & 0 \\ 0 & 0 & 0 & 0 & 0 & 1 \\ 0 & 0 & 0 & 0 & 1 & 0 \end{smallmatrix} \right], \ 
 \left[ \begin{smallmatrix} 0 & 0 & 1 & 0 & 0 & 0 \\ 0 & 0 & 1 & 0 & 0 & 0 \\ 1 & 1 & 0 & 1 & 0 & 0 \\ 0 & 0 & 1 & 1 & 0 & 0 \\ 0 & 0 & 0 & 0 & 1 & 1 \\ 0 & 0 & 0 & 0 & 1 & 1 \end{smallmatrix} \right], \ 
 \left[ \begin{smallmatrix} 0 & 0 & 0 & 1 & 0 & 0 \\ 0 & 0 & 0 & 1 & 0 & 0 \\ 0 & 0 & 1 & 1 & 0 & 0 \\ 1 & 1 & 1 & 0 & 0 & 0 \\ 0 & 0 & 0 & 0 & 1 & 1 \\ 0 & 0 & 0 & 0 & 1 & 1 \end{smallmatrix} \right], \ 
 \left[ \begin{smallmatrix} 0 & 0 & 0 & 0 & 1 & 0 \\ 0 & 0 & 0 & 0 & 0 & 1 \\ 0 & 0 & 0 & 0 & 1 & 1 \\ 0 & 0 & 0 & 0 & 1 & 1 \\ 0 & 1 & 1 & 1 & 0 & 0 \\ 1 & 0 & 1 & 1 & 0 & 0 \end{smallmatrix} \right], \ 
 \left[ \begin{smallmatrix} 0 & 0 & 0 & 0 & 0 & 1 \\ 0 & 0 & 0 & 0 & 1 & 0 \\ 0 & 0 & 0 & 0 & 1 & 1 \\ 0 & 0 & 0 & 0 & 1 & 1 \\ 1 & 0 & 1 & 1 & 0 & 0 \\ 0 & 1 & 1 & 1 & 0 & 0  \end{smallmatrix} \right]} $$

There are $29$ real variables $[1, 2, 2, 2, 1, 2]$, $[1, 2, 2, 2, 3, 2]$, $[1, 2, 2, 3, 2, 3]$, $[1, 2, 2, 3, 3, 3]$, $[1, 3, 3, 3, 1, 3]$, $[1, 3, 3, 3, 2, 3]$, $[1, 4, 4, 4, 2, 5]$, $[1, 4, 4, 4, 3, 5]$, $[1, 4, 4, 5, 1, 4]$, $[2, 2, 3, 3, 2, 2]$, $[2, 2, 3, 4, 4, 4]$, $[2, 2, 3, 4, 5, 5]$, $[2, 2, 3, 5, 4, 4]$, $[2, 2, 3, 5, 5, 5]$, $[2, 3, 3, 3, 2, 3]$, $[2, 3, 3, 4, 4, 4]$, $[2, 3, 3, 4, 5, 4]$, $[2, 3, 3, 5, 4, 5]$, $[2, 3, 3, 5, 5, 5]$, $[2, 4, 4, 4, 2, 5]$, $[2, 4, 4, 4, 3, 5]$, $[2, 4, 4, 5, 2, 4]$, $[2, 4, 5, 4, 2, 4]$, $[2, 4, 5, 4, 3, 4]$, $[2, 4, 5, 5, 3, 5]$, $[2, 5, 4, 5, 3, 5]$, $[3, 4, 4, 4, 3, 5]$, $[3, 4, 4, 5, 3, 4]$, $[3, 4, 5, 4, 3, 4]$; $16$ complex variables $[1, 2, 2, 4, 4, 5]$, $[1, 2, 2, 4, 5, 5]$, $[1, 3, 3, 4, 4, 5]$, $[1, 3, 3, 4, 5, 5]$, $[1, 4, 4, 5, 2, 4]$, $[1, 4, 4, 5, 3, 4]$, $[2, 2, 3, 3, 2, 3]$, $[2, 2, 3, 4, 4, 5]$, $[2, 2, 3, 5, 4, 5]$, $[2, 3, 3, 4, 4, 5]$, $[2, 3, 3, 4, 5, 5]$, $[2, 4, 4, 4, 2, 4]$, $[2, 4, 4, 4, 3, 4]$, $[2, 4, 4, 5, 3, 4]$, $[2, 4, 4, 5, 3, 5]$, $[3, 4, 4, 4, 3, 4]$, and then their reflection. So we need to consider $61$ variables. The polynomial ring modulo the ideal generated by the ($1231$) equations
is of Krull dimension two (see Remark \ref{rk:krull}). So, by fixing two variables appropriately, we got some unitary solutions: 
\begin{itemize}
\item for the $29$ real variables: $$(\frac{1}{2}, -\frac{1}{2}, -\epsilon_1 \frac{1}{2}, \epsilon_1 \frac{1}{2}, \frac{1}{2}, -\frac{1}{2}, -a, -a, a, 0,  \epsilon_2 b,  \epsilon_2 b, -\epsilon_2 b, -\epsilon_2 b, 0,  \epsilon_2 b, -\epsilon_2 b,  \epsilon_2 b, -\epsilon_2 b, c, -d, -c, -c, d, -d, d, c, -c, -c),$$ 
\item for the $16$ complex variables: $$(e, e, -\epsilon_1 \epsilon_2 e, -\epsilon_1 \epsilon_2 e, -a, -a, \frac{1}{2}, b, -b, -b, b,  \epsilon_3 fI, -\epsilon_2 \epsilon_3 gI, d,  \epsilon_2 \epsilon_3 gI, -\epsilon_3 fI),$$
\end{itemize} 
and then their complex-conjugate (for the reflections), with $\epsilon_i \in \{-1,1 \}$ and $$(a,b,c,d,e,f,g) = (\frac{\sqrt{5}}{5}, 80^{-1/4},\frac{5+\sqrt{5}}{20},\frac{5-\sqrt{5}}{20}, 20^{-1/4}, \sqrt{\frac{5-\sqrt{5}}{40}}, \sqrt{\frac{5+\sqrt{5}}{40}}).$$

\subsection{Models by zesting construction} \label{sub:zest}

E.C. Rowell pointed out to us a new construction called \emph{zesting} \cite{zest}, providing models for some of the new Grothendieck rings mentioned in this section. The metaplectic categories are those Grothendieck equivalent to $\SO(N)_2$ with $N=2n+1 \ge 1$. At fixed $n$, it is of multiplicity one, rank $n+4$, type $[[1,2],[2,n],[\sqrt{N},2]]$. Let $z$ be the non-trivial object of $\FPdim \ 1$, $(y_i)$ those of $\FPdim \ 2$, and $x_1, x_2$ those of $\FPdim \ \sqrt{N}$. As mentioned in \cite[\S 3]{ACMRW16}, the (commutative) fusion rules are the following: 
\begin{itemize}
\item[(1)] $zy_i=y_i$, $zx_1=x_2$, $zx_2=x_1$, $z^2=1$,
\item[(2)] $x^2_i=1+\sum_i y_i$,
\item[(3)] $x_1x_2=z+\sum_i y_i$,
\item[(4)] $y_iy_j=y_{\min(i+j,N-i-j)}+y_{|i-j|}$, for $i \neq j$, $y_i^2=1+z+y_{\min(2i,N-2i)}$.
\end{itemize} 

\begin{theorem}[Twisted metaplectic categories] \label{thm:twist}
The exchange of $1$ and $z$ in (2) and (3) above produces a new family of complex Grothendieck rings.
\end{theorem}
\begin{proof}
The result follows from the new zesting construction \cite[Proposition 6.3]{zest} using the braiding of $\SO(N)_2$ and its $C_2$-grading, which twists the associativity by a $3$-cocycle. E.C. Rowell provided more details about that in \cite{RowellMO}.
\end{proof}

The cases $n=1, 2$ correspond to the new (unitary) Grothendieck rings of \S\ref{r5dim12} and \S\ref{r6dim20} respectively (and $n=0$ to $\VVec(C_4)$). Note that in the same way, the ones of \S \ref{r6dim8} are zestings of $\VVec(C_2) \otimes \SU(2)_2$. Finally, the one of \S \ref{sub:r5dim24} cannot be a zesting, because there is no grading. But observe that we can produce two new families of fusion rings, the first one by adding $n(x_1+x_2)$ to the right hand-side of (2) and (3) above (which recovers $\Rep(S_4)$ when $n=1$), and the second one by twisting the first one as in Theorem \ref{thm:twist}, which (for $n=1$) would provide a model for the new (unitary) Grothendieck ring of \S \ref{sub:r5dim24}.

\begin{question} \label{M+}
Are the fusion rings of these two new families, (unitary) Grothendieck rings?
\end{question}

%

\section{Observations and questions} \label{sec:O&Q}

This classification leads to many observations and questions, grouped in this section.

\subsection{All criteria passed and categorification}
Every fusion ring ruled out here was directly excluded by some of the criteria mentioned in Section \ref{sec:crit} (without considering that of Subsection \ref{sub:mod}).

\begin{question}
Is there a fusion ring of multiplicity one which passes all the criteria of Section \ref{sec:crit} without being categorifiable?
\end{question}

Note that without the multiplicity one assumption (but pivotal or characteristic zero), the above question already admits a negative answer in \cite{LPR1} with the fusion ring denoted $\mathcal{F}_{210}$, of multiplicity $2$, rank $7$, $\FPdim \ 210$, type $[1,5,5,5,6,7,7]$ and fusion matrices:
$$
\left[\begin{smallmatrix}
1 & 0 & 0 & 0& 0& 0& 0 \\
0 & 1 & 0 & 0& 0& 0& 0 \\
0 & 0 & 1 & 0& 0& 0& 0 \\
0 & 0 & 0 & 1& 0& 0& 0 \\
0 & 0 & 0 & 0& 1& 0& 0 \\
0 & 0 & 0 & 0& 0& 1& 0 \\
0 & 0 & 0 & 0& 0& 0& 1
\end{smallmatrix}\right] , \
\left[\begin{smallmatrix}
0 & 1 & 0 & 0& 0& 0& 0 \\
1 & 1 & 0 & 1& 0& 1& 1 \\
0 & 0 & 1 & 0& 1& 1& 1 \\
0 & 1 & 0 & 0& 1& 1& 1 \\
0 & 0 & 1 & 1& 1& 1& 1 \\
0 & 1 & 1 & 1& 1& 1& 1 \\
0 & 1 & 1 & 1& 1& 1& 1
\end{smallmatrix}\right] , \
\left[\begin{smallmatrix}
0 & 0 & 1 & 0& 0& 0& 0 \\
0 & 0 & 1 & 0& 1& 1& 1 \\
1 & 1 & 1 & 0& 0& 1& 1 \\
0 & 0 & 0 & 1& 1& 1& 1 \\
0 & 1 & 0 & 1& 1& 1& 1 \\
0 & 1 & 1 & 1& 1& 1& 1 \\
0 & 1 & 1 & 1& 1& 1& 1
\end{smallmatrix}\right] , \
\left[\begin{smallmatrix}
0 & 0 & 0 & 1& 0& 0& 0 \\
0 & 1 & 0 & 0& 1& 1& 1 \\
0 & 0 & 0 & 1& 1& 1& 1 \\
1 & 0 & 1 & 1& 0& 1& 1 \\
0 & 1 & 1 & 0& 1& 1& 1 \\
0 & 1 & 1 & 1& 1& 1& 1 \\
0 & 1 & 1 & 1& 1& 1& 1
\end{smallmatrix}\right] , \
\left[\begin{smallmatrix}
0 & 0 & 0 & 0& 1& 0& 0 \\
0 & 0 & 1 & 1& 1& 1& 1 \\
0 & 1 & 0 & 1& 1& 1& 1 \\
0 & 1 & 1 & 0& 1& 1& 1 \\
1 & 1 & 1 & 1& 1& 1& 1 \\
0 & 1 & 1 & 1& 1& 2& 1 \\
0 & 1 & 1 & 1& 1& 1& 2
\end{smallmatrix}\right] , \
\left[\begin{smallmatrix}
0 & 0 & 0 & 0& 0& 1& 0 \\
0 & 1 & 1 & 1& 1& 1& 1 \\
0 & 1 & 1 & 1& 1& 1& 1 \\
0 & 1 & 1 & 1& 1& 1& 1 \\
0 & 1 & 1 & 1& 1& 2& 1 \\
1 & 1 & 1 & 1& 2& 1& 2 \\
0 & 1 & 1 & 1& 1& 2& 2
\end{smallmatrix}\right] , \
\left[\begin{smallmatrix}
0 & 0 & 0 & 0& 0& 0& 1 \\
0 & 1 & 1 & 1& 1& 1& 1 \\
0 & 1 & 1 & 1& 1& 1& 1 \\
0 & 1 & 1 & 1& 1& 1& 1 \\
0 & 1 & 1 & 1& 1& 1& 2 \\
0 & 1 & 1 & 1& 1& 2& 2 \\
1 & 1 & 1 & 1& 2& 2& 1
\end{smallmatrix}\right] $$

\noindent It corresponds to the case $q=6$ of the interpolated family of fusion rings of Lie type in \cite{LPR4}, all of them being of multiplicity $2$ or $3$, and satifying all the criteria of Section \ref{sec:crit} (the existence of a categorification is open for all non prime-power $q \neq 6$).

There are fusion rings of multiplicity one and rank $7$ which pass all these criteria but not the one of Conjecture \ref{conj:isa}; one of them is of $\FPdim \ 20+4\sqrt{5}$, type $[1, 1, 1, 1, 2, 1+\sqrt{5}, 1+\sqrt{5}]$ and fusion matrices:
$$ \normalsize{\left[ \begin{smallmatrix}1 & 0 & 0 & 0 & 0 & 0 & 0 \\ 0 & 1 & 0 & 0 & 0 & 0 & 0 \\ 0 & 0 & 1 & 0 & 0 & 0 & 0 \\ 0 & 0 & 0 & 1 & 0 & 0 & 0 \\ 0 & 0 & 0 & 0 & 1 & 0 & 0 \\ 0 & 0 & 0 & 0 & 0 & 1 & 0 \\ 0 & 0 & 0 & 0 & 0 & 0 & 1 \end{smallmatrix} \right]  , \  \left[ \begin{smallmatrix} 0 & 1 & 0 & 0 & 0 & 0 & 0 \\ 1 & 0 & 0 & 0 & 0 & 0 & 0 \\ 0 & 0 & 0 & 1 & 0 & 0 & 0 \\ 0 & 0 & 1 & 0 & 0 & 0 & 0 \\ 0 & 0 & 0 & 0 & 1 & 0 & 0 \\ 0 & 0 & 0 & 0 & 0 & 1 & 0 \\ 0 & 0 & 0 & 0 & 0 & 0 & 1 \end{smallmatrix} \right]  , \  \left[ \begin{smallmatrix} 0 & 0 & 1 & 0 & 0 & 0 & 0 \\ 0 & 0 & 0 & 1 & 0 & 0 & 0 \\ 1 & 0 & 0 & 0 & 0 & 0 & 0 \\ 0 & 1 & 0 & 0 & 0 & 0 & 0 \\ 0 & 0 & 0 & 0 & 1 & 0 & 0 \\ 0 & 0 & 0 & 0 & 0 & 0 & 1 \\ 0 & 0 & 0 & 0 & 0 & 1 & 0 \end{smallmatrix} \right]  , \  \left[ \begin{smallmatrix} 0 & 0 & 0 & 1 & 0 & 0 & 0 \\ 0 & 0 & 1 & 0 & 0 & 0 & 0 \\ 0 & 1 & 0 & 0 & 0 & 0 & 0 \\ 1 & 0 & 0 & 0 & 0 & 0 & 0 \\ 0 & 0 & 0 & 0 & 1 & 0 & 0 \\ 0 & 0 & 0 & 0 & 0 & 0 & 1 \\ 0 & 0 & 0 & 0 & 0 & 1 & 0 \end{smallmatrix} \right]  , \  \left[ \begin{smallmatrix} 0 & 0 & 0 & 0 & 1 & 0 & 0 \\ 0 & 0 & 0 & 0 & 1 & 0 & 0 \\ 0 & 0 & 0 & 0 & 1 & 0 & 0 \\ 0 & 0 & 0 & 0 & 1 & 0 & 0 \\ 1 & 1 & 1 & 1 & 0 & 0 & 0 \\ 0 & 0 & 0 & 0 & 0 & 1 & 1 \\ 0 & 0 & 0 & 0 & 0 & 1 & 1 \end{smallmatrix} \right]  , \  \left[ \begin{smallmatrix} 0 & 0 & 0 & 0 & 0 & 1 & 0 \\ 0 & 0 & 0 & 0 & 0 & 1 & 0 \\ 0 & 0 & 0 & 0 & 0 & 0 & 1 \\ 0 & 0 & 0 & 0 & 0 & 0 & 1 \\ 0 & 0 & 0 & 0 & 0 & 1 & 1 \\ 1 & 1 & 0 & 0 & 1 & 1 & 1 \\ 0 & 0 & 1 & 1 & 1 & 1 & 1 \end{smallmatrix} \right]  , \  \left[ \begin{smallmatrix} 0 & 0 & 0 & 0 & 0 & 0 & 1 \\ 0 & 0 & 0 & 0 & 0 & 0 & 1 \\ 0 & 0 & 0 & 0 & 0 & 1 & 0 \\ 0 & 0 & 0 & 0 & 0 & 1 & 0 \\ 0 & 0 & 0 & 0 & 0 & 1 & 1 \\ 0 & 0 & 1 & 1 & 1 & 1 & 1 \\ 1 & 1 & 0 & 0 & 1 & 1 & 1  \end{smallmatrix} \right]} $$

\subsection{Categorification in positive characteristic} \label{sub:pos}

Among the $25$ fusion rings ruled out from complex categorification, $16$ are also ruled out from any categorification (over any field) by Theorem \ref{thm:zero} or \ref{thm:lagrange}. So, among the remaining $9$ ones (\textnumero 14, 15 at rank 5 and \textnumero 16-19, 22, 30, 39 at rank 6): 
\begin{question}
Which ones admit a categorification over a field of positive characteristic?
\end{question}
\begin{question}
In general, is there a fusion ring without categorification in characteristic zero but positive?
\end{question}

\subsection{Unitary categorification}

\begin{corollary}
A complex Grothendieck ring of multiplicity one up to rank six is unitary.
\end{corollary}

\begin{question}
Is there a complex Grothendieck ring of multiplicity one which is not unitary?
\end{question}

Note that without the multiplicity one assumption, the above question already admits a negative answer in \cite{sch20} providing a complex non pseudo-unitary Grothendieck ring of multiplicity $2$, rank $6$, $\FPdim \ 9(3+3a_1-a_4) \simeq 74.6177$ (with $a_k = 2 \cos(k \pi/9)$), type $[1,1+a_1,1+a_1,1+a_1,1+2a_1-a_4,2+2a_1-a_4]$ and fusion matrices:
$$ \left[\begin{smallmatrix}1 & 0 & 0 & 0 & 0 & 0  \\ 0 & 1 & 0 & 0 & 0 & 0  \\ 0 & 0 & 1 & 0 & 0 & 0  \\ 0 & 0 & 0 & 1 & 0 & 0  \\ 0 & 0 & 0 & 0 & 1 & 0  \\ 0 & 0 & 0 & 0 & 0 & 1\end{smallmatrix}\right]  , \ \left[\begin{smallmatrix}0 & 1 & 0 & 0 & 0 & 0  \\ 1 & 0 & 1 & 0 & 1 & 0  \\ 0 & 1 & 0 & 0 & 0 & 1  \\ 0 & 0 & 0 & 1 & 0 & 1  \\ 0 & 1 & 0 & 0 & 1 & 1  \\ 0 & 0 & 1 & 1 & 1 & 1\end{smallmatrix}\right]  , \ \left[\begin{smallmatrix}0 & 0 & 1 & 0 & 0 & 0  \\ 0 & 1 & 0 & 0 & 0 & 1  \\ 1 & 0 & 0 & 1 & 1 & 0  \\ 0 & 0 & 1 & 0 & 0 & 1  \\ 0 & 0 & 1 & 0 & 1 & 1  \\ 0 & 1 & 0 & 1 & 1 & 1\end{smallmatrix}\right]  , \ \left[\begin{smallmatrix}0 & 0 & 0 & 1 & 0 & 0  \\ 0 & 0 & 0 & 1 & 0 & 1  \\ 0 & 0 & 1 & 0 & 0 & 1  \\ 1 & 1 & 0 & 0 & 1 & 0  \\ 0 & 0 & 0 & 1 & 1 & 1  \\ 0 & 1 & 1 & 0 & 1 & 1\end{smallmatrix}\right]  , \ \left[\begin{smallmatrix}0 & 0 & 0 & 0 & 1 & 0  \\ 0 & 1 & 0 & 0 & 1 & 1  \\ 0 & 0 & 1 & 0 & 1 & 1  \\ 0 & 0 & 0 & 1 & 1 & 1  \\ 1 & 1 & 1 & 1 & 1 & 1  \\ 0 & 1 & 1 & 1 & 1 & 2\end{smallmatrix}\right]  , \ \left[\begin{smallmatrix}0 & 0 & 0 & 0 & 0 & 1  \\ 0 & 0 & 1 & 1 & 1 & 1  \\ 0 & 1 & 0 & 1 & 1 & 1  \\ 0 & 1 & 1 & 0 & 1 & 1  \\ 0 & 1 & 1 & 1 & 1 & 2  \\ 1 & 1 & 1 & 1 & 2 & 2 
\end{smallmatrix}\right] $$

More generally (a particular case of \cite[Question 4.8.3]{EGNO15}):
\begin{question}
Does every complex fusion category admits a pivotal structure? A spherical structure?
\end{question}

Recall that in the list ``unitary, pseudo-unitary, $\varphi$-pseudo-unitary, spherical, pivotal'', one implies its successor.

\subsection{Non-cyclotomic fusion categories}
On page 591 of \cite{ENO05}, it is asked whether any (multi-)fusion category is defined over a cyclotomic field. This question was answer negatively in \cite{MoSn}. Now \S \ref{r5dim12} and \S \ref{r6dim20} mention non-cyclotomic solutions of their PE: they have F-symbols equal to $ \frac{\pm 3^{3/4}}{6}$, $20^{-1/4}$ or $80^{-1/4}$. By Kronecker-Weber theorem and the following SageMath computation, all these numbers are non-cyclotomic.
\begin{verbatim}
sage: x=(3^(3/4))/6 # or 20^(-1/4) or 80^(-1/4)
sage: f=minpoly(x)
sage: K.<a> = f.splitting_field()
sage: K.is_abelian()
False
\end{verbatim}
To be non-cyclotomic, a fusion category must have non-cyclotomic F-symbols for every choice of basis, whereas the solutions mention above come from a single choice. 

\begin{question}
Is the fusion category mentioned in \S \ref{r5dim12} or \S \ref{r6dim20} non-cyclotomic?
\end{question}

\subsection{Integral Grothendieck rings}

\begin{corollary}
A weakly integral fusion ring of multiplicity one up to rank six is always unitarily categorifiable.
\end{corollary}

There are integral fusion rings of multiplicity one without any categorification. Below is an example at rank $7$, FPdim $42$ and type $[1, 1, 2, 3, 3, 3, 3]$ which is non-Czero $(3, 3, 2, 2, 4, 3, 3, 3, 3)$:
$$ \left[ \begin{smallmatrix}1 & 0 & 0 & 0 & 0 & 0 & 0 \\ 0 & 1 & 0 & 0 & 0 & 0 & 0 \\ 0 & 0 & 1 & 0 & 0 & 0 & 0 \\ 0 & 0 & 0 & 1 & 0 & 0 & 0 \\ 0 & 0 & 0 & 0 & 1 & 0 & 0 \\ 0 & 0 & 0 & 0 & 0 & 1 & 0 \\ 0 & 0 & 0 & 0 & 0 & 0 & 1 \end{smallmatrix} \right] , \  \left[ \begin{smallmatrix} 0 & 1 & 0 & 0 & 0 & 0 & 0 \\ 1 & 0 & 0 & 0 & 0 & 0 & 0 \\ 0 & 0 & 1 & 0 & 0 & 0 & 0 \\ 0 & 0 & 0 & 0 & 1 & 0 & 0 \\ 0 & 0 & 0 & 1 & 0 & 0 & 0 \\ 0 & 0 & 0 & 0 & 0 & 0 & 1 \\ 0 & 0 & 0 & 0 & 0 & 1 & 0 \end{smallmatrix} \right] , \  \left[ \begin{smallmatrix} 0 & 0 & 1 & 0 & 0 & 0 & 0 \\ 0 & 0 & 1 & 0 & 0 & 0 & 0 \\ 1 & 1 & 1 & 0 & 0 & 0 & 0 \\ 0 & 0 & 0 & 1 & 1 & 0 & 0 \\ 0 & 0 & 0 & 1 & 1 & 0 & 0 \\ 0 & 0 & 0 & 0 & 0 & 1 & 1 \\ 0 & 0 & 0 & 0 & 0 & 1 & 1 \end{smallmatrix} \right] , \  \left[ \begin{smallmatrix} 0 & 0 & 0 & 1 & 0 & 0 & 0 \\ 0 & 0 & 0 & 0 & 1 & 0 & 0 \\ 0 & 0 & 0 & 1 & 1 & 0 & 0 \\ 1 & 0 & 1 & 1 & 0 & 1 & 0 \\ 0 & 1 & 1 & 0 & 1 & 0 & 1 \\ 0 & 0 & 0 & 1 & 0 & 1 & 1 \\ 0 & 0 & 0 & 0 & 1 & 1 & 1 \end{smallmatrix} \right] , \  \left[ \begin{smallmatrix} 0 & 0 & 0 & 0 & 1 & 0 & 0 \\ 0 & 0 & 0 & 1 & 0 & 0 & 0 \\ 0 & 0 & 0 & 1 & 1 & 0 & 0 \\ 0 & 1 & 1 & 0 & 1 & 0 & 1 \\ 1 & 0 & 1 & 1 & 0 & 1 & 0 \\ 0 & 0 & 0 & 0 & 1 & 1 & 1 \\ 0 & 0 & 0 & 1 & 0 & 1 & 1 \end{smallmatrix} \right] , \  \left[ \begin{smallmatrix} 0 & 0 & 0 & 0 & 0 & 1 & 0 \\ 0 & 0 & 0 & 0 & 0 & 0 & 1 \\ 0 & 0 & 0 & 0 & 0 & 1 & 1 \\ 0 & 0 & 0 & 1 & 0 & 1 & 1 \\ 0 & 0 & 0 & 0 & 1 & 1 & 1 \\ 1 & 0 & 1 & 1 & 1 & 0 & 0 \\ 0 & 1 & 1 & 1 & 1 & 0 & 0 \end{smallmatrix} \right] , \  \left[ \begin{smallmatrix} 0 & 0 & 0 & 0 & 0 & 0 & 1 \\ 0 & 0 & 0 & 0 & 0 & 1 & 0 \\ 0 & 0 & 0 & 0 & 0 & 1 & 1 \\ 0 & 0 & 0 & 0 & 1 & 1 & 1 \\ 0 & 0 & 0 & 1 & 0 & 1 & 1 \\ 0 & 1 & 1 & 1 & 1 & 0 & 0 \\ 1 & 0 & 1 & 1 & 1 & 0 & 0  \end{smallmatrix} \right] $$

\subsection{Simple Grothendieck rings}

A fusion ring is called \emph{simple} if it has no proper non-trivial fusion subring. A Grothendieck ring is called simple if it is so as fusion ring.

\begin{corollary}
A simple complex Grothendieck ring of multiplicity one up to rank six is given by the following: 
\begin{itemize}
\item $\VVec(C_p)$, with $C_p$ the cyclic group of order $p$ prime,
\item $\PSU(2)_{2n+1}$, with $n \ge 0$.
\end{itemize}
\end{corollary}

\begin{question}
Is there a simple complex Grothendieck ring of multiplicity one, not in above families?
\end{question}

\section{Appendix: SageMath code} \label{sec:code}

This section provides the SageMath code for the criteria of Theorems \ref{thm:drinfeld}, \ref{thm:Dnumb} and \ref{thm:cyclo} (the only criteria needed to prove Theorem \ref{thm:main}). They apply in the commutative case only (as needed). Just apply the function \emph{Checking} below to a fusion ring written as a list $M$, for example, the following computation shows that the fusion ring \textnumero $5$ at rank $4$ is non-Drinfeld and non-d-number (as written in Subsection \ref{sub:proof}).

\begin{verbatim}
sage: M=[[[1, 0, 0, 0], [0, 1, 0, 0], [0, 0, 1, 0], [0, 0, 0, 1]],
....:    [[0, 1, 0, 0], [0, 0, 1, 0], [1, 0, 0, 0], [0, 0, 0, 1]],
....:    [[0, 0, 1, 0], [1, 0, 0, 0], [0, 1, 0, 0], [0, 0, 0, 1]],
....:    [[0, 0, 0, 1], [0, 0, 0, 1], [0, 0, 0, 1], [1, 1, 1, 1]]]
sage: Checking(M)
non-Drinfeld
non-d-number
\end{verbatim}

Here is the code of the function Checking:
\begin{verbatim}
UCF=UniversalCyclotomicField()

def Checking(M):
    r=len(M)
    N=zero_matrix(QQ,r)
    for i in range(r):
        Mi=matrix(QQ,M[i])
        for j in range(i):
            Mj=matrix(QQ,M[j])
            if Mi*Mj!=Mj*Mi:
                return 'non-commutative'
        Ti=Mi.transpose()
        Ni=Mi*Ti
        N+=Ni
    f = N.minpoly()
    ff=N.charpoly()
    if not Cyclo(M):
        return 'non-cyclo'    # Extended cyclotomic criterion    
    K.<a> = f.splitting_field()
    n = K.conductor()
    L=ff.roots(CyclotomicField(n))
    LL=[UCF(l[0]) for l in L]
    rL=[l[0].n() for l in L]
    mm=max(rL)
    for ii in range(len(L)):
            if mm==rL[ii]:
                dim=LL[ii]
                break
    c=0
    for x in LL:
        d=0
        for y in LL:
            yy=UCF(x/y)
            if '/' in list(str(yy)):
                d=1
                break
        if d==0:
            c=1
    if c==0:
        print('non-Drinfeld')    # Drinfeld center criterion
    for x in LL:                
        p=list(UCF(x).minpoly())
        n=len(p)-1
        A=p[0]
        d=0
        for i in range(n+1):
            a=p[i]
            j=n-i
            y=UCF((a^n)/(A^j))
            if '/' in list(str(y)):
                d=1
                break
        if d==1:
            print('non-d-number')    # d-number criterion
            break

def Cyclo(M):    
    r=len(M)
    for k in range(len(M)):
        N=matrix(QQ,M[k])
        f = N.minpoly()
        K.<a> = f.splitting_field()
        if not K.is_abelian():        
            return false
    return true
\end{verbatim}

\begin{st}
On behalf of all authors, the corresponding author states that there is no conflict of interest. The datasets generated during and/or analysed during the current study are available from the corresponding author on reasonable request.
\end{st}


\begin{thebibliography}{99}

\bibitem{ACMRW16}
{\sc E.~Ardonne, M.~Cheng, E.~C.~Rowell, Z.~Wang}, {\em Classification of metaplectic modular categories}, J. Algebra 466 (2016), 141--146.

\bibitem{PhDBond}
{\sc P.H.~Bonderson}, {\em Non-Abelian Anyons and Interferometry}, PhD Thesis, California Institute of Technology Pasadena (2007).

\bibitem{BGNPRW16}
{\sc P.~Bruillard, C.~Galindo, S.H.~Ng, J.~Plavnik, E.C.~Rowell, Z.~Wang}, {\em On the classification of weakly integral modular categories}, J. Pure Appl. Algebra, 220-6, pp.~2364--2388 (2016).

\bibitem{DaHaWa}
{\sc O.~Davidovich, T.~Hagge, Z.~Wang}, {\em On Arithmetic Modular Categories}, arXiv:1305.2229 (2013).

\bibitem{zest}
{\sc C.~Delaney, C.~Galindo, J.~Plavnik, E.~C.~Rowell, Q.~Zhang}, {\em Braided zesting and its applications}, Comm. Math. Phys. 386 (2021), no. 1, 1--55.

\bibitem{DGNO}
{\sc V.~Drinfeld, S.~Gelaki, D.~Nikshych, V.~Ostrik}, {\em On braided fusion categories. I.}, Selecta Math. (N.S.) 16 (2010), no. 1, 1--119. 

 \bibitem{EGNO15}
{\sc P.~Etingof, S.~Gelaki, D.~Nikshych, V.~Ostrik}, {\em Tensor
  Categories}, Mathematical Surveys and Monographs Volume 205 (2015).

\bibitem{ENO05}
{\sc P.~Etingof, D.~Nikshych, V.~Ostrik}, {\em On fusion categories}, Annals of Mathematics, 162, pp.~581--642 (2005).

\bibitem{ENO11}
{\sc P.~Etingof, D.~Nikshych, V.~Ostrik}, {\em Weakly group-theoretical and
  solvable fusion categories}, Adv. Math., 226 (2011), pp.~176--205.  

\bibitem{eti20}
{\sc P.~Etingof}, {private communication} (2020).

\bibitem{ENO21}
{\sc P.~Etingof, D.~Nikshych, V.~Ostrik}, {\em On a necessary condition for unitary categorification of fusion rings}, arXiv:2102.13239, (2021).

\bibitem{EG14}
{\sc D.~E. Evans, T.~Gannon, Terry}. {\em Near-group fusion categories and their doubles}, Adv. Math. 255 (2014), 586--640.

\bibitem{GP95}
{\sc D.~Gepner, A.~Kapustin}, {\em On the classification of fusion rings}, Phys. Lett. B 349, 71–75 (1995).

\bibitem{HNW14}
{\sc M.B.~Hastings, C.~Nayak, Z.~Wang},  {\em On metaplectic modular categories and their applications.}, Commun. Math. Phys. 330, No. 1, 45-68 (2014).

\bibitem{Isa}
{\sc I.M.~Isaacs}, {\em Character theory of finite groups}, Corrected reprint of the 1976 original, AMS Chelsea Publishing, xii+310 (2006).

\bibitem{IzuKo}
{\sc M.~Izumi, H.~Kosaki}, {\em Kac algebras arising from composition of subfactors: general theory and classification}, Mem. Amer. Math. Soc. 158, no. 750, 198pp (2002).

\bibitem{lang}
{\sc S.~Lang}, {\em Cyclotomic fields I and II}, Graduate Texts in Mathematics, 121 (1990) {\rm xviii}+433 pp.

\bibitem{LMP15}
{\sc Z.~Liu,  S.~Morrison, D.~Penneys}, {\em 1-supertransitive subfactors with index at most $6\frac{1}{5}$}, Comm. Math. Phys. 334 (2015), no. 2, 889--922.

\bibitem{LPR1}
{\sc Z.~Liu, S.~Palcoux, Y.~Ren}, {\em Triangular prism equations and categorification}, arXiv:2203.06522. 

\bibitem{LPR4}
{\sc Z.~Liu, S.~Palcoux, Y.~Ren}, {\em Interpolated family of non group-like simple integral fusion rings of Lie type}, arXiv:2102.01663.

\bibitem{LPW20}
{\sc Z.~Liu, S.~Palcoux, J.~Wu}, {\em Fusion Bialgebras and Fourier Analysis}, Adv. Math. 390 (2021), 107905.

\bibitem{Lus87}
{\sc G.~Lusztig}, {\em \href{}{Leading coefficients of character values of Hecke algebras}}, Proc. Symp. in Pure Math., 47, pp.~235--262 (1987).

\bibitem{MoSn}
{\sc S.~Morrison, N.~Snyder}, {\em Non-cyclotomic fusion categories.} Trans. Amer. Math. Soc. 364 (2012), no. 9, 4713--4733.

\bibitem{Ost15}
{\sc V.~Ostrik}, {\em Pivotal fusion categories of rank 3}, Mosc. Math. J., 15, pp.~373--396, 405  (2015).

\bibitem{ostrik} 
{\sc V.~Ostrik}, \emph{On formal codegrees of fusion categories}. Math. Res. Lett. 16 (2009), no. 5, 895–901.

\bibitem{FusionAtlas}
{\sc S.~Palcoux}, \url{https://sites.google.com/view/sebastienpalcoux/fusion-rings} (2020).

\bibitem{A348305}
{\sc S.~Palcoux}, \emph{Number of fusion rings of multiplicity one and rank n}, OEIS, \url{http://oeis.org/A348305}.

\bibitem{A352506}
{\sc S.~Palcoux}, \emph{Number of complex Grothendieck rings of multiplicity one and rank n}, OEIS, \url{http://oeis.org/A352506}.

\bibitem{RowellMO} 
{\sc E.C.~Rowell}, \emph{Existence of twisted metaplectic categories}, MathOverflow, \url{https://mathoverflow.net/a/369169/34538}

\bibitem{sch20}{\sc A.~Schopieray}, {\em Non-pseudounitary fusion}, J. Pure Appl. Algebra 226 (2022), no. 5, Paper No. 106927, 19 pp.

\bibitem{AnyonWiki}
{\sc J.~Slingerland, G.~Vercleyen}, {\em AnyonWiki}, \url{http://www.thphys.nuim.ie/AnyonWiki}.

\bibitem{SliVer}
{\sc J.~Slingerland, G.~Vercleyen}, {\em On Low Rank Fusion Rings}, 	arXiv:2205.15637.

\bibitem{sage}
{\sc The Sage Developers}, {\em {S}ageMath, the {S}age {M}athematics {S}oftware {S}ystem} ({V}ersion 9.0), {sagemath.org} (2020) 

\bibitem{TY98}
{\sc D.~Tambara, S.~Yamagami}, {\em Tensor categories with fusion rules of self-duality for finite abelian groups}, J. Algebra 209, 692–707 (1998).

\bibitem{thor}
{\sc J.E.~Thornton}, {\em Generalized near-group categories}, Thesis (Ph.D.)–University of Oregon (2012) 72 pp.

\bibitem{wang}{\sc Z.~Wang}, {\em Topological quantum computation},  CBMS Reg. Conf. Ser. Math. (112) xiii + 115pp, (2010).



\end{thebibliography}
\end{document}